\documentclass[11pt]{article}
\usepackage[english]{babel}
\usepackage{fullpage}
\usepackage[utf8]{inputenc}
\usepackage{amsmath,amssymb,amsthm}
\usepackage{mathrsfs}
\usepackage{graphicx}
\usepackage{tikz}
\usepackage{xcolor}
\usepackage{cite}
\usepackage{tikz-cd}
\usepackage{tikzit}

\tikzstyle{w puncture}=[fill=none, draw=none, shape=circle, text={rgb,255: red,226; green,121; blue,46}, tikzit fill={rgb,255: red,226; green,121; blue,46}]
\tikzstyle{z puncture}=[fill=none, draw=none, shape=circle, text={rgb,255: red,84; green,174; blue,50}, tikzit fill={rgb,255: red,84; green,174; blue,50}]
\tikzstyle{intersection point}=[fill=black, draw=black, shape=circle, minimum size=2pt, inner sep=0pt]
\tikzstyle{dot}=[fill=black, draw=none, shape=circle, inner sep=0pt, minimum size=2.5pt]
\tikzstyle{su2 puncture}=[fill=none, draw=none, shape=circle, tikzit fill={rgb,255: red,218; green,59; blue,38}, text={rgb,255: red,218; green,59; blue,38}]
\tikzstyle{stop}=[fill=black, draw=black, shape=rectangle]
\tikzstyle{verma}=[fill=black, draw=black, shape=circle, inner sep=0pt, minimum size=7pt]

\tikzstyle{brane}=[-, thick]
\tikzstyle{map}=[dashed, ->, draw={rgb,255: red,50; green,116; blue,181}]
\tikzstyle{w puncture strand}=[-, draw={rgb,255: red,226; green,121; blue,46}, thick]
\tikzstyle{z puncture strand}=[-, draw={rgb,255: red,84; green,174; blue,50}, thick]
\tikzstyle{arrow}=[->]
\tikzstyle{I-brane}=[-]
\tikzstyle{disk}=[-, draw=none, tikzit draw={rgb,255: red,255; green,128; blue,0}, fill={rgb,255: red,191; green,191; blue,191}]
\tikzstyle{brane with arrow}=[->, thick, draw={rgb,255: red,50; green,116; blue,181}]
\tikzstyle{sl3 e1 brane}=[-, draw={rgb,255: red,50; green,116; blue,181}, thick]
\tikzstyle{sl3 w1 puncture strand}=[-, thick, draw={rgb,255: red,50; green,116; blue,181}]
\tikzstyle{red I-brane}=[-, draw={rgb,255: red,218; green,59; blue,38}, thick]
\tikzstyle{red I-brane with arrow}=[draw={rgb,255: red,218; green,59; blue,38}, ->, thick]
\tikzstyle{su2 puncture strand}=[-, draw={rgb,255: red,218; green,59; blue,38}, thick, densely dotted]
\tikzstyle{braid}=[-, very thick]
\tikzstyle{braid arrow}=[very thick, ->, draw=black]

\usepackage{hyperref}
\hypersetup{
    colorlinks,
    linkcolor={blue!80!black},
    citecolor={blue!80!black},
    urlcolor={blue!80!black}
}
\usepackage{caption}
\usepackage{subcaption}
\usetikzlibrary{shapes.geometric}
\usepackage{wasysym}
\usepackage{pifont}
\usetikzlibrary{calc}
\usetikzlibrary{decorations.markings}
\usetikzlibrary{intersections}

\newtheorem{theorem}{Theorem}[section]
\newtheorem{lemma}[theorem]{Lemma}
\newtheorem{corollary}[theorem]{Corollary}
\newtheorem{proposition}[theorem]{Proposition}
\newtheorem{definition}[theorem]{Definition}

\newtheorem{remark}[theorem]{Remark}

\def\R{\mathbb{R}}
\def\C{\mathbb{C}}
\def\Z{\mathbb{Z}}

\def\MCB{\mathcal{M}^\times}

\DeclareMathOperator{\Hom}{Hom}
\DeclareMathOperator{\module}{mod}

\title{Aganagic's invariant is Khovanov homology}
\author{Elise LePage and Vivek Shende}
\date{}
\begin{document}

\maketitle

\begin{abstract}
    On the Coulomb branch of a quiver gauge theory, there is a family of functions parameterized by choices of points in the punctured plane.  Aganagic has  predicted that Khovanov homology
    can be recovered from the braid group action on Fukaya-Seidel categories arising from monodromy in said space of potentials.
    These categories have since been rigorously studied, and shown to contain a certain (combinatorially defined) category on which Webster had previously constructed a (combinatorially defined) braid group action from which the Khovanov homology can be recovered. 

    Here we show, by a direct calculation, that the aforementioned containment intertwines said combinatorially defined braid group action with the braid group action arising naturally from monodromy.  This provides a mathematical verification that Aganagic's proposal gives a symplectic construction of Khovanov homology -- with both gradings, and over the integers. 
\end{abstract}

\section{Introduction}

Around the turn of the millenium, 
knot theorists learned to associate certain graded-vector-space-valued invariants to knots and links.  In fact, two rather different such invariants appeared. 
The first, Khovanov homology, 
is a combinatorially defined chain 
complex whose Euler
characteristic recovered the Jones 
polynomial \cite{Khovanov}, and is by now situated in the context of the `categorification' of quantum groups and their representations \cite{Rouquier-2kac, Khovanov-Lauda-diagrammatics-1, Khovanov-Lauda-diagrammatics-2, Webster-weighted, webster}.  
The second, the Ozsv\'ath-Szab\'o `Heegard-Floer' homology, is constructed using Lagrangian Floer homology in a symmetric power of a Heegard surface which splits the knot (or 3-manifold) of interest into trivial pieces, and arose from Atiyah-Floer type considerations around the Seiberg-Witten
gauge theory \cite{Ozsvath-Szabo-1, Ozsvath-Szabo-knot}.  

It is desirable to bring these two theories into a common framework.  One natural direction is to try and construct the Khovanov homology in a manner as similar as possible to the Heegard-Floer homology. 
The first progress in this direction was the
`symplectic Khovanov homology' of Seidel and Smith \cite{SS}, further developed by many authors \cite{manolescu, manolescu2, abouzaid-smith-arc,  abouzaid-smith:khovanov,  Cheng-bigrading}. 
These works establish that Khovanov homology, at least over $\mathbb{Q}$, can be realized in terms of the Lagrangian Floer homology in a certain space of representation theoretic origin.  (One can also travel the other direction and try to understand Heegard-Floer homology as categorified representation theory of $\mathfrak{gl}(1|1)$ \cite{Rouquier-Manion}.)

More recently, Aganagic has proposed a new setting for constructing homological knot invariants from Lagrangian Floer theory \cite{aganagic-icm, aganagic-knot-2}; it is related by Atiyah-Floer type considerations to the gauge theory which Witten  previously predicted should produce Khovanov homology \cite{witten-fivebranes, gaiottowitten2011}.  
Aganagic's proposal has a version appropriate to each semisimple Lie algebra, and also a generalization to superalgebras, with the $\mathfrak{gl}(1|1)$ case recovering the Heegard-Floer theory.  For $\mathfrak{sl}(2)$, it is a close relative of the Seidel-Smith setup, differing by removing a divisor and adding a superpotential.  A key advantage is that the crucial `Jones grading' on the homology groups has a far more transparent origin (the grading arises from a class in ordinary, rather than symplectic, cohomology).  

We recall the setup. 
Let $\Gamma$ be a quiver whose underlying graph is the Dynkin diagram of the desired Lie algebra.  Fix a dimension vector $\vec{d}$, and let 
$\mathcal{M}^\times(\Gamma, \vec{d})$ be the multiplicative Coulomb branch associated to the quiver gauge theory $(\Gamma, \vec{d})$, as mathematically defined in \cite{BFN}.  In addition, fix a collection $\mathbf{a}$ of (`red') points in the annulus, labelled by the vertices of the Dynkin diagram.  There is a certain regular function $\mathcal{W}_{\mathbf{a}}:\mathcal{M}^\times(\Gamma, \vec{d}) \to \mathbb{C}$; we are interested in the Fukaya-Seidel category $Fuk(\mathcal{M}^\times(\Gamma, \vec{d}), \mathcal{W}_{\mathbf{a}})$. 

Monodromy in the space of parameters `$\mathbf{a}$' gives an action of the annular braid group: 
$$\rho_{A}: ABr_{|\mathbf{a}|} \to Aut(Fuk(\mathcal{M}^\times(\Gamma, \vec{d}), \mathcal{W}_{\mathbf{a}})).$$

Let $\beta$ be a 2n-stranded ordinary (not annular) braid, and $\overline{\beta}$ its plat closure.  Consider the special case of the above setup where $\Gamma = \bullet, \vec{d} = (n)$, $|\mathbf{a}| = 2n$.  In this case, Aganagic proposed a certain object $\cup_A^{n} \in Fuk(\mathcal{M}^\times(\bullet, n), \mathcal{W}_{\mathbf{a}})$ and conjectured that the Khovanov homology of $\overline{\beta}$ is recovered by the Hom pairing (up to a degree shift involving the writhe which we systematically omit):\footnote{In \cite{aganagic-knot-2, ALR}, the braid is put instead in the first factor, but in compensation (for therefore having the mirror knot) there is a   $q^{1/2} \to -q^{-1/2}$ substitution in \cite[Eq. 5.64]{ALR}.} 
\begin{equation}\label{aganagic conjecture}
    Kh(\overline{\beta}) \stackrel{?}{\cong} \mathrm{Hom}_{Fuk(\mathcal{M}^\times(\bullet, n), \mathcal{W}_{\mathbf{a}})}(\cup^{n}_A, \rho_A(\beta)\, \cup^{n}_A) 
\end{equation}
An argument that the right hand side is a link invariant decategorifying to the Jones polynomial was presented in \cite{ALR}, along with calculational techniques motivated by the new geometric setup.  

\vspace{2mm}

A mathematically rigorous account of 
$Fuk(\mathcal{M}^\times(\Gamma, \vec{d}), \mathcal{W}_{\mathbf{a}})$ 
is now available \cite{ADLSZ}.  The main result of said article is an embedding of 
Webster's \cite{Webster-weighted} combinatorial/diagrammatically defined category of modules over a quiver Hecke algebra: 
\begin{equation} \label{five authors result} \mathrm{Perf} \, \mathcal{A}^{cyl}(\Gamma, \vec{d}, \mathbf{a}) \hookrightarrow Fuk(\mathcal{M}^\times(\Gamma, \vec{d}), \mathcal{W}_{\mathbf{a}}).
\end{equation}
(Webster's construction depends on the choice of cyclically ordered points on labeled by representations; by writing `$\mathbf{a}$' on the LHS above, we mean to ask that the elements of $\mathbf{a}$ have distinct arguments, take these arguments as points on the circle, and label by the fundamental representation corresponding to the given node of the diagram.)

Webster has previously constructed, 
by explicit diagrammatic/combinatorial formulas, a braid group action $\rho_W$
on $\mathrm{Perf} \, \mathcal{A}^{cyl}(\Gamma, \vec{d}, \mathbf{a})$, and a certain distinguished object $\cup^{n}_W$, and shown that 
the Khovanov homology of a plat closure $\overline{\beta}$ of a braid $\beta$ is recovered as \cite{webster}:\footnote{Webster originally worked with a planar, rather than cylindrical, version of the quiver Hecke algebra.  The cylindrical version appears in his later \cite{webster2019coherent}, and carries an annular braid group action given by the same formulas, hence can also be used to compute Khovanov homology when one restricts attention to the ordinary braid group.  Working with the full annular braid group presumably recovers the annular Khovanov homology.}
\begin{equation}\label{webster result}
    Kh(\overline{\beta}) \cong \mathrm{Hom}_{\mathcal{A}(\bullet, n, 2n)-mod}(\cup^{n}_W, \rho_W(\beta) \cup^{n}_W).
\end{equation}

Here we show: 

\begin{theorem} \label{thm: intertwining}
    For $\Gamma = \bullet$, the embedding \eqref{five authors result} intertwines the braid group representations $\rho_W$ and $\rho_A$, and
    carries $\cup_W^n \mapsto \cup_A^n$.    
    In particular, 
    $$
   Kh(\overline{\beta}) \cong \mathrm{Hom}_{Fuk(\mathcal{M}^\times(\bullet, n), \mathcal{W}(2n))}( \cup^{n}_A, \rho_A(\beta) \, \cup^{n}_A).
   $$
\end{theorem}

We note that Theorem \ref{thm: intertwining} is valid over $\mathbb{Z}$.\footnote{Webster's construction of the isomorphism between Hom spaces and Khovanov homology was constructed in  \cite{webster}, which  officially took coefficients in a field (of arbitrary characteristic).  However, the construction of the morphism there is valid over the integers, and such a morphism can be checked to be an isomorphism after base change to every field.  We thank Webster for this clarification.}  By contrast, at present, the Seidel-Smith construction is only known to recover Khovanov homology over $\mathbb{Q}$ \cite{abouzaid-smith-arc, abouzaid-smith:khovanov}.

Our proof of Theorem \ref{thm: intertwining} is essentially a direct calculation, using the tools of \cite{ADLSZ}.  
Let us make some brief comments here on the proof.  First, since the category $\mathcal{A}^{cyl}(\Gamma, \vec{d}, \vec{m})$
consists of a collection of objects whose morphism spaces are always in  degree zero, to show that two braid group actions on the module category agree amounts to constructing isomorphisms $\rho_W(\beta)(X) \cong \rho_A(\beta)(X)$ for the (Yoneda modules of) said objects, and  checking that the effect on morphisms (which can be compared via said  isomorphisms) is the same -- as opposed to constructing a tower of further structures.  In fact, we will largely avoid even computing the effect on morphisms by using various auxiliary gradings to constrain possible automorphisms: the key result here is Corollary \ref{cor:automorphisms}.   As far as determining the action on objects is concerned, the main innovation here is to find a convenient resolution of the generating objects, for which the effect of braiding is comparably simple with respect to both $\rho_W$ and $\rho_A$.\footnote{Had we taken a more direct approach (directly apply $\rho_A$ to generating objects and resolve the result in the most naive way, i.e. by the procedure of Corollary \ref{weak generation}), then for each $n$, we would obtain a complex of size $3^{n}$ on the $\rho_A$ side (not all of whose maps we are entirely certain we can rigorously determine) whereas, on the $\rho_W$ side, the expected complex has $n+1$ terms, and thus we we would be left with the problem of identifying a large acyclic subcomplex.}  

Producing this resolution and determining the action of braiding is the most effortful part of this article, occupying Sections \ref{2 red 1 black}, \ref{Lambda object}, \ref{zeroes}, \ref{more lambda object}.  We then establish the intertwining of braid group actions in Section \ref{sec: braiding}, and conclude with a discussion of cups and caps in Section \ref{sec: cupcap}.
The remaining Sections are \ref{sec: KLRW} and \ref{sec: ADLSZ}, which review the relevant notions about KLRW algebra (from \cite{webster, Webster-weighted, webster2016tensor}) and Fukaya categories of Coulomb branches (from \cite{ADLSZ}), respectively, and  \ref{sec: conesurgery}, which records how the general `cone over a Reeb chord is surgery at infinity' prescription of e.g. \cite{GPS2} looks from the point of view of the cylindrical model of \cite{ADLSZ}. 

\vspace{2mm}

While our proof proceeds by direct computation, let us mention that Aganagic has previously sketched in   \cite[Sec. 8]{aganagic-knot-2} an approach to establishing such intertwinings by arguing that the action is uniquely characterized by certain perverse data in the sense of \cite{Chuang-Rouquier}.

\vspace{2mm}
{\bf Acknowledgements.} This article stems from ongoing joint work with Mina Aganagic and Peng Zhou, who we thank for many ideas and helpful discussions.  We also thank Ben Webster.
VS is supported by the
Villum Fonden Villum Investigator grant 37814, Novo Nordisk Foundation grant
NNF20OC0066298, and Danish National Research Foundation grant DNRF157. 

\section{KLRW diagrammatics} \label{sec: KLRW}

We recall in this section the diagrammatics of Khovanov-Lauda-Rouquier-Webster algebras, originally developed in \cite{Khovanov-Lauda-diagrammatics-1, Khovanov-Lauda-diagrammatics-2, Rouquier-2kac, Webster-weighted}. 

\subsection{Generators and relations}

\begin{figure}[ht]
    \centering
     \begin{subfigure}[b]{0.1\textwidth}
     \centering     
     \begin{tikzpicture}[very thick]
        \def\br{40}
        \draw (0,0) to [out=\br, in=-\br] node[pos=0, below ]{$(i)$} (0,2);
        \draw (0.5,0) to [out=180-\br, in=\br - 180] node[pos=0, below ]{$(i)$} (0.5,2);
        \node at (1,1) {$=0$};
        \end{tikzpicture}
        \caption{bigon}\label{fig:KLRW-bigon}
     \end{subfigure} 
     \hspace{2cm}
      \begin{subfigure}[b]{0.3\linewidth}
     \centering     
     \begin{tikzpicture}[very thick]
     \begin{scope}
        \def\br{40}
        \draw (0,0) to [out=\br, in=-\br] node[pos=0, below ]{$(j)$} (0,2);
        \draw (0.5,0) to [out=180-\br, in=\br - 180] node[pos=0, below ]{$(i)$} (0.5,2);
        \end{scope}
        \node at (1,1) {$= u $};
        \begin{scope}[shift=({1.5,0})]
        \def\br{40}
        \draw (0,0) to   node[pos=0, below ] {$(j)$} (0,2);
        \draw (0.5,0) to node[pos=0.5, fill, circle, inner sep=2pt]{} node[pos=0, below ]{$(i)$} (0.5,2);
        \end{scope}
        \node at (2.5,1) {$- u $};
        \begin{scope}[shift=({3,0})]
        \def\br{40}
        \draw (0,0) to  node[pos=0.5, fill, circle, inner sep=2pt]{} node[pos=0, below ] {$(j)$} (0,2);
        \draw (0.5,0) to   node[pos=0, below ]{$(i)$} (0.5,2);
        \end{scope}
        \end{tikzpicture}
        \caption{bigon for  $(j) \to (i)$ }
     \label{fig:KLRW-ij}\end{subfigure} 
     \hspace{2cm}
    \begin{subfigure}[b]{0.2\textwidth}
     \centering     
     \begin{tikzpicture}[very thick]
        \def\br{20}
        \begin{scope}
        \draw[red, densely dotted] (0,0) to node[pos=0, below ]{$[i]$} (0,2);
        \draw (0.4,0) to [out=90, in=-90] node[pos=0, below ]{$(i)$} (-0.4,1);
        \draw (-0.4, 1)  to [out=90, in=-90]  (0.4, 2);
        \end{scope}
        \node at (0.7,1) {$=\; u$};
        \begin{scope}[shift =({1.5,0})]
        \draw[red, densely dotted] (0,0) to node[pos=0, below ]{$[i]$} (0,2);
        \draw (0.5,0) to node[pos=0.5, fill, circle, inner sep=2pt]{}  node[pos=0, below ]{$(i)$} (0.5,2);
        \end{scope}
        \end{tikzpicture}
        \caption{bigon with red}\label{fig:KLRW-bigon-red}
     \end{subfigure}
     \vskip 1cm
          \begin{subfigure}[b]{0.4\linewidth}
     \centering     
       \begin{tikzpicture}[very thick]
        \def\br{40}
        \begin{scope}
        \draw (0,0) -- node[pos=0, below ]{$(i)$}  (1,2);
        \draw (1,0) -- node[pos=0, below ]{$(i)$}  (0,2);
        \draw (0.5,0) to [bend left] node[pos=0, below ]{$(j)$} (0.5,2);
        \end{scope}
        \node at (1.5,1) {$-$};
        \begin{scope}[shift=({2,0})]
        \draw (0,0) -- node[pos=0, below ]{$(i)$}  (1,2);
        \draw (1,0) -- node[pos=0, below ]{$(i)$}  (0,2);
         \draw (0.5,0) to [bend right] node[pos=0, below ]{$(j)$} (0.5,2);
        \end{scope}
 \node at (3.5,1) {$= u \hbar $};
 \begin{scope}[shift=({4,0})]
        \draw (0.1,0) -- node[pos=0, below ]{$(i)$}  (0.1,2);
        \draw (1.1,0) -- node[pos=0, below ]{$(i)$}  (1.1,2);
         \draw (0.6,0) to node[pos=0, below ]{$(j)$} (0.6,2);
        \end{scope}
        \end{tikzpicture}
        \caption{braid with neighbour $(j) \to (i)$.}
     \label{fig:KLRW-jij}\end{subfigure} 
     \hspace{2cm}
         \begin{subfigure}[b]{0.3\textwidth}
     \centering     
       \begin{tikzpicture}[very thick]
        \def\br{40}
        \begin{scope}
        \draw (0,0) to [bend right]  node[pos=0, below ]{$(i)$}  (1,2);
        \draw (1,0) to [bend right] node[pos=0, below ]{$(i)$}  (0,2);
        \draw[red, densely dotted] (0.5,0) to node[pos=0, below ]{$[i]$} (0.5,2);
        \end{scope}
        \node at (1.5,1) {$-$};
        \begin{scope}[shift=({2,0})]
        \draw (0,0) to [bend left] node[pos=0, below ]{$(i)$}  (1,2);
        \draw (1,0) to [bend left] node[pos=0, below ]{$(i)$}  (0,2);
        \draw[red, densely dotted] (0.5,0) to node[pos=0, below ]{$[i]$} (0.5,2);
        \end{scope}
 \node at (3.5,1) {$=u \hbar $};
 \begin{scope}[shift=({4.1,0})]
        \draw (0,0) -- node[pos=0, below ]{$(i)$}  (0,2);
        \draw (1,0) -- node[pos=0, below ]{$(i)$}  (1,2);
        \draw[red, densely dotted] (0.5,0) to node[pos=0, below ]{$[i]$} (0.5,2);
        \end{scope}
        \end{tikzpicture}
        \caption{braid with red}\label{fig:KLRW-braid-red}
     \end{subfigure} 
  \vskip 1cm
          \begin{subfigure}[b]{0.4\textwidth}
     \centering     
       \begin{tikzpicture}[very thick]
        \def\br{40}
        \begin{scope}
        \draw (0,0) -- node[pos=0, below ]{$(i)$}  (1,2);
        \draw (1,0) -- node[pos=0, below ]{$(i)$} node[pos=0.8, inner sep=2pt, circle, fill]{}  (0,2);
        \end{scope}
        \node at (1.5,1) {$-$};
        \begin{scope}[shift=({2,0})]
        \draw (0,0) -- node[pos=0, below ]{$(i)$}  (1,2);
        \draw (1,0) -- node[pos=0, below ]{$(i)$}  (0,2);
         \draw (1,0) -- node[pos=0, below ]{$(i)$} node[pos=0.2, inner sep=2pt, circle, fill]{}  (0,2);
        \end{scope}
        \node at (3.5,1) {$= \;\; \hbar $};
         \begin{scope}[shift=({4.3,0})]
        \draw (0,0) -- node[pos=0, below ]{$(i)$}  (0,2);
        \draw (1,0) -- node[pos=0, below ]{$(i)$}  (1,2);
        \end{scope}
        \end{tikzpicture}
        \caption{dot-pass-crossing}\label{fig:KLRW-skein-1}
     \end{subfigure} \hfill
         \begin{subfigure}[b]{0.4\textwidth}
     \centering     
       \begin{tikzpicture}[very thick]
        \def\br{40}
        \begin{scope}
        \draw (0,0) -- node[pos=0.2, inner sep=2pt, circle, fill]{}  node[pos=0, below ]{$(i)$}  (1,2);
        \draw (1,0) -- node[pos=0, below ]{$(i)$}  (0,2);
        \end{scope}
        \node at (1.5,1) {$-$};
        \begin{scope}[shift=({2,0})]
        \draw (0,0) -- node[pos=0.8, inner sep=2pt, circle, fill]{}  node[pos=0, below ]{$(i)$}  (1,2);
        \draw (1,0) -- node[pos=0, below ]{$(i)$}  (0,2);
         \draw (1,0) -- node[pos=0, below ]{$(i)$}  (0,2);
        \end{scope}
        \node at (3.5,1) {$= \;\; \hbar $};
         \begin{scope}[shift=({4.3,0})]
        \draw (0,0) -- node[pos=0, below ]{$(i)$}  (0,2);
        \draw (1,0) -- node[pos=0, below ]{$(i)$}  (1,2);
        \end{scope}
        \end{tikzpicture}
        \caption{another dot-pass-crossing}\label{fig:KLRW-skein-2}
     \end{subfigure} 
         
    \caption{Nontrivial KLRW relations. Exchanging $i$ and $j$ in diagrams (b) and (d), i.e. if we have an arrow $(j) \gets (i)$, the right-hand-side gets an extra (-1) factor.  
    }
    \label{fig:KLRW} 
\end{figure}
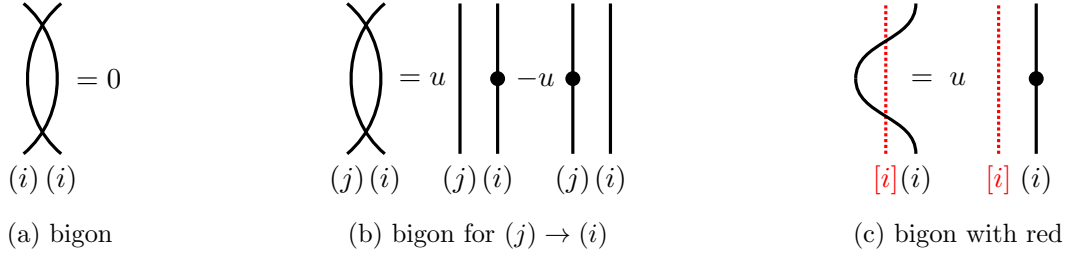
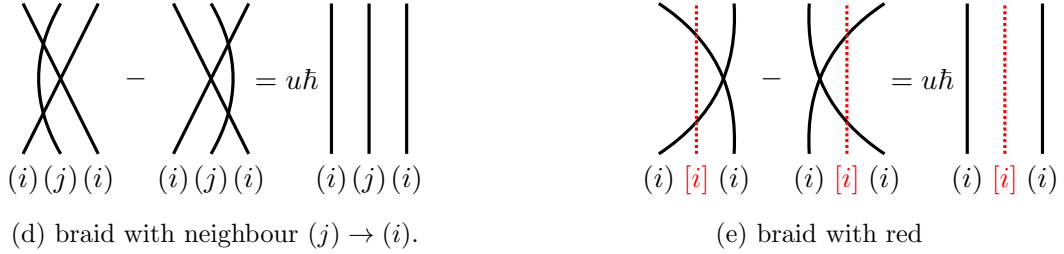
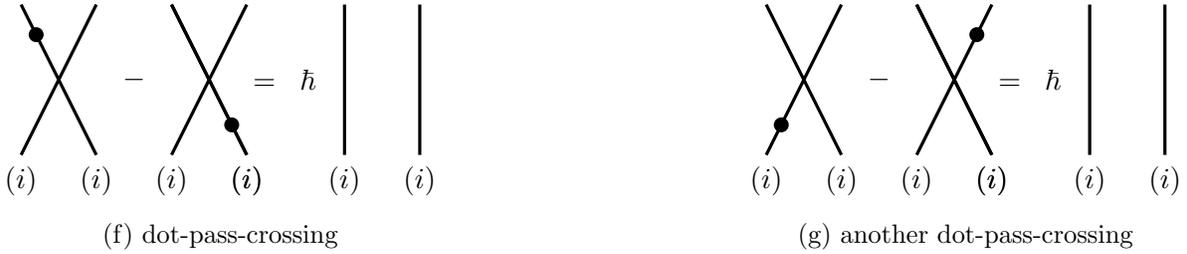

Let $\Gamma$ be a directed graph with  no multiple edges, or loops. Pick a natural number $d_i \ge 0$ for each vertex $i$ of $\Gamma$, and a collection (possibly empty) $F$ of points on a line, labeled by vertices of $\Gamma$.  In diagrams, we depict the points of $F$ as red.\footnote{Webster allows red points to be colored by arbitrary dominant weights \cite{webster}.  Here we restrict attention to the case of fundamental weights, which we identify with nodes of the Dynkin diagram.} The KLRW category 
$\mathcal{C}_{\Gamma, \vec{d}, F}$ is defined as follows: 

\begin{itemize}
    \item Objects are collections of 
    points on a line $\R$, all distinct
    from the points of $F$, with
    $d_i$ points labeled by the vertex $i$
    of $\Gamma$.  In diagrams, we depict
    these points as black. 
    \item Morphism spaces
    are generated by strand diagrams in
    the plane $\R \times [0,1]$, with no  horizontal or vertical
    tangencies, and no non-generic intersections. Black strands
    may be decorated by 
    dots.  
    \item 
    Composition $D_1 \circ D_2$ is given by stacking the diagram $D_1$ on top of $D_2$.
    \item 
    Diagrams are considered
    up to isotopy and satisfy relations
    in Figure \ref{fig:KLRW}. 
\end{itemize}

\begin{remark}
Any isotopy of black points (not crossing each other or red points) gives an isomorphism of the corresponding objects
of $\mathcal{C}_{\Gamma, \vec{d}, F}$, and the space of such isotopies is  contractible.\footnote{In the cylindrical case, one marks a point on the circle, and carrying an object across said point shifts a certain grading, see Sec. \ref{KLRW grading} below; in view of which the space of objects in a given isomorphism-of-degree-zero-with-respect-to-this-grading class remains contractible.}
Webster prefers to choose a representative 
object in each isotopy class, and work with the algebra $\mathcal{A}_{\Gamma, \vec{d}, F}$ given by the 
endomorphisms of the direct sum of these representatives.  
This is obviously equivalent to the presentation here, and we will not belabor the translations in what follows. 
\end{remark}

We write 
$\mathcal{C}^{cyl}_{\Gamma, \vec d, F}$ for the analogous structure with 
the line $\R$ replaced by a circle, and
the plane correspondingly replaced 
by a cylinder \cite{webster2022coherent}.

\subsection{Modules, bimodules, and braiding}
In general, for a category $C$,
we have the left modules 
$C-\module := \Hom(C^{op}, \Z-\module)$ and right modules $\module-C := \Hom(C, \Z-\module)$.  Here, by $\Z-\module$ we mean the stable $\infty$- (or equivalently dg derived) category of complexes of $\Z$-modules, localized along quasi-isomorphisms.  As such, all functors discussed below will be (automatically) `derived' and we correspondingly omit the $R$ or $L$ notations etc.

For categories $C, D$ one similarly defines bimodules as `bifunctors' $C^{op} \times D \to \Z-\module$.  
Rather than discuss what is a bifunctor, one can take as the definition
$$C - \module - D :=  \Hom(D, C-\module) =  \Hom(C^{op} , \module-D).$$  
We denote the tautological maps as 
\begin{eqnarray*}
  \module-C \times C -\module -D & \to & \module-D \\
  (c, b) \mapsto c \otimes_C b
\end{eqnarray*}
\begin{eqnarray*}
   C -\module -D \times D-\module& \to & C-\module \\
  (b, d) \mapsto b \otimes_D d
\end{eqnarray*} 
There is similarly an operation $C-\module-D \times D-\module-E \to C-\module-E$, which we denote $\otimes_D$.  
\vspace{2mm}

Let us return to considering KLRW categories.  The following discussion makes sense equally for the cylindrical and usual KLRW categories, so we omit the $cyl$ notations. The definition of such categories $\mathcal{C}$ asserts that the diagonal bimodule $\Hom(\theta', \theta)$ is given by the linear span of strand diagrams between with $\theta'$ at the top and $\theta$ at the bottom, modulo certain relations.  Or in other words, 
the Yoneda module 
$\Hom(\cdot, \theta): \mathcal{C}^{op} \to \Z-mod$ sends a configuration $\theta'$ to the linear span of all diagrams with $\theta'$ at the top and $\theta$ at the bottom.

One can describe more bimodules by introducing new elements in the diagrams, and imposing new relations. 

Of particular relevance to us are the braiding bimodules introduced by Webster in \cite{webster}.  
Fix a positive braid $\beta \in Br_{|F|}$ expressed as a word in positive crossings. 
We write $\mathscr{B}_\beta$ for the
$\mathcal{C}_{\Gamma, \vec d, F'} - \module - \mathcal{C}_{\Gamma, \vec d, F}$ bimodule
where $\mathscr{B}_\beta(\theta', \theta)$ is the formal linear span of all strand diagrams with $\theta$ at the bottom and $\theta'$ at the top, where now the red strands are allowed to move and cross so as to trace the projection of the braid --  modulo the KLRW relations shown in Figure \ref{fig:KLRW}, along with the additional braid relations in Figure \ref{fig:KLRW-braiding}.  
Bifunctoriality of the bimodule is given by stacking KLRW diagrams at the top and bottom.

\begin{figure}
    \centering
         \begin{subfigure}[b]{0.3\textwidth}
     \centering     
       \begin{tikzpicture}[very thick]
        \def\br{40}
        \begin{scope}
        \draw[red, densely dotted] (0,0) -- node[pos=0, below ]{$[i]$}  (1,2);
        \draw[red, densely dotted] (1,0) -- node[pos=0, below ]{$[k]$}  (0,2);
        \draw (0.5,0) to [bend left] node[pos=0, below ]{$(j)$} (0.5,2);
        \end{scope}
        \node at (1.5,1) {$=$};
        \begin{scope}[shift=({2,0})]
        \draw[red, densely dotted] (0,0) -- node[pos=0, below ]{$[i]$}  (1,2);
        \draw[red, densely dotted] (1,0) -- node[pos=0, below ]{$[k]$}  (0,2);
         \draw (0.5,0) to [bend right] node[pos=0, below ]{$(j)$} (0.5,2);
        \end{scope}
        \end{tikzpicture}
        \caption{}
     \end{subfigure} 
     \hspace{0.5cm}
         \begin{subfigure}[b]{0.3\textwidth}
     \centering     
       \begin{tikzpicture}[very thick]
        \def\br{40}
        \begin{scope}
        \draw[red, densely dotted] (0,0) -- node[pos=0, below ]{$[i]$}  (1,2);
        \draw[red, densely dotted] (1,0) -- node[pos=0, below ]{$[j]$}  (0,2);
        \draw (1.5,0) to [bend left] node[pos=0, below ]{$(k)$} (-0.5,2);
        \end{scope}
        \node at (1.5,1) {$=$};
        \begin{scope}[shift=({2,0})]
        \draw[red, densely dotted] (0,0) -- node[pos=0, below ]{$[i]$}  (1,2);
        \draw[red, densely dotted] (1,0) -- node[pos=0, below ]{$[j]$}  (0,2);
        \draw (1.5,0) to [bend right] node[pos=0, below ]{$(k)$} (-0.5,2);
        \end{scope}
        \end{tikzpicture}
        \caption{}
     \end{subfigure} 
  \hspace{0.5cm}
  \begin{subfigure}[b]{0.3\textwidth}
     \centering     
       \begin{tikzpicture}[very thick]
        \def\br{40}
        \begin{scope}
        \draw[red, densely dotted] (0,0) -- node[pos=0, below ]{$[i]$}  (1,2);
        \draw[red, densely dotted] (1,0) -- node[pos=0, below ]{$[k]$}  (0,2);
        \draw[red, densely dotted] (0.5,0) to [bend left] node[pos=0, below ]{$[j]$} (0.5,2);
        \end{scope}
        \node at (1.5,1) {$=$};
        \begin{scope}[shift=({2,0})]
        \draw[red, densely dotted] (0,0) -- node[pos=0, below ]{$[i]$}  (1,2);
        \draw[red, densely dotted] (1,0) -- node[pos=0, below ]{$[k]$}  (0,2);
        \draw[red, densely dotted] (0.5,0) to [bend right] node[pos=0, below ]{$[j]$} (0.5,2);
        \end{scope}
        \end{tikzpicture}
        \caption{}
     \end{subfigure}         
    \caption{}\label{fig:KLRW-braiding}
\end{figure}

It is essentially immediate from the definition that 
$$\mathcal{B}_{\beta' \beta} = \mathcal{B}_{\beta'} \otimes_{\mathcal{C}_{\Gamma, \vec d, F'}} \mathcal{B}_{\beta},$$
Webster ``checks the braid relations'' by constructing a canonical representative for the braid associated to each positive word, showing that the above defines a representation of the positive braid monoid into bimodules.  Webster proves that the bimodules are in fact perfect, so in particular the functor
\begin{eqnarray*}
    \mathbb{B}'_\beta: \mathcal{C}_{\Gamma, \vec d, F} & \to & \mathcal{C}_{\Gamma, \vec d, F'}-\module \\
    \theta & \mapsto & \mathcal{B}_\beta \otimes_{\mathcal{C}_{\Gamma, \vec d, F}} \theta 
\end{eqnarray*}
in fact lands in 
$\mathcal{C}_{\Gamma, \vec d, F'}-\mathrm{perf}$, i.e., the image of every object under this functor admits a finite resolution by Yoneda representatives of objects in $\mathcal{C}_{\Gamma, \vec d, F'}$.  We write 
$\mathbb{B}_\beta$ for the induced endofunctor on $\mathcal{C}_{\Gamma, \vec d, F'}-\mathrm{perf}$; varying over $\beta$, we have an action of the positive braid monoid on $\mathcal{C}_{\Gamma, \vec d, F'}-\mathrm{perf}$.  
Webster also proves that the $\mathbb{B}_\beta$ are invertible, and hence that the positive braid monoid action extends to an action of the  braid group.

We turn to  bimodules which compare categories for different numbers of red points, eventually to be associated to the `cups' and `caps' in plat diagrams.  Suppose in a configuration $F$ of red points, there is a pair $\pi$ of two adjacent red points with dual labels $i$ and $i^*$. 
We write $F \setminus \pi$ for the red point configuration obtained by removing these two points.  If $\vec d = (d_1, d_2, \ldots)$, we write $\vec d - 1_i := (d_1, d_2, \ldots, d_i-1, \ldots)$. 

There is a certain bimodule $\mathfrak{R}_\pi \in \mathcal{C}_{\Gamma, \vec d, F}-\module-\mathcal{C}_{\Gamma, \vec d - 1_i, F \setminus \pi}$.  
As usual, 
$\mathfrak{R}_\pi$ is defined taking
$\mathfrak{R}_\pi(\theta', \theta)$ 
to the be the space spanned by certain diagrams, modulo certain relations.  The diagrams are described as follows.  Above a certain horizontal strip, they are $\mathcal{C}_{\Gamma, \vec d, F}$ diagrams.  Below the strip they are $\mathcal{C}_{\Gamma, \vec d - 1_i, F \setminus \pi}$ diagrams.  On the strip, 
the diagram should have  only straight vertical red and black lines save in the neighborhood of the pair $\pi$.  

The main content in defining this bimodule concerns what to do in this neighborhood.  The general description is somewhat involved; see \cite[Sec. 7]{webster}.  We will restrict attention to the case when 
$\Gamma = \bullet$.  In this case the description is that in the horizontal strip, all lines are vertical save in the region containing the pair $\pi$, where the diagram must look like: 
\begin{equation}
	\vcenter{\hbox{\includegraphics{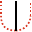}}}
\end{equation}
The relations are 
\begin{subequations}
\begin{align}
	\vcenter{\hbox{\begin{tikzpicture}[scale=.5]
	\begin{pgfonlayer}{nodelayer}
		\node [style=none] (0) at (-0.5, 0.5) {};
		\node [style=none] (1) at (0.5, 0.5) {};
		\node [style=none] (2) at (-0.5, 0) {};
		\node [style=none] (3) at (0.5, 0) {};
		\node [style=none] (4) at (0, -0.5) {};
		\node [style=none] (5) at (-0.75, 0.5) {};
	\end{pgfonlayer}
	\begin{pgfonlayer}{edgelayer}
		\draw [style=su2 puncture strand] (0.center) to (2.center);
		\draw [style=su2 puncture strand, bend right=45, looseness=1.25] (2.center) to (4.center);
		\draw [style=su2 puncture strand, bend right=45] (4.center) to (3.center);
		\draw [style=su2 puncture strand] (3.center) to (1.center);
		\draw [style=brane, in=90, out=-90] (5.center) to (4.center);
	\end{pgfonlayer}
\end{tikzpicture}}} &=0, \\
	\vcenter{\hbox{\begin{tikzpicture}[scale=.5]
	\begin{pgfonlayer}{nodelayer}
		\node [style=none] (0) at (-0.5, 0.5) {};
		\node [style=none] (1) at (0.5, 0.5) {};
		\node [style=none] (2) at (-0.5, 0) {};
		\node [style=none] (3) at (0.5, 0) {};
		\node [style=none] (4) at (0, -0.5) {};
		\node [style=none] (5) at (0.75, 0.5) {};
	\end{pgfonlayer}
	\begin{pgfonlayer}{edgelayer}
		\draw [style=su2 puncture strand] (0.center) to (2.center);
		\draw [style=su2 puncture strand, bend right=45, looseness=1.25] (2.center) to (4.center);
		\draw [style=su2 puncture strand, bend right=45] (4.center) to (3.center);
		\draw [style=su2 puncture strand] (3.center) to (1.center);
		\draw [style=brane, in=90, out=-90] (5.center) to (4.center);
	\end{pgfonlayer}
\end{tikzpicture}}} &=0, \\
	\vcenter{\hbox{\begin{tikzpicture}[scale=.5]
	\begin{pgfonlayer}{nodelayer}
		\node [style=none] (0) at (-0.5, 0.5) {};
		\node [style=none] (1) at (0.5, 0.5) {};
		\node [style=none] (2) at (-0.5, 0) {};
		\node [style=none] (3) at (0.5, 0) {};
		\node [style=none] (4) at (0, -0.5) {};
		\node [style=none] (5) at (0, 0.5) {};
		\node [style=none] (6) at (1, -0.5) {};
		\node [style=none] (7) at (-1, 0.5) {};
		\node [style=none] (8) at (1, -1) {};
	\end{pgfonlayer}
	\begin{pgfonlayer}{edgelayer}
		\draw [style=su2 puncture strand] (0.center) to (2.center);
		\draw [style=su2 puncture strand, bend right=45, looseness=1.25] (2.center) to (4.center);
		\draw [style=su2 puncture strand, bend right=45] (4.center) to (3.center);
		\draw [style=su2 puncture strand] (3.center) to (1.center);
		\draw [style=brane, in=90, out=-90] (5.center) to (4.center);
		\draw [style=brane, in=-90, out=90, looseness=0.75] (6.center) to (7.center);
		\draw [style=brane] (6.center) to (8.center);
	\end{pgfonlayer}
\end{tikzpicture}}} &= u\hbar\, \vcenter{\hbox{\begin{tikzpicture}[scale=.5]
	\begin{pgfonlayer}{nodelayer}
		\node [style=none] (0) at (-0.5, 0.5) {};
		\node [style=none] (1) at (0.5, 0.5) {};
		\node [style=none] (2) at (-0.5, 0) {};
		\node [style=none] (3) at (0.5, 0) {};
		\node [style=none] (4) at (0, -0.5) {};
		\node [style=none] (5) at (0, 0.5) {};
		\node [style=none] (6) at (1, -1) {};
		\node [style=none] (7) at (-1, -0.25) {};
		\node [style=none] (8) at (-1, 0.5) {};
	\end{pgfonlayer}
	\begin{pgfonlayer}{edgelayer}
		\draw [style=su2 puncture strand] (0.center) to (2.center);
		\draw [style=su2 puncture strand, bend right=45, looseness=1.25] (2.center) to (4.center);
		\draw [style=su2 puncture strand, bend right=45] (4.center) to (3.center);
		\draw [style=su2 puncture strand] (3.center) to (1.center);
		\draw [style=brane, in=90, out=-90] (5.center) to (4.center);
		\draw [style=brane, in=-90, out=90, looseness=0.75] (6.center) to (7.center);
		\draw [style=brane] (8.center) to (7.center);
	\end{pgfonlayer}
\end{tikzpicture}}}, \\
	\vcenter{\hbox{\begin{tikzpicture}[scale=.5]
	\begin{pgfonlayer}{nodelayer}
		\node [style=none] (0) at (-0.5, 0.5) {};
		\node [style=none] (1) at (0.5, 0.5) {};
		\node [style=none] (2) at (-0.5, 0) {};
		\node [style=none] (3) at (0.5, 0) {};
		\node [style=none] (4) at (0, -0.5) {};
		\node [style=none] (5) at (0, 0.5) {};
		\node [style=none] (6) at (-1, -0.5) {};
		\node [style=none] (7) at (1, 0.5) {};
		\node [style=none] (8) at (-1, -1) {};
	\end{pgfonlayer}
	\begin{pgfonlayer}{edgelayer}
		\draw [style=su2 puncture strand] (0.center) to (2.center);
		\draw [style=su2 puncture strand, bend right=45, looseness=1.25] (2.center) to (4.center);
		\draw [style=su2 puncture strand, bend right=45] (4.center) to (3.center);
		\draw [style=su2 puncture strand] (3.center) to (1.center);
		\draw [style=brane, in=90, out=-90] (5.center) to (4.center);
		\draw [style=brane, in=-90, out=90, looseness=0.75] (6.center) to (7.center);
		\draw [style=brane] (6.center) to (8.center);
	\end{pgfonlayer}
\end{tikzpicture}}} &= -u\hbar\, \vcenter{\hbox{\begin{tikzpicture}[scale=.5]
	\begin{pgfonlayer}{nodelayer}
		\node [style=none] (0) at (-0.5, 0.5) {};
		\node [style=none] (1) at (0.5, 0.5) {};
		\node [style=none] (2) at (-0.5, 0) {};
		\node [style=none] (3) at (0.5, 0) {};
		\node [style=none] (4) at (0, -0.5) {};
		\node [style=none] (5) at (0, 0.5) {};
		\node [style=none] (6) at (-1, -1) {};
		\node [style=none] (7) at (1, -0.25) {};
		\node [style=none] (8) at (1, 0.5) {};
	\end{pgfonlayer}
	\begin{pgfonlayer}{edgelayer}
		\draw [style=su2 puncture strand] (0.center) to (2.center);
		\draw [style=su2 puncture strand, bend right=45, looseness=1.25] (2.center) to (4.center);
		\draw [style=su2 puncture strand, bend right=45] (4.center) to (3.center);
		\draw [style=su2 puncture strand] (3.center) to (1.center);
		\draw [style=brane, in=90, out=-90] (5.center) to (4.center);
		\draw [style=brane, in=-90, out=90, looseness=0.75] (6.center) to (7.center);
		\draw [style=brane] (8.center) to (7.center);
	\end{pgfonlayer}
\end{tikzpicture}}}.
\end{align}
\end{subequations}

Webster proves that the bimodule is perfect; we write $\mathbb{R}_\pi$ the induced functor on perfect module categories.  Moreover, writing $\tau(\pi)$ for the positive braid exchanging the two red strands associated to $\pi$, Webster
shows that $\mathbb{R}_\pi$ has a left adjoint $\mathbb{R}_\pi^*$, and establishes an exact triangle: \cite{webster2016tensor}
\begin{equation} \label{webster exact triangle} 1 \xrightarrow{v} \mathbb{R}_\pi \circ \mathbb{R}_\pi^* \to \mathbb{B}_{\tau(\pi)} \xrightarrow{[1]} \end{equation}
where $v$ is the composition of the cap and cup operations. 

\subsection{Gradings} \label{KLRW grading}

There are additional gradings on the KLRW algebra, arising from rescaling elements of the algebra in a way that is consistent with the relations. In general, we can rescale each type of dot, the red-black crossings for each $i$, and the black-black crossings for each pair $i,j$ separately.

For the $\mathfrak{sl}_2$ case, this means we can rescale the red-black crossings, the black-black crossing, and the dot separately. Denote their rescaling gradings as $(1,0,0)$, $(0,1,0)$, and $(0,0,1)$, respectively. Using the relations, we find that the parameter $u$ must scale as $(2,0,-1)$ and $\hbar$ must scale as $(0,1,1)$. Rescaling by $(\frac{1}{2},0,1)$ does not rescale $u$, and rescaling by $(0,1,-1)$ does not rescale $\hbar$. Call the gradings under these rescalings the $u$- and $\hbar$-gradings. Rescaling by $(1,-2,2)$ leaves both $u$ and $\hbar$ invariant. We refer to this as the J grading, as it eventually gives rise to the q's in the Jones polynomial.  Braiding preserves the $u$-, $\hbar$-, and J gradings separately.
(In general, one can introduce independent $u, \hbar$ parameters and gradings for each node of $\Gamma$.)

The cylindrical KLRW algebra has additional gradings associated to wrapping a black strand labeled by $(i)$ around the back of the cylinder; we will call it the $C_i$ grading. To define this grading, pick a point $\varphi \in S^1$ away from the red and black points. If a black strand labeled by $(i)$ crosses $\varphi$ from right to left, the $C_i$ degree changes by $+1$, and if it crosses from left to right, $C_i$ changes by $-1$.  

We will use the gradings to constrain morphisms, in particular through the following results.  (We state the results for $\Gamma = \bullet$, which is the case we will use, but analogous results hold for general $\Gamma$, using the additional gradings mentioned above.) 

\begin{lemma} \label{morphism uniqueness}
    A morphism with a single crossing of a red and black strands and no other crossings or dots is the only morphism in that degree for that source and target, up to a scalar.
\end{lemma}

\begin{proof}
    Such a morphism has J degree $1$, $C_1$ degree 0, $u$-degree $\frac{1}{2}$ and $\hbar$-degree 0. Any morphism with the same source and target must have at least one red-black crossing. Adding dots will increase the $u$-degree, and no diagrams or parameters decrease the $u$-degree, so we cannot add any dots. Similarly, any additional crossings will increase the $u$- and $\hbar$-degrees, and no diagrams (except dots) or parameters decrease these degrees, so any morphism with the same source and target must have exactly one red-black crossing. There is only one way to draw a strand diagram for a fixed top and bottom and a single red-black crossing without winding around the back of the cylinder (which would change the $C_1$ degree).
\end{proof}

\begin{corollary} \label{cor:automorphisms}
    Let $\Phi$ be a graded automorphism of the KLRW category, and suppose given graded-degree-zero isomorphisms
    $\eta_\theta: \Phi(\theta) \cong \theta$ for all KLRW objects $\theta$.  Suppose that the induced action on KLRW diagrams acts as the identity on: 
    \begin{enumerate}
        \item  diagrams for which no black strands are between some fixed pair $\pi$ of adjacent red strands
        \item the particular element $\begin{tikzpicture}[scale=.5]
	\begin{pgfonlayer}{nodelayer}
		\node [style=none] (0) at (0.375, 0.5) {};
		\node [style=none] (1) at (0.125, 0.5) {};
		\node [style=none] (2) at (-0.375, -0.25) {};
		\node [style=none] (3) at (0.125, -0.25) {};
		\node [style=none] (4) at (-0.125, -0.25) {};
		\node [style=none] (5) at (-0.125, 0.5) {};
	\end{pgfonlayer}
	\begin{pgfonlayer}{edgelayer}
		\draw [style=su2 puncture strand] (3.center) to (1.center);
		\draw [style=su2 puncture strand] (5.center) to (4.center);
		\draw [style=brane, in=90, out=-90] (0.center) to (2.center);
	\end{pgfonlayer}
\end{tikzpicture}$
    \end{enumerate}
    Then there is a choice of signs $\epsilon(\theta) \in {\pm 1}$ such that $\epsilon(\theta) \cdot \eta_\theta$ is a natural  isomorphism from $\Phi$ to the identity.  
\end{corollary}
\begin{proof}
    Note that $\eta$ being a natural transformation is a property and not a further structure because the morphism spaces in question are all in homological degree zero. 
    What we must check is that $\eta$ commutes with all KLRW diagrams; it suffices to check commuting with the generating diagrams consisting of either a single crossing or a single dot.

    By Lemma \ref{morphism uniqueness}, $\Phi$ sends each red-black crossing to itself, up to a scalar, which our hypothesis only permits to be nontrivial when 
    the crossing involves one strand of the pair $\pi$.  
    There are four such red-black crossings, two involving the left red strand and two involving the right red strand.  For the two involving a given strand, they can be composed to give a dot outside $\pi$ (on which $\Phi$ acts trivially) so these scalars must be inverse, hence in fact be equal and $\pm 1$ since we are working over $\mathbb{Z}$.  We write $\epsilon_L$ for the sign on the left strand and $\epsilon_R$ for the sign on the right strand.  By condition (2) above, we must have $\epsilon_L = \epsilon_R$.  We choose $\epsilon(\theta) := (\epsilon_L)^k$, where $k$ is the number of black strands in $\theta$ between the pair $\pi$.  Then $\epsilon(\theta)\cdot \eta_\theta $ commutes with the red-black crossings.

    Next we will show that the black-black crossings between $\pi$ are sent to themselves. We have
    \begin{align}
    \Phi \Big(
    \vcenter{\hbox{
    \begin{tikzpicture}[scale=0.3]
        \draw[style=su2 puncture strand] (3,0) -- (3,2);
        \draw[style=su2 puncture strand] (0,0) -- (0,2);
        \draw[in=-90, out=90, looseness=1] (.5,0) to (-2.5,1);
        \draw (-2.5,1) -- (-2.5,2);
        \draw (.75,0) -- (.75,.25);
        \draw[in=-90, out=90, looseness=1] (.75,.25) to (-2.25,1.25);
        \draw (-2.25,1.25) -- (-2.25,2);
        \node at (1.5,.5) {$\scriptstyle\cdots$};
        \node at (-1.5,1.5) {$\scriptstyle\cdots$};
        \draw (2.25,0) -- (2.25,.75);
        \draw[in=-90, out=90, looseness=1] (2.25,.75) to (-.75,1.75);
        \draw (-.75,1.75) -- (-.75,2);
        \draw (2.5,0) -- (2.5,1);
        \draw[in=-90, out=90, looseness=1] (2.5,1) to (-.5,2);
    \end{tikzpicture}
    }} \cdot
    \vcenter{ \vspace{1em} \hbox{
    \begin{tikzpicture}[scale=0.3]
        \draw[style=su2 puncture strand] (3,0) -- (3,2);
        \draw[style=su2 puncture strand] (0,0) -- (0,2);
        \draw (.5,0) -- (.5,2);
        \node at (1,1) {$\scriptstyle\cdots$};
        \draw (1.75,0) -- (1.25,2);
        \node[below] at (2,0){$\ \scriptstyle k+1$};
        \draw (1.25,0) -- (1.75,2);
        \node[below] at (1,0){$\scriptstyle k$};
        \node at (2,1) {$\scriptstyle\cdots$};
        \draw (2.5,0) -- (2.5,2);
    \end{tikzpicture}
    }} \Big)
    &= \Phi \left(
    \vcenter{\hbox{
    \begin{tikzpicture}[scale=0.3]
        \draw[style=su2 puncture strand] (3,0) -- (3,2);
        \draw[style=su2 puncture strand] (0,0) -- (0,2);
        \draw[in=-90, out=90, looseness=1] (.5,0) to (-2.5,1);
        \draw (-2.5,1) -- (-2.5,2);
        \draw (.75,0) -- (.75,.25);
        \draw[in=-90, out=90, looseness=1] (.75,.25) to (-2.25,1.25);
        \draw (-2.25,1.25) -- (-2.25,2);
        \node at (1.5,.5) {$\scriptstyle\cdots$};
        \node at (-1.5,1.5) {$\scriptstyle\cdots$};
        \draw (2.25,0) -- (2.25,.75);
        \draw[in=-90, out=90, looseness=1] (2.25,.75) to (-.75,1.75);
        \draw (-.75,1.75) -- (-.75,2);
        \draw (2.5,0) -- (2.5,1);
        \draw[in=-90, out=90, looseness=1] (2.5,1) to (-.5,2);
    \end{tikzpicture}
    }} \right) \cdot \Phi \Big(
    \vcenter{ \vspace{1em} \hbox{
    \begin{tikzpicture}[scale=0.3]
        \draw[style=su2 puncture strand] (3,0) -- (3,2);
        \draw[style=su2 puncture strand] (0,0) -- (0,2);
        \draw (.5,0) -- (.5,2);
        \node at (1,1) {$\scriptstyle\cdots$};
        \draw (1.75,0) -- (1.25,2);
        \node[below] at (2,0){$\ \scriptstyle k+1$};
        \draw (1.25,0) -- (1.75,2);
        \node[below] at (1,0){$\scriptstyle k$};
        \node at (2,1) {$\scriptstyle\cdots$};
        \draw (2.5,0) -- (2.5,2);
    \end{tikzpicture}
    }} \Big) \\
    &=
    \vcenter{\hbox{
    \begin{tikzpicture}[scale=0.3]
        \draw[style=su2 puncture strand] (3,0) -- (3,2);
        \draw[style=su2 puncture strand] (0,0) -- (0,2);
        \draw[in=-90, out=90, looseness=1] (.5,0) to (-2.5,1);
        \draw (-2.5,1) -- (-2.5,2);
        \draw (.75,0) -- (.75,.25);
        \draw[in=-90, out=90, looseness=1] (.75,.25) to (-2.25,1.25);
        \draw (-2.25,1.25) -- (-2.25,2);
        \node at (1.5,.5) {$\scriptstyle\cdots$};
        \node at (-1.5,1.5) {$\scriptstyle\cdots$};
        \draw (2.25,0) -- (2.25,.75);
        \draw[in=-90, out=90, looseness=1] (2.25,.75) to (-.75,1.75);
        \draw (-.75,1.75) -- (-.75,2);
        \draw (2.5,0) -- (2.5,1);
        \draw[in=-90, out=90, looseness=1] (2.5,1) to (-.5,2);
    \end{tikzpicture}
    }} \cdot \Phi \Big(
    \vcenter{ \vspace{1em} \hbox{
    \begin{tikzpicture}[scale=0.3]
        \draw[style=su2 puncture strand] (3,0) -- (3,2);
        \draw[style=su2 puncture strand] (0,0) -- (0,2);
        \draw (.5,0) -- (.5,2);
        \node at (1,1) {$\scriptstyle\cdots$};
        \draw (1.75,0) -- (1.25,2);
        \node[below] at (2,0){$\ \scriptstyle k+1$};
        \draw (1.25,0) -- (1.75,2);
        \node[below] at (1,0){$\scriptstyle k$};
        \node at (2,1) {$\scriptstyle\cdots$};
        \draw (2.5,0) -- (2.5,2);
    \end{tikzpicture}
    }} \Big)
    \end{align}
    On the other hand, we can commute the two terms to get
    \begin{align}
    \Phi \Big(
    \vcenter{\hbox{
    \begin{tikzpicture}[scale=0.3]
        \draw[style=su2 puncture strand](3,0) -- (3,2);
        \draw[style=su2 puncture strand] (0,0) -- (0,2);
        \draw[in=-90, out=90, looseness=1] (.5,0) to (-2.5,1);
        \draw (-2.5,1) -- (-2.5,2);
        \draw (.75,0) -- (.75,.25);
        \draw[in=-90, out=90, looseness=1] (.75,.25) to (-2.25,1.25);
        \draw (-2.25,1.25) -- (-2.25,2);
        \node at (1.5,.5) {$\scriptstyle\cdots$};
        \node at (-1.5,1.5) {$\scriptstyle\cdots$};
        \draw (2.25,0) -- (2.25,.75);
        \draw[in=-90, out=90, looseness=1] (2.25,.75) to (-.75,1.75);
        \draw (-.75,1.75) -- (-.75,2);
        \draw (2.5,0) -- (2.5,1);
        \draw[in=-90, out=90, looseness=1] (2.5,1) to (-.5,2);
    \end{tikzpicture}
    }} \cdot
    \vcenter{ \vspace{1em} \hbox{
    \begin{tikzpicture}[scale=0.3]
        \draw[style=su2 puncture strand] (3,0) -- (3,2);
        \draw[style=su2 puncture strand] (0,0) -- (0,2);
        \draw (.5,0) -- (.5,2);
        \node at (1,1) {$\scriptstyle\cdots$};
        \draw (1.75,0) -- (1.25,2);
        \node[below] at (2,0){$\ \scriptstyle k+1$};
        \draw (1.25,0) -- (1.75,2);
        \node[below] at (1,0){$\scriptstyle k$};
        \node at (2,1) {$\scriptstyle\cdots$};
        \draw (2.5,0) -- (2.5,2);
    \end{tikzpicture}
    }} \Big)
    &= \Phi \Big(
    \vcenter{ \vspace{1em} \hbox{
    \begin{tikzpicture}[scale=0.3]
        \draw[style=su2 puncture strand] (3,0) -- (3,2);
        \draw[style=su2 puncture strand] (3.5,0) -- (3.5,2);
        \draw (.5,0) -- (.5,2);
        \node at (1,1) {$\scriptstyle\cdots$};
        \draw (1.75,0) -- (1.25,2);
        \node[below] at (2,0){$\ \scriptstyle k+1$};
        \draw (1.25,0) -- (1.75,2);
        \node[below] at (1,0){$\scriptstyle k$};
        \node at (2,1) {$\scriptstyle\cdots$};
        \draw (2.5,0) -- (2.5,2);
    \end{tikzpicture}
    }} \cdot
    \vcenter{\hbox{
    \begin{tikzpicture}[scale=0.3]
        \draw[style=su2 puncture strand] (3,0) -- (3,2);
        \draw[style=su2 puncture strand] (0,0) -- (0,2);
        \draw[in=-90, out=90, looseness=1] (.5,0) to (-2.5,1);
        \draw (-2.5,1) -- (-2.5,2);
        \draw (.75,0) -- (.75,.25);
        \draw[in=-90, out=90, looseness=1] (.75,.25) to (-2.25,1.25);
        \draw (-2.25,1.25) -- (-2.25,2);
        \node at (1.5,.5) {$\scriptstyle\cdots$};
        \node at (-1.5,1.5) {$\scriptstyle\cdots$};
        \draw (2.25,0) -- (2.25,.75);
        \draw[in=-90, out=90, looseness=1] (2.25,.75) to (-.75,1.75);
        \draw (-.75,1.75) -- (-.75,2);
        \draw (2.5,0) -- (2.5,1);
        \draw[in=-90, out=90, looseness=1] (2.5,1) to (-.5,2);
    \end{tikzpicture}
    }} \Big) \\
    &=
    \vcenter{ \vspace{1em} \hbox{
    \begin{tikzpicture}[scale=0.3]
        \draw[style=su2 puncture strand] (3,0) -- (3,2);
        \draw[style=su2 puncture strand] (3.5,0) -- (3.5,2);
        \draw (.5,0) -- (.5,2);
        \node at (1,1) {$\scriptstyle\cdots$};
        \draw (1.75,0) -- (1.25,2);
        \node[below] at (2,0){$\ \scriptstyle k+1$};
        \draw (1.25,0) -- (1.75,2);
        \node[below] at (1,0){$\scriptstyle k$};
        \node at (2,1) {$\scriptstyle\cdots$};
        \draw (2.5,0) -- (2.5,2);
    \end{tikzpicture}
    }} \cdot
    \vcenter{\hbox{
    \begin{tikzpicture}[scale=0.3]
        \draw[style=su2 puncture strand] (3,0) -- (3,2);
        \draw[style=su2 puncture strand] (0,0) -- (0,2);
        \draw[in=-90, out=90, looseness=1] (.5,0) to (-2.5,1);
        \draw (-2.5,1) -- (-2.5,2);
        \draw (.75,0) -- (.75,.25);
        \draw[in=-90, out=90, looseness=1] (.75,.25) to (-2.25,1.25);
        \draw (-2.25,1.25) -- (-2.25,2);
        \node at (1.5,.5) {$\scriptstyle\cdots$};
        \node at (-1.5,1.5) {$\scriptstyle\cdots$};
        \draw (2.25,0) -- (2.25,.75);
        \draw[in=-90, out=90, looseness=1] (2.25,.75) to (-.75,1.75);
        \draw (-.75,1.75) -- (-.75,2);
        \draw (2.5,0) -- (2.5,1);
        \draw[in=-90, out=90, looseness=1] (2.5,1) to (-.5,2);
    \end{tikzpicture}
    }} \\
    &= \vcenter{\hbox{
    \begin{tikzpicture}[scale=0.3]
        \draw[style=su2 puncture strand] (3,0) -- (3,2);
        \draw[style=su2 puncture strand] (0,0) -- (0,2);
        \draw[in=-90, out=90, looseness=1] (.5,0) to (-2.5,1);
        \draw (-2.5,1) -- (-2.5,2);
        \draw (.75,0) -- (.75,.25);
        \draw[in=-90, out=90, looseness=1] (.75,.25) to (-2.25,1.25);
        \draw (-2.25,1.25) -- (-2.25,2);
        \node at (1.5,.5) {$\scriptstyle\cdots$};
        \node at (-1.5,1.5) {$\scriptstyle\cdots$};
        \draw (2.25,0) -- (2.25,.75);
        \draw[in=-90, out=90, looseness=1] (2.25,.75) to (-.75,1.75);
        \draw (-.75,1.75) -- (-.75,2);
        \draw (2.5,0) -- (2.5,1);
        \draw[in=-90, out=90, looseness=1] (2.5,1) to (-.5,2);
    \end{tikzpicture}
    }} \cdot
    \vcenter{ \vspace{1em} \hbox{
    \begin{tikzpicture}[scale=0.3]
        \draw[style=su2 puncture strand] (3,0) -- (3,2);
        \draw[style=su2 puncture strand] (0,0) -- (0,2);
        \draw (.5,0) -- (.5,2);
        \node at (1,1) {$\scriptstyle\cdots$};
        \draw (1.75,0) -- (1.25,2);
        \node[below] at (2,0){$\ \scriptstyle k+1$};
        \draw (1.25,0) -- (1.75,2);
        \node[below] at (1,0){$\scriptstyle k$};
        \node at (2,1) {$\scriptstyle\cdots$};
        \draw (2.5,0) -- (2.5,2);
    \end{tikzpicture}
    }}
    \end{align}
    The only solution is
    \begin{equation}
    \Phi \Big(
    \vcenter{ \vspace{1em} \hbox{
    \begin{tikzpicture}[scale=0.3]
        \draw[style=su2 puncture strand] (3,0) -- (3,2);
        \draw[style=su2 puncture strand] (0,0) -- (0,2);
        \draw (.5,0) -- (.5,2);
        \node at (1,1) {$\scriptstyle\cdots$};
        \draw (1.75,0) -- (1.25,2);
        \node[below] at (2,0){$\ \scriptstyle k+1$};
        \draw (1.25,0) -- (1.75,2);
        \node[below] at (1,0){$\scriptstyle k$};
        \node at (2,1) {$\scriptstyle\cdots$};
        \draw (2.5,0) -- (2.5,2);
    \end{tikzpicture}
    }} \Big) =
    \vcenter{ \vspace{1em} \hbox{
    \begin{tikzpicture}[scale=0.3]
        \draw[style=su2 puncture strand] (3,0) -- (3,2);
        \draw[style=su2 puncture strand] (0,0) -- (0,2);
        \draw (.5,0) -- (.5,2);
        \node at (1,1) {$\scriptstyle\cdots$};
        \draw (1.75,0) -- (1.25,2);
        \node[below] at (2,0){$\ \scriptstyle k+1$};
        \draw (1.25,0) -- (1.75,2);
        \node[below] at (1,0){$\scriptstyle k$};
        \node at (2,1) {$\scriptstyle\cdots$};
        \draw (2.5,0) -- (2.5,2);
    \end{tikzpicture}
    }}
    \end{equation}
    A similar calculation shows that  dot morphisms are preserved by $\Phi$,  completing the proof.
\end{proof}

\begin{remark}
    Without assuming condition (2) above, one can instead take $\epsilon(\theta) = \epsilon_L^l \epsilon_R^r$ where $l$ is the number of black strands to the left of $\pi$ and right of the cylindrical marking point  $\phi$, similarly $r$ is the number of black strands to the right of $\pi$ and left of $\phi$. Then $\epsilon(\theta) \cdot \eta_\theta$ is either a natural isomorphism to the identity, or to the functor $(-1)^C$, which is the identity on objects and $(-1)^C$ on morphisms, where $C$ is the number of times a strand crosses the marked point $\phi$.  
\end{remark}

\begin{remark}
    The comparison between a Hom space and Khovanov homology happens after setting $u = \hbar = 1$, so the $u, \hbar$ gradings are not retained by the Khovanov homology. In fact, if one sets these parameters to zero, the resulting Hom space does not give a knot invariant. 
\end{remark}

\section{2 red 1 black KLRW calculations} \label{2 red 1 black}

Any  $\mathcal{C}_{\Gamma, 0, F}$ (no black strands) is a category with one element, with endomorphisms $\Z$.  
Consider  $\mathbb{R}_\pi \in \mathcal{C}_{\bullet, 1, F} - \module - \mathcal{C}_{\bullet, 0, F \setminus \pi} =
\mathcal{C}_{\bullet, 1, F} - \module$.
To emphasize that we consider it as a module, let us choose some $\theta$ in the interval between the two red points of $\pi$, and write $S_\theta:= \mathbb{R}_\pi$.  Note that $\theta$ also determines an element of $\mathcal{C}_{\bullet, 1, F}$.  We also write $\theta_+$ and $\theta_-$ for points immediately to the right and left of the interval between the pair $\pi$.   

Webster gives the following resolution \cite[\S 4.3]{webster2016tensor}: 

\begin{equation}\label{eqn:simple resolution}
    \theta\{-2\}
    \xrightarrow{\begin{pmatrix} -\vcenter{\hbox{\begin{tikzpicture}[scale=.5]
	\begin{pgfonlayer}{nodelayer}
		\node [style=none] (0) at (-0.25, 0.5) {};
		\node [style=none] (1) at (0.25, 0.5) {};
		\node [style=none] (2) at (-0.25, -0.25) {};
		\node [style=none] (3) at (0.25, -0.25) {};
		\node [style=none] (4) at (0.5, -0.25) {};
		\node [style=none] (5) at (0, 0.5) {};
	\end{pgfonlayer}
	\begin{pgfonlayer}{edgelayer}
		\draw [style=su2 puncture strand] (0.center) to (2.center);
		\draw [style=su2 puncture strand] (3.center) to (1.center);
		\draw [style=brane, in=90, out=-90, looseness=1.25] (5.center) to (4.center);
	\end{pgfonlayer}
\end{tikzpicture}}} \\ \vcenter{\hbox{\begin{tikzpicture}[scale=.5]
	\begin{pgfonlayer}{nodelayer}
		\node [style=none] (0) at (-0.25, 0.5) {};
		\node [style=none] (1) at (0.25, 0.5) {};
		\node [style=none] (2) at (-0.25, -0.25) {};
		\node [style=none] (3) at (0.25, -0.25) {};
		\node [style=none] (4) at (-0.5, -0.25) {};
		\node [style=none] (5) at (0, 0.5) {};
	\end{pgfonlayer}
	\begin{pgfonlayer}{edgelayer}
		\draw [style=su2 puncture strand] (0.center) to (2.center);
		\draw [style=su2 puncture strand] (3.center) to (1.center);
		\draw [style=brane, in=90, out=-90, looseness=1.25] (5.center) to (4.center);
	\end{pgfonlayer}
\end{tikzpicture}}} \end{pmatrix}}
    \begin{matrix} \theta_+\{-1\} \\ \oplus \\ \theta_-\{-1\} \end{matrix}
    \xrightarrow{\begin{pmatrix} \vcenter{\hbox{\begin{tikzpicture}[scale=.5]
	\begin{pgfonlayer}{nodelayer}
		\node [style=none] (0) at (-0.25, 0.5) {};
		\node [style=none] (1) at (0.25, 0.5) {};
		\node [style=none] (2) at (-0.25, -0.25) {};
		\node [style=none] (3) at (0.25, -0.25) {};
		\node [style=none] (4) at (0, -0.25) {};
		\node [style=none] (5) at (0.5, 0.5) {};
	\end{pgfonlayer}
	\begin{pgfonlayer}{edgelayer}
		\draw [style=su2 puncture strand] (0.center) to (2.center);
		\draw [style=su2 puncture strand] (3.center) to (1.center);
		\draw [style=brane, in=90, out=-90, looseness=1.25] (5.center) to (4.center);
	\end{pgfonlayer}
\end{tikzpicture}}} & \vcenter{\hbox{\begin{tikzpicture}[scale=.5]
	\begin{pgfonlayer}{nodelayer}
		\node [style=none] (0) at (-0.25, 0.5) {};
		\node [style=none] (1) at (0.25, 0.5) {};
		\node [style=none] (2) at (-0.25, -0.25) {};
		\node [style=none] (3) at (0.25, -0.25) {};
		\node [style=none] (4) at (0, -0.25) {};
		\node [style=none] (5) at (-0.5, 0.5) {};
	\end{pgfonlayer}
	\begin{pgfonlayer}{edgelayer}
		\draw [style=su2 puncture strand] (0.center) to (2.center);
		\draw [style=su2 puncture strand] (3.center) to (1.center);
		\draw [style=brane, in=90, out=-90, looseness=1.25] (5.center) to (4.center);
	\end{pgfonlayer}
\end{tikzpicture}}} \end{pmatrix}}
    \theta
    \rightarrow S_\theta
\end{equation}
where the differentials act by right multiplication. The number in braces is a J-degree shift.

We will now describe the action of $\mathbb{B}_{\tau(\pi)}$ on certain elements using the exact triangle \eqref{webster exact triangle}.

\begin{lemma}
\begin{equation}
    \mathbb{B}_{\tau(\pi)} \theta \cong
    \theta \{-2\}
    \xrightarrow{\begin{pmatrix} -\vcenter{\hbox{\begin{tikzpicture}[scale=.5]
	\begin{pgfonlayer}{nodelayer}
		\node [style=none] (0) at (-0.25, 0.5) {};
		\node [style=none] (1) at (0.25, 0.5) {};
		\node [style=none] (2) at (-0.25, -0.25) {};
		\node [style=none] (3) at (0.25, -0.25) {};
		\node [style=none] (4) at (0.5, -0.25) {};
		\node [style=none] (5) at (0, 0.5) {};
	\end{pgfonlayer}
	\begin{pgfonlayer}{edgelayer}
		\draw [style=su2 puncture strand] (0.center) to (2.center);
		\draw [style=su2 puncture strand] (3.center) to (1.center);
		\draw [style=brane, in=90, out=-90, looseness=1.25] (5.center) to (4.center);
	\end{pgfonlayer}
\end{tikzpicture}}} & \vcenter{\hbox{\begin{tikzpicture}[scale=.5]
	\begin{pgfonlayer}{nodelayer}
		\node [style=none] (0) at (-0.25, 0.5) {};
		\node [style=none] (1) at (0.25, 0.5) {};
		\node [style=none] (2) at (-0.25, -0.25) {};
		\node [style=none] (3) at (0.25, -0.25) {};
		\node [style=none] (4) at (-0.5, -0.25) {};
		\node [style=none] (5) at (0, 0.5) {};
	\end{pgfonlayer}
	\begin{pgfonlayer}{edgelayer}
		\draw [style=su2 puncture strand] (0.center) to (2.center);
		\draw [style=su2 puncture strand] (3.center) to (1.center);
		\draw [style=brane, in=90, out=-90, looseness=1.25] (5.center) to (4.center);
	\end{pgfonlayer}
\end{tikzpicture}}} \end{pmatrix}}
    \theta_+\{-1\} \oplus \theta_-\{-1\},
\end{equation}
\end{lemma}

\begin{proof}
We will give two different calculations. 

The first is via the exact triangle \eqref{webster exact triangle}, from which we see 
$$\mathbb{B}_{\tau(\pi)} \theta \cong \mathrm{cone}  \left(  \theta \xrightarrow{\vcenter{\hbox{\includegraphics[scale=.7]{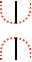}}}} \vcenter{\hbox{\includegraphics[scale=.7]{cupcap.pdf}}} \right)$$
Here, the same symbol names both a map a module; as explained above, the module is spanned by KLRW diagrams attached on the top modulo some relations; meanwhile the map is defined by attaching said picture to the bottom of the KLRW diagrams which are elements of $\theta$.

The cap at the bottom of \includegraphics[scale=.5]{cupcap.pdf} does not change anything, so we can just as well write 
\begin{equation}\label{eqn:braiding theta}
    \bigg\{ \theta \xrightarrow{\vcenter{\hbox{\includegraphics[scale=.7]{su2_simple.pdf}}}} \vcenter{\hbox{\includegraphics[scale=.7]{su2_simple.pdf}}} \bigg\} \cong \bigg\{ \theta \xrightarrow{\vcenter{\hbox{\includegraphics[scale=.7]{su2_simple.pdf}}}} S_\theta \bigg\}.
\end{equation}
Resolving $S_\theta$ gives
\begin{equation}\label{eqn:braiding theta resolved}
    \theta\{-2\}
    \xrightarrow{\begin{pmatrix} 0 & -\vcenter{\hbox{\begin{tikzpicture}[scale=.5]
	\begin{pgfonlayer}{nodelayer}
		\node [style=none] (0) at (-0.25, 0.5) {};
		\node [style=none] (1) at (0.25, 0.5) {};
		\node [style=none] (2) at (-0.25, -0.25) {};
		\node [style=none] (3) at (0.25, -0.25) {};
		\node [style=none] (4) at (0.5, -0.25) {};
		\node [style=none] (5) at (0, 0.5) {};
	\end{pgfonlayer}
	\begin{pgfonlayer}{edgelayer}
		\draw [style=su2 puncture strand] (0.center) to (2.center);
		\draw [style=su2 puncture strand] (3.center) to (1.center);
		\draw [style=brane, in=90, out=-90, looseness=1.25] (5.center) to (4.center);
	\end{pgfonlayer}
\end{tikzpicture}}} & \vcenter{\hbox{\begin{tikzpicture}[scale=.5]
	\begin{pgfonlayer}{nodelayer}
		\node [style=none] (0) at (-0.25, 0.5) {};
		\node [style=none] (1) at (0.25, 0.5) {};
		\node [style=none] (2) at (-0.25, -0.25) {};
		\node [style=none] (3) at (0.25, -0.25) {};
		\node [style=none] (4) at (-0.5, -0.25) {};
		\node [style=none] (5) at (0, 0.5) {};
	\end{pgfonlayer}
	\begin{pgfonlayer}{edgelayer}
		\draw [style=su2 puncture strand] (0.center) to (2.center);
		\draw [style=su2 puncture strand] (3.center) to (1.center);
		\draw [style=brane, in=90, out=-90, looseness=1.25] (5.center) to (4.center);
	\end{pgfonlayer}
\end{tikzpicture}}} \end{pmatrix}}
    \theta \oplus \theta_+\{-1\} \oplus \theta_-\{-1\}
    \xrightarrow{\begin{pmatrix} \vcenter{\hbox{\begin{tikzpicture}[scale=.5]
	\begin{pgfonlayer}{nodelayer}
		\node [style=none] (0) at (-0.25, 0.5) {};
		\node [style=none] (1) at (0.25, 0.5) {};
		\node [style=none] (2) at (-0.25, -0.25) {};
		\node [style=none] (3) at (0.25, -0.25) {};
		\node [style=none] (4) at (0, -0.25) {};
		\node [style=none] (5) at (0, 0.5) {};
	\end{pgfonlayer}
	\begin{pgfonlayer}{edgelayer}
		\draw [style=su2 puncture strand] (0.center) to (2.center);
		\draw [style=su2 puncture strand] (3.center) to (1.center);
		\draw [style=brane, in=90, out=-90, looseness=1.25] (5.center) to (4.center);
	\end{pgfonlayer}
\end{tikzpicture}}} \\ \vcenter{\hbox{\begin{tikzpicture}[scale=.5]
	\begin{pgfonlayer}{nodelayer}
		\node [style=none] (0) at (-0.25, 0.5) {};
		\node [style=none] (1) at (0.25, 0.5) {};
		\node [style=none] (2) at (-0.25, -0.25) {};
		\node [style=none] (3) at (0.25, -0.25) {};
		\node [style=none] (4) at (0, -0.25) {};
		\node [style=none] (5) at (0.5, 0.5) {};
	\end{pgfonlayer}
	\begin{pgfonlayer}{edgelayer}
		\draw [style=su2 puncture strand] (0.center) to (2.center);
		\draw [style=su2 puncture strand] (3.center) to (1.center);
		\draw [style=brane, in=90, out=-90, looseness=1.25] (5.center) to (4.center);
	\end{pgfonlayer}
\end{tikzpicture}}} \\ \vcenter{\hbox{\begin{tikzpicture}[scale=.5]
	\begin{pgfonlayer}{nodelayer}
		\node [style=none] (0) at (-0.25, 0.5) {};
		\node [style=none] (1) at (0.25, 0.5) {};
		\node [style=none] (2) at (-0.25, -0.25) {};
		\node [style=none] (3) at (0.25, -0.25) {};
		\node [style=none] (4) at (0, -0.25) {};
		\node [style=none] (5) at (-0.5, 0.5) {};
	\end{pgfonlayer}
	\begin{pgfonlayer}{edgelayer}
		\draw [style=su2 puncture strand] (0.center) to (2.center);
		\draw [style=su2 puncture strand] (3.center) to (1.center);
		\draw [style=brane, in=90, out=-90, looseness=1.25] (5.center) to (4.center);
	\end{pgfonlayer}
\end{tikzpicture}}} \end{pmatrix}}
    \theta
\end{equation}

The second way is to resolve $\mathbb{B}_{\tau(\pi)}(\theta)$ directly using its diagrammatic definition.  
The module $\mathbb{B}_{\tau(\pi)}(\theta)$ contains all diagrams where the two red strands cross once and the ordering of the strands at the bottom of the diagram is $\vcenter{\hbox{\begin{tikzpicture}[scale=.5]
	\begin{pgfonlayer}{nodelayer}
		\node [style=none] (0) at (-0.25, 0.5) {};
		\node [style=none] (1) at (0.25, 0.5) {};
		\node [style=none] (2) at (-0.25, -0.25) {};
		\node [style=none] (3) at (0.25, -0.25) {};
		\node [style=none] (4) at (0, -0.25) {};
		\node [style=none] (5) at (0, 0.5) {};
	\end{pgfonlayer}
	\begin{pgfonlayer}{edgelayer}
		\draw [style=su2 puncture strand] (0.center) to (2.center);
		\draw [style=su2 puncture strand] (3.center) to (1.center);
		\draw [style=brane, in=90, out=-90, looseness=1.25] (5.center) to (4.center);
	\end{pgfonlayer}
\end{tikzpicture}}}$. We can push this crossing of the red strands to the bottom of the diagram. Then, it is easy to see that the image of the map $\begin{pmatrix} \vcenter{\hbox{\begin{tikzpicture}[scale=.5]
	\begin{pgfonlayer}{nodelayer}
		\node [style=none] (0) at (-0.25, 0.5) {};
		\node [style=none] (1) at (0.25, 0.5) {};
		\node [style=none] (2) at (0.25, -0.25) {};
		\node [style=none] (3) at (-0.25, -0.25) {};
		\node [style=none] (4) at (0, -0.25) {};
		\node [style=none] (5) at (0.5, 0.5) {};
	\end{pgfonlayer}
	\begin{pgfonlayer}{edgelayer}
		\draw [style=su2 puncture strand] (0.center) to (2.center);
		\draw [style=su2 puncture strand] (3.center) to (1.center);
		\draw [style=brane] (5.center) to (4.center);
	\end{pgfonlayer}
\end{tikzpicture}}} & \vcenter{\hbox{\begin{tikzpicture}[scale=.5]
	\begin{pgfonlayer}{nodelayer}
		\node [style=none] (0) at (-0.25, 0.5) {};
		\node [style=none] (1) at (0.25, 0.5) {};
		\node [style=none] (2) at (0.25, -0.25) {};
		\node [style=none] (3) at (-0.25, -0.25) {};
		\node [style=none] (4) at (0, -0.25) {};
		\node [style=none] (5) at (-0.5, 0.5) {};
	\end{pgfonlayer}
	\begin{pgfonlayer}{edgelayer}
		\draw [style=su2 puncture strand] (0.center) to (2.center);
		\draw [style=su2 puncture strand] (3.center) to (1.center);
		\draw [style=brane] (5.center) to (4.center);
	\end{pgfonlayer}
\end{tikzpicture}}} \end{pmatrix}^{\intercal}$ spans $\mathbb{B}_{\tau(\pi)}(\theta)$. The kernel of this map comes from the first relation is Figure \ref{fig:KLRW-braiding} and hence lies in the image of $\begin{pmatrix} -\vcenter{\hbox{\begin{tikzpicture}[scale=.5]
	\begin{pgfonlayer}{nodelayer}
		\node [style=none] (0) at (-0.25, 0.5) {};
		\node [style=none] (1) at (0.25, 0.5) {};
		\node [style=none] (2) at (-0.25, -0.25) {};
		\node [style=none] (3) at (0.25, -0.25) {};
		\node [style=none] (4) at (0.5, -0.25) {};
		\node [style=none] (5) at (0, 0.5) {};
	\end{pgfonlayer}
	\begin{pgfonlayer}{edgelayer}
		\draw [style=su2 puncture strand] (0.center) to (2.center);
		\draw [style=su2 puncture strand] (3.center) to (1.center);
		\draw [style=brane, in=90, out=-90, looseness=1.25] (5.center) to (4.center);
	\end{pgfonlayer}
\end{tikzpicture}}} & \vcenter{\hbox{\begin{tikzpicture}[scale=.5]
	\begin{pgfonlayer}{nodelayer}
		\node [style=none] (0) at (-0.25, 0.5) {};
		\node [style=none] (1) at (0.25, 0.5) {};
		\node [style=none] (2) at (-0.25, -0.25) {};
		\node [style=none] (3) at (0.25, -0.25) {};
		\node [style=none] (4) at (-0.5, -0.25) {};
		\node [style=none] (5) at (0, 0.5) {};
	\end{pgfonlayer}
	\begin{pgfonlayer}{edgelayer}
		\draw [style=su2 puncture strand] (0.center) to (2.center);
		\draw [style=su2 puncture strand] (3.center) to (1.center);
		\draw [style=brane, in=90, out=-90, looseness=1.25] (5.center) to (4.center);
	\end{pgfonlayer}
\end{tikzpicture}}} \end{pmatrix}$. Since $\begin{pmatrix} -\vcenter{\hbox{\begin{tikzpicture}[scale=.5]
	\begin{pgfonlayer}{nodelayer}
		\node [style=none] (0) at (-0.25, 0.5) {};
		\node [style=none] (1) at (0.25, 0.5) {};
		\node [style=none] (2) at (-0.25, -0.25) {};
		\node [style=none] (3) at (0.25, -0.25) {};
		\node [style=none] (4) at (0.5, -0.25) {};
		\node [style=none] (5) at (0, 0.5) {};
	\end{pgfonlayer}
	\begin{pgfonlayer}{edgelayer}
		\draw [style=su2 puncture strand] (0.center) to (2.center);
		\draw [style=su2 puncture strand] (3.center) to (1.center);
		\draw [style=brane, in=90, out=-90, looseness=1.25] (5.center) to (4.center);
	\end{pgfonlayer}
\end{tikzpicture}}} & \vcenter{\hbox{\begin{tikzpicture}[scale=.5]
	\begin{pgfonlayer}{nodelayer}
		\node [style=none] (0) at (-0.25, 0.5) {};
		\node [style=none] (1) at (0.25, 0.5) {};
		\node [style=none] (2) at (-0.25, -0.25) {};
		\node [style=none] (3) at (0.25, -0.25) {};
		\node [style=none] (4) at (-0.5, -0.25) {};
		\node [style=none] (5) at (0, 0.5) {};
	\end{pgfonlayer}
	\begin{pgfonlayer}{edgelayer}
		\draw [style=su2 puncture strand] (0.center) to (2.center);
		\draw [style=su2 puncture strand] (3.center) to (1.center);
		\draw [style=brane, in=90, out=-90, looseness=1.25] (5.center) to (4.center);
	\end{pgfonlayer}
\end{tikzpicture}}} \end{pmatrix}$ has only zero in its kernel, we arrive at the resolution:

\begin{equation}
    \theta \{-2\}
    \xrightarrow{\begin{pmatrix} -\vcenter{\hbox{\begin{tikzpicture}[scale=.5]
	\begin{pgfonlayer}{nodelayer}
		\node [style=none] (0) at (-0.25, 0.5) {};
		\node [style=none] (1) at (0.25, 0.5) {};
		\node [style=none] (2) at (-0.25, -0.25) {};
		\node [style=none] (3) at (0.25, -0.25) {};
		\node [style=none] (4) at (0.5, -0.25) {};
		\node [style=none] (5) at (0, 0.5) {};
	\end{pgfonlayer}
	\begin{pgfonlayer}{edgelayer}
		\draw [style=su2 puncture strand] (0.center) to (2.center);
		\draw [style=su2 puncture strand] (3.center) to (1.center);
		\draw [style=brane, in=90, out=-90, looseness=1.25] (5.center) to (4.center);
	\end{pgfonlayer}
\end{tikzpicture}}} & \vcenter{\hbox{\begin{tikzpicture}[scale=.5]
	\begin{pgfonlayer}{nodelayer}
		\node [style=none] (0) at (-0.25, 0.5) {};
		\node [style=none] (1) at (0.25, 0.5) {};
		\node [style=none] (2) at (-0.25, -0.25) {};
		\node [style=none] (3) at (0.25, -0.25) {};
		\node [style=none] (4) at (-0.5, -0.25) {};
		\node [style=none] (5) at (0, 0.5) {};
	\end{pgfonlayer}
	\begin{pgfonlayer}{edgelayer}
		\draw [style=su2 puncture strand] (0.center) to (2.center);
		\draw [style=su2 puncture strand] (3.center) to (1.center);
		\draw [style=brane, in=90, out=-90, looseness=1.25] (5.center) to (4.center);
	\end{pgfonlayer}
\end{tikzpicture}}} \end{pmatrix}}
    \theta_+\{-1\} \oplus \theta_-\{-1\}
    \xrightarrow{\begin{pmatrix} \vcenter{\hbox{\begin{tikzpicture}[scale=.5]
	\begin{pgfonlayer}{nodelayer}
		\node [style=none] (0) at (-0.25, 0.5) {};
		\node [style=none] (1) at (0.25, 0.5) {};
		\node [style=none] (2) at (0.25, -0.25) {};
		\node [style=none] (3) at (-0.25, -0.25) {};
		\node [style=none] (4) at (0, -0.25) {};
		\node [style=none] (5) at (0.5, 0.5) {};
	\end{pgfonlayer}
	\begin{pgfonlayer}{edgelayer}
		\draw [style=su2 puncture strand] (0.center) to (2.center);
		\draw [style=su2 puncture strand] (3.center) to (1.center);
		\draw [style=brane] (5.center) to (4.center);
	\end{pgfonlayer}
\end{tikzpicture}}} \\ \vcenter{\hbox{\begin{tikzpicture}[scale=.5]
	\begin{pgfonlayer}{nodelayer}
		\node [style=none] (0) at (-0.25, 0.5) {};
		\node [style=none] (1) at (0.25, 0.5) {};
		\node [style=none] (2) at (0.25, -0.25) {};
		\node [style=none] (3) at (-0.25, -0.25) {};
		\node [style=none] (4) at (0, -0.25) {};
		\node [style=none] (5) at (-0.5, 0.5) {};
	\end{pgfonlayer}
	\begin{pgfonlayer}{edgelayer}
		\draw [style=su2 puncture strand] (0.center) to (2.center);
		\draw [style=su2 puncture strand] (3.center) to (1.center);
		\draw [style=brane] (5.center) to (4.center);
	\end{pgfonlayer}
\end{tikzpicture}}} \end{pmatrix}}
    \mathbb{B}_{\tau(\pi)}(\theta)
\end{equation}
\end{proof}

We will later be interested in the action of braiding on complexes like the following: 
\begin{equation}\label{eqn:Uminus}
    \upsilon_- := \big\{ \theta\{-1\} \xrightarrow{\begin{tikzpicture}[scale=.5]
	\begin{pgfonlayer}{nodelayer}
		\node [style=none] (0) at (-0.25, 0.5) {};
		\node [style=none] (1) at (0.25, 0.5) {};
		\node [style=none] (2) at (-0.25, -0.25) {};
		\node [style=none] (3) at (0.25, -0.25) {};
		\node [style=none] (4) at (-0.5, -0.25) {};
		\node [style=none] (5) at (0, 0.5) {};
	\end{pgfonlayer}
	\begin{pgfonlayer}{edgelayer}
		\draw [style=su2 puncture strand] (0.center) to (2.center);
		\draw [style=su2 puncture strand] (3.center) to (1.center);
		\draw [style=brane, in=90, out=-90, looseness=1.25] (5.center) to (4.center);
	\end{pgfonlayer}
\end{tikzpicture}} \theta_- \big\}
\end{equation}
and
\begin{equation}\label{eqn:Uplus}
    \upsilon_+ := \big\{ \theta_+\{-1\} \xrightarrow{\begin{tikzpicture}[scale=.5]
	\begin{pgfonlayer}{nodelayer}
		\node [style=none] (0) at (-0.25, 0.5) {};
		\node [style=none] (1) at (0.25, 0.5) {};
		\node [style=none] (2) at (-0.25, -0.25) {};
		\node [style=none] (3) at (0.25, -0.25) {};
		\node [style=none] (4) at (0, -0.25) {};
		\node [style=none] (5) at (0.5, 0.5) {};
	\end{pgfonlayer}
	\begin{pgfonlayer}{edgelayer}
		\draw [style=su2 puncture strand] (0.center) to (2.center);
		\draw [style=su2 puncture strand] (3.center) to (1.center);
		\draw [style=brane, in=90, out=-90, looseness=1.25] (5.center) to (4.center);
	\end{pgfonlayer}
\end{tikzpicture}} \theta \big\}.
\end{equation}
The modules $\upsilon_+$ and $\upsilon_-$ are examples of Webster's standard modules, defined in \cite[Def. 5.1]{webster}.

\begin{proposition}\label{braiding U}
    Braiding sends $\upsilon_+$ to $\upsilon_-\{-1\}$.
\end{proposition}

\begin{proof}
    Braiding acts by the identity on $\theta_+$, sends $\theta$ to the complex in equation \ref{eqn:braiding theta}, and acts by the identity on the morphism $\begin{tikzpicture}[scale=.5]
	\begin{pgfonlayer}{nodelayer}
		\node [style=none] (0) at (-0.25, 0.5) {};
		\node [style=none] (1) at (0.25, 0.5) {};
		\node [style=none] (2) at (-0.25, -0.25) {};
		\node [style=none] (3) at (0.25, -0.25) {};
		\node [style=none] (4) at (0, -0.25) {};
		\node [style=none] (5) at (0.5, 0.5) {};
	\end{pgfonlayer}
	\begin{pgfonlayer}{edgelayer}
		\draw [style=su2 puncture strand] (0.center) to (2.center);
		\draw [style=su2 puncture strand] (3.center) to (1.center);
		\draw [style=brane, in=90, out=-90, looseness=1.25] (5.center) to (4.center);
	\end{pgfonlayer}
\end{tikzpicture}$:
    \begin{equation}
        \mathbb{B}_{\tau(\pi)}(\upsilon_+) \cong \bigg\{ \theta_+\{-1\} \xrightarrow{\begin{tikzpicture}[scale=.5]
	\begin{pgfonlayer}{nodelayer}
		\node [style=none] (0) at (-0.25, 0.5) {};
		\node [style=none] (1) at (0.25, 0.5) {};
		\node [style=none] (2) at (-0.25, -0.25) {};
		\node [style=none] (3) at (0.25, -0.25) {};
		\node [style=none] (4) at (0, -0.25) {};
		\node [style=none] (5) at (0.5, 0.5) {};
	\end{pgfonlayer}
	\begin{pgfonlayer}{edgelayer}
		\draw [style=su2 puncture strand] (0.center) to (2.center);
		\draw [style=su2 puncture strand] (3.center) to (1.center);
		\draw [style=brane, in=90, out=-90, looseness=1.25] (5.center) to (4.center);
	\end{pgfonlayer}
\end{tikzpicture}} \theta \xrightarrow{\vcenter{\hbox{\includegraphics[scale=.7]{su2_simple.pdf}}}} S_\theta \bigg\}
    \end{equation}
    After resolving $S_\theta$, the resulting complex is
    \begin{align}
        \mathbb{B}_{\tau(\pi)}(U_+) &\cong \bigg\{ \theta\{-2\} \oplus \theta_+\{-1\}
        \xrightarrow{\begin{pmatrix} 0 & -\vcenter{\hbox{\begin{tikzpicture}[scale=.5]
	\begin{pgfonlayer}{nodelayer}
		\node [style=none] (0) at (-0.25, 0.5) {};
		\node [style=none] (1) at (0.25, 0.5) {};
		\node [style=none] (2) at (-0.25, -0.25) {};
		\node [style=none] (3) at (0.25, -0.25) {};
		\node [style=none] (4) at (0.5, -0.25) {};
		\node [style=none] (5) at (0, 0.5) {};
	\end{pgfonlayer}
	\begin{pgfonlayer}{edgelayer}
		\draw [style=su2 puncture strand] (0.center) to (2.center);
		\draw [style=su2 puncture strand] (3.center) to (1.center);
		\draw [style=brane, in=90, out=-90, looseness=1.25] (5.center) to (4.center);
	\end{pgfonlayer}
\end{tikzpicture}}} & \vcenter{\hbox{\begin{tikzpicture}[scale=.5]
	\begin{pgfonlayer}{nodelayer}
		\node [style=none] (0) at (-0.25, 0.5) {};
		\node [style=none] (1) at (0.25, 0.5) {};
		\node [style=none] (2) at (-0.25, -0.25) {};
		\node [style=none] (3) at (0.25, -0.25) {};
		\node [style=none] (4) at (-0.5, -0.25) {};
		\node [style=none] (5) at (0, 0.5) {};
	\end{pgfonlayer}
	\begin{pgfonlayer}{edgelayer}
		\draw [style=su2 puncture strand] (0.center) to (2.center);
		\draw [style=su2 puncture strand] (3.center) to (1.center);
		\draw [style=brane, in=90, out=-90, looseness=1.25] (5.center) to (4.center);
	\end{pgfonlayer}
\end{tikzpicture}}} \\
        \vcenter{\hbox{\begin{tikzpicture}[scale=.5]
	\begin{pgfonlayer}{nodelayer}
		\node [style=none] (0) at (-0.25, 0.5) {};
		\node [style=none] (1) at (0.25, 0.5) {};
		\node [style=none] (2) at (-0.25, -0.25) {};
		\node [style=none] (3) at (0.25, -0.25) {};
		\node [style=none] (4) at (0, -0.25) {};
		\node [style=none] (5) at (0.5, 0.5) {};
	\end{pgfonlayer}
	\begin{pgfonlayer}{edgelayer}
		\draw [style=su2 puncture strand] (0.center) to (2.center);
		\draw [style=su2 puncture strand] (3.center) to (1.center);
		\draw [style=brane, in=90, out=-90, looseness=1.25] (5.center) to (4.center);
	\end{pgfonlayer}
\end{tikzpicture}}} & -\vcenter{\hbox{\begin{tikzpicture}[scale=.5]
	\begin{pgfonlayer}{nodelayer}
		\node [style=none] (0) at (-0.25, 0.5) {};
		\node [style=none] (1) at (0.25, 0.5) {};
		\node [style=none] (2) at (-0.25, -0.25) {};
		\node [style=none] (3) at (0.25, -0.25) {};
		\node [style=none] (4) at (0, -0.25) {};
		\node [style=none] (5) at (0, 0.5) {};
	\end{pgfonlayer}
	\begin{pgfonlayer}{edgelayer}
		\draw [style=su2 puncture strand] (0.center) to (2.center);
		\draw [style=brane] (3.center) to (1.center);
		\draw [style=su2 puncture strand, in=90, out=-90, looseness=1.25] (5.center) to (4.center);
	\end{pgfonlayer}
\end{tikzpicture}}} & 0 \end{pmatrix}}
        \theta \oplus \theta_+\{-1\} \oplus \theta_-
        \xrightarrow{\begin{pmatrix} \vcenter{\hbox{\begin{tikzpicture}[scale=.5]
	\begin{pgfonlayer}{nodelayer}
		\node [style=none] (0) at (-0.25, 0.5) {};
		\node [style=none] (1) at (0.25, 0.5) {};
		\node [style=none] (2) at (-0.25, -0.25) {};
		\node [style=none] (3) at (0.25, -0.25) {};
		\node [style=none] (4) at (0, -0.25) {};
		\node [style=none] (5) at (0, 0.5) {};
	\end{pgfonlayer}
	\begin{pgfonlayer}{edgelayer}
		\draw [style=su2 puncture strand] (0.center) to (2.center);
		\draw [style=su2 puncture strand] (3.center) to (1.center);
		\draw [style=brane, in=90, out=-90, looseness=1.25] (5.center) to (4.center);
	\end{pgfonlayer}
\end{tikzpicture}}} \\ \vcenter{\hbox{\begin{tikzpicture}[scale=.5]
	\begin{pgfonlayer}{nodelayer}
		\node [style=none] (0) at (-0.25, 0.5) {};
		\node [style=none] (1) at (0.25, 0.5) {};
		\node [style=none] (2) at (-0.25, -0.25) {};
		\node [style=none] (3) at (0.25, -0.25) {};
		\node [style=none] (4) at (0, -0.25) {};
		\node [style=none] (5) at (0.5, 0.5) {};
	\end{pgfonlayer}
	\begin{pgfonlayer}{edgelayer}
		\draw [style=su2 puncture strand] (0.center) to (2.center);
		\draw [style=su2 puncture strand] (3.center) to (1.center);
		\draw [style=brane, in=90, out=-90, looseness=1.25] (5.center) to (4.center);
	\end{pgfonlayer}
\end{tikzpicture}}} \\ \vcenter{\hbox{\begin{tikzpicture}[scale=.5]
	\begin{pgfonlayer}{nodelayer}
		\node [style=none] (0) at (-0.25, 0.5) {};
		\node [style=none] (1) at (0.25, 0.5) {};
		\node [style=none] (2) at (-0.25, -0.25) {};
		\node [style=none] (3) at (0.25, -0.25) {};
		\node [style=none] (4) at (0, -0.25) {};
		\node [style=none] (5) at (-0.5, 0.5) {};
	\end{pgfonlayer}
	\begin{pgfonlayer}{edgelayer}
		\draw [style=su2 puncture strand] (0.center) to (2.center);
		\draw [style=su2 puncture strand] (3.center) to (1.center);
		\draw [style=brane, in=90, out=-90, looseness=1.25] (5.center) to (4.center);
	\end{pgfonlayer}
\end{tikzpicture}}} \end{pmatrix}}
        \theta \bigg\} \\
        &\cong \bigg\{ \theta\{-2\} \oplus \theta_+\{-1\}
        \xrightarrow{\begin{pmatrix} -\vcenter{\hbox{\begin{tikzpicture}[scale=.5]
	\begin{pgfonlayer}{nodelayer}
		\node [style=none] (0) at (-0.25, 0.5) {};
		\node [style=none] (1) at (0.25, 0.5) {};
		\node [style=none] (2) at (-0.25, -0.25) {};
		\node [style=none] (3) at (0.25, -0.25) {};
		\node [style=none] (4) at (0.5, -0.25) {};
		\node [style=none] (5) at (0, 0.5) {};
	\end{pgfonlayer}
	\begin{pgfonlayer}{edgelayer}
		\draw [style=su2 puncture strand] (0.center) to (2.center);
		\draw [style=su2 puncture strand] (3.center) to (1.center);
		\draw [style=brane, in=90, out=-90, looseness=1.25] (5.center) to (4.center);
	\end{pgfonlayer}
\end{tikzpicture}}} & \vcenter{\hbox{\begin{tikzpicture}[scale=.5]
	\begin{pgfonlayer}{nodelayer}
		\node [style=none] (0) at (-0.25, 0.5) {};
		\node [style=none] (1) at (0.25, 0.5) {};
		\node [style=none] (2) at (-0.25, -0.25) {};
		\node [style=none] (3) at (0.25, -0.25) {};
		\node [style=none] (4) at (-0.5, -0.25) {};
		\node [style=none] (5) at (0, 0.5) {};
	\end{pgfonlayer}
	\begin{pgfonlayer}{edgelayer}
		\draw [style=su2 puncture strand] (0.center) to (2.center);
		\draw [style=su2 puncture strand] (3.center) to (1.center);
		\draw [style=brane, in=90, out=-90, looseness=1.25] (5.center) to (4.center);
	\end{pgfonlayer}
\end{tikzpicture}}} \\
        -\vcenter{\hbox{\begin{tikzpicture}[scale=.5]
	\begin{pgfonlayer}{nodelayer}
		\node [style=none] (0) at (-0.25, 0.5) {};
		\node [style=none] (1) at (0.25, 0.5) {};
		\node [style=none] (2) at (-0.25, -0.25) {};
		\node [style=none] (3) at (0.25, -0.25) {};
		\node [style=none] (4) at (0, -0.25) {};
		\node [style=none] (5) at (0, 0.5) {};
	\end{pgfonlayer}
	\begin{pgfonlayer}{edgelayer}
		\draw [style=su2 puncture strand] (0.center) to (2.center);
		\draw [style=brane] (3.center) to (1.center);
		\draw [style=su2 puncture strand, in=90, out=-90, looseness=1.25] (5.center) to (4.center);
	\end{pgfonlayer}
\end{tikzpicture}}} & 0 \end{pmatrix}}
        \theta_+\{-1\} \oplus \theta_-\{-1\} \bigg\} \\
        &\cong \bigg\{ \theta\{-2\}
        \xrightarrow{\begin{tikzpicture}[scale=.5]
	\begin{pgfonlayer}{nodelayer}
		\node [style=none] (0) at (-0.25, 0.5) {};
		\node [style=none] (1) at (0.25, 0.5) {};
		\node [style=none] (2) at (-0.25, -0.25) {};
		\node [style=none] (3) at (0.25, -0.25) {};
		\node [style=none] (4) at (-0.5, -0.25) {};
		\node [style=none] (5) at (0, 0.5) {};
	\end{pgfonlayer}
	\begin{pgfonlayer}{edgelayer}
		\draw [style=su2 puncture strand] (0.center) to (2.center);
		\draw [style=su2 puncture strand] (3.center) to (1.center);
		\draw [style=brane, in=90, out=-90, looseness=1.25] (5.center) to (4.center);
	\end{pgfonlayer}
\end{tikzpicture}}
        \theta_-\{-1\} \bigg\},
    \end{align}
 which is $\upsilon_-\{-1\}$.
\end{proof}

Next we want to consider maps to and from $\upsilon_-$ and $\upsilon_+$. When there is no ambiguity, we will denote these maps using strand diagrams and it is understood that we mean maps to and from the complexes \eqref{eqn:Uminus} and \eqref{eqn:Uplus}. For example, $\vcenter{\hbox{\begin{tikzpicture}[scale=.5]
	\begin{pgfonlayer}{nodelayer}
		\node [style=none] (0) at (-0.25, 0.5) {};
		\node [style=none] (1) at (0.25, 0.5) {};
		\node [style=none] (2) at (-0.25, -0.25) {};
		\node [style=none] (3) at (0.25, -0.25) {};
		\node [style=none] (4) at (0, -0.25) {};
		\node [style=none] (5) at (0, 0.5) {};
	\end{pgfonlayer}
	\begin{pgfonlayer}{edgelayer}
		\draw [style=su2 puncture strand] (0.center) to (2.center);
		\draw [style=brane] (3.center) to (1.center);
		\draw [style=su2 puncture strand, in=90, out=-90, looseness=1.25] (5.center) to (4.center);
	\end{pgfonlayer}
\end{tikzpicture}}}:\upsilon_+\to\theta_+\{-1\}$ denotes the chain map $f$ with components $f_0=0$ and $f_1 = \vcenter{\hbox{\begin{tikzpicture}[scale=.5]
	\begin{pgfonlayer}{nodelayer}
		\node [style=none] (0) at (-0.25, 0.5) {};
		\node [style=none] (1) at (0.25, 0.5) {};
		\node [style=none] (2) at (-0.25, -0.25) {};
		\node [style=none] (3) at (0.25, -0.25) {};
		\node [style=none] (4) at (0, -0.25) {};
		\node [style=none] (5) at (0, 0.5) {};
	\end{pgfonlayer}
	\begin{pgfonlayer}{edgelayer}
		\draw [style=su2 puncture strand] (0.center) to (2.center);
		\draw [style=brane] (3.center) to (1.center);
		\draw [style=su2 puncture strand, in=90, out=-90, looseness=1.25] (5.center) to (4.center);
	\end{pgfonlayer}
\end{tikzpicture}}}$.

\begin{proposition}\label{braiding U map}
    Braiding sends the map $\vcenter{\hbox{\begin{tikzpicture}[scale=.5]
	\begin{pgfonlayer}{nodelayer}
		\node [style=none] (0) at (-0.25, 0.5) {};
		\node [style=none] (1) at (0.25, 0.5) {};
		\node [style=none] (2) at (-0.25, -0.25) {};
		\node [style=none] (3) at (0.25, -0.25) {};
		\node [style=none] (4) at (0, -0.25) {};
		\node [style=none] (5) at (0, 0.5) {};
	\end{pgfonlayer}
	\begin{pgfonlayer}{edgelayer}
		\draw [style=su2 puncture strand] (0.center) to (2.center);
		\draw [style=brane] (3.center) to (1.center);
		\draw [style=su2 puncture strand, in=90, out=-90, looseness=1.25] (5.center) to (4.center);
	\end{pgfonlayer}
\end{tikzpicture}}}:\upsilon_+\to\theta_+\{-1\}$ to $\vcenter{\hbox{\begin{tikzpicture}[scale=.5]
	\begin{pgfonlayer}{nodelayer}
		\node [style=none] (0) at (-0.25, 0.5) {};
		\node [style=none] (1) at (0.25, 0.5) {};
		\node [style=none] (2) at (-0.25, -0.25) {};
		\node [style=none] (3) at (0.25, -0.25) {};
		\node [style=none] (4) at (0.5, -0.25) {};
		\node [style=none] (5) at (0, 0.5) {};
	\end{pgfonlayer}
	\begin{pgfonlayer}{edgelayer}
		\draw [style=su2 puncture strand] (0.center) to (2.center);
		\draw [style=su2 puncture strand] (3.center) to (1.center);
		\draw [style=brane, in=90, out=-90, looseness=1.25] (5.center) to (4.center);
	\end{pgfonlayer}
\end{tikzpicture}}}:\upsilon_-\{-1\}\to\theta_+\{-1\}$.
\end{proposition}

\begin{proof}
The cone of this map is
\begin{equation}
    \big\{ \upsilon_+ \xrightarrow{\begin{tikzpicture}[scale=.5]
	\begin{pgfonlayer}{nodelayer}
		\node [style=none] (0) at (-0.25, 0.5) {};
		\node [style=none] (1) at (0.25, 0.5) {};
		\node [style=none] (2) at (-0.25, -0.25) {};
		\node [style=none] (3) at (0.25, -0.25) {};
		\node [style=none] (4) at (0, -0.25) {};
		\node [style=none] (5) at (0, 0.5) {};
	\end{pgfonlayer}
	\begin{pgfonlayer}{edgelayer}
		\draw [style=su2 puncture strand] (0.center) to (2.center);
		\draw [style=brane] (3.center) to (1.center);
		\draw [style=su2 puncture strand, in=90, out=-90, looseness=1.25] (5.center) to (4.center);
	\end{pgfonlayer}
\end{tikzpicture}} \theta_+\{-1\} \big\} =
    \big\{ \theta_+\{-1\} \xrightarrow{\begin{pmatrix}
        \vcenter{\hbox{\begin{tikzpicture}[scale=.5]
	\begin{pgfonlayer}{nodelayer}
		\node [style=none] (0) at (-0.25, 0.5) {};
		\node [style=none] (1) at (0.25, 0.5) {};
		\node [style=none] (2) at (-0.25, -0.25) {};
		\node [style=none] (3) at (0.25, -0.25) {};
		\node [style=none] (4) at (0, -0.25) {};
		\node [style=none] (5) at (0.5, 0.5) {};
	\end{pgfonlayer}
	\begin{pgfonlayer}{edgelayer}
		\draw [style=su2 puncture strand] (0.center) to (2.center);
		\draw [style=su2 puncture strand] (3.center) to (1.center);
		\draw [style=brane, in=90, out=-90, looseness=1.25] (5.center) to (4.center);
	\end{pgfonlayer}
\end{tikzpicture}}} & \vcenter{\hbox{\begin{tikzpicture}[scale=.5]
	\begin{pgfonlayer}{nodelayer}
		\node [style=none] (0) at (-0.25, 0.5) {};
		\node [style=none] (1) at (0.25, 0.5) {};
		\node [style=none] (2) at (-0.25, -0.25) {};
		\node [style=none] (3) at (0.25, -0.25) {};
		\node [style=none] (4) at (0, -0.25) {};
		\node [style=none] (5) at (0, 0.5) {};
	\end{pgfonlayer}
	\begin{pgfonlayer}{edgelayer}
		\draw [style=su2 puncture strand] (0.center) to (2.center);
		\draw [style=brane] (3.center) to (1.center);
		\draw [style=su2 puncture strand, in=90, out=-90, looseness=1.25] (5.center) to (4.center);
	\end{pgfonlayer}
\end{tikzpicture}}}
    \end{pmatrix}} \theta \oplus \theta_+\{-1\} \big\}
\end{equation}
Acting with braiding gives
\begin{align}
    & \mathbb{B}_{\tau(\pi)}\bigg( \upsilon_+ \xrightarrow{\begin{tikzpicture}[scale=.5]
	\begin{pgfonlayer}{nodelayer}
		\node [style=none] (0) at (-0.25, 0.5) {};
		\node [style=none] (1) at (0.25, 0.5) {};
		\node [style=none] (2) at (-0.25, -0.25) {};
		\node [style=none] (3) at (0.25, -0.25) {};
		\node [style=none] (4) at (0, -0.25) {};
		\node [style=none] (5) at (0, 0.5) {};
	\end{pgfonlayer}
	\begin{pgfonlayer}{edgelayer}
		\draw [style=su2 puncture strand] (0.center) to (2.center);
		\draw [style=brane] (3.center) to (1.center);
		\draw [style=su2 puncture strand, in=90, out=-90, looseness=1.25] (5.center) to (4.center);
	\end{pgfonlayer}
\end{tikzpicture}} \theta_+\{-1\} \bigg) \\
    &\cong \bigg\{ \theta_+\{-1\} \xrightarrow{\begin{pmatrix}
        \vcenter{\hbox{\begin{tikzpicture}[scale=.5]
	\begin{pgfonlayer}{nodelayer}
		\node [style=none] (0) at (-0.25, 0.5) {};
		\node [style=none] (1) at (0.25, 0.5) {};
		\node [style=none] (2) at (-0.25, -0.25) {};
		\node [style=none] (3) at (0.25, -0.25) {};
		\node [style=none] (4) at (0, -0.25) {};
		\node [style=none] (5) at (0.5, 0.5) {};
	\end{pgfonlayer}
	\begin{pgfonlayer}{edgelayer}
		\draw [style=su2 puncture strand] (0.center) to (2.center);
		\draw [style=su2 puncture strand] (3.center) to (1.center);
		\draw [style=brane, in=90, out=-90, looseness=1.25] (5.center) to (4.center);
	\end{pgfonlayer}
\end{tikzpicture}}} & \vcenter{\hbox{\begin{tikzpicture}[scale=.5]
	\begin{pgfonlayer}{nodelayer}
		\node [style=none] (0) at (-0.25, 0.5) {};
		\node [style=none] (1) at (0.25, 0.5) {};
		\node [style=none] (2) at (-0.25, -0.25) {};
		\node [style=none] (3) at (0.25, -0.25) {};
		\node [style=none] (4) at (0, -0.25) {};
		\node [style=none] (5) at (0, 0.5) {};
	\end{pgfonlayer}
	\begin{pgfonlayer}{edgelayer}
		\draw [style=su2 puncture strand] (0.center) to (2.center);
		\draw [style=brane] (3.center) to (1.center);
		\draw [style=su2 puncture strand, in=90, out=-90, looseness=1.25] (5.center) to (4.center);
	\end{pgfonlayer}
\end{tikzpicture}}}
    \end{pmatrix}}
    \theta \oplus \theta_+ \{-1\}
    \xrightarrow{\begin{pmatrix}
        \vcenter{\hbox{\includegraphics[scale=.7]{su2_simple.pdf}}} \\ 0
    \end{pmatrix}} S_\theta \bigg\} \\
    &\cong \bigg\{ \theta\{-2\} \oplus \theta_+\{-1\}
        \xrightarrow{\begin{pmatrix}
        0 & -\vcenter{\hbox{\begin{tikzpicture}[scale=.5]
	\begin{pgfonlayer}{nodelayer}
		\node [style=none] (0) at (-0.25, 0.5) {};
		\node [style=none] (1) at (0.25, 0.5) {};
		\node [style=none] (2) at (-0.25, -0.25) {};
		\node [style=none] (3) at (0.25, -0.25) {};
		\node [style=none] (4) at (0.5, -0.25) {};
		\node [style=none] (5) at (0, 0.5) {};
	\end{pgfonlayer}
	\begin{pgfonlayer}{edgelayer}
		\draw [style=su2 puncture strand] (0.center) to (2.center);
		\draw [style=su2 puncture strand] (3.center) to (1.center);
		\draw [style=brane, in=90, out=-90, looseness=1.25] (5.center) to (4.center);
	\end{pgfonlayer}
\end{tikzpicture}}} & \vcenter{\hbox{\begin{tikzpicture}[scale=.5]
	\begin{pgfonlayer}{nodelayer}
		\node [style=none] (0) at (-0.25, 0.5) {};
		\node [style=none] (1) at (0.25, 0.5) {};
		\node [style=none] (2) at (-0.25, -0.25) {};
		\node [style=none] (3) at (0.25, -0.25) {};
		\node [style=none] (4) at (-0.5, -0.25) {};
		\node [style=none] (5) at (0, 0.5) {};
	\end{pgfonlayer}
	\begin{pgfonlayer}{edgelayer}
		\draw [style=su2 puncture strand] (0.center) to (2.center);
		\draw [style=su2 puncture strand] (3.center) to (1.center);
		\draw [style=brane, in=90, out=-90, looseness=1.25] (5.center) to (4.center);
	\end{pgfonlayer}
\end{tikzpicture}}} & 0 \\
        \vcenter{\hbox{\begin{tikzpicture}[scale=.5]
	\begin{pgfonlayer}{nodelayer}
		\node [style=none] (0) at (-0.25, 0.5) {};
		\node [style=none] (1) at (0.25, 0.5) {};
		\node [style=none] (2) at (-0.25, -0.25) {};
		\node [style=none] (3) at (0.25, -0.25) {};
		\node [style=none] (4) at (0, -0.25) {};
		\node [style=none] (5) at (0.5, 0.5) {};
	\end{pgfonlayer}
	\begin{pgfonlayer}{edgelayer}
		\draw [style=su2 puncture strand] (0.center) to (2.center);
		\draw [style=su2 puncture strand] (3.center) to (1.center);
		\draw [style=brane, in=90, out=-90, looseness=1.25] (5.center) to (4.center);
	\end{pgfonlayer}
\end{tikzpicture}}} & -\vcenter{\hbox{\begin{tikzpicture}[scale=.5]
	\begin{pgfonlayer}{nodelayer}
		\node [style=none] (0) at (-0.25, 0.5) {};
		\node [style=none] (1) at (0.25, 0.5) {};
		\node [style=none] (2) at (-0.25, -0.25) {};
		\node [style=none] (3) at (0.25, -0.25) {};
		\node [style=none] (4) at (0, -0.25) {};
		\node [style=none] (5) at (0, 0.5) {};
	\end{pgfonlayer}
	\begin{pgfonlayer}{edgelayer}
		\draw [style=su2 puncture strand] (0.center) to (2.center);
		\draw [style=brane] (3.center) to (1.center);
		\draw [style=su2 puncture strand, in=90, out=-90, looseness=1.25] (5.center) to (4.center);
	\end{pgfonlayer}
\end{tikzpicture}}} & 0 & \vcenter{\hbox{\begin{tikzpicture}[scale=.5]
	\begin{pgfonlayer}{nodelayer}
		\node [style=none] (0) at (-0.25, 0.5) {};
		\node [style=none] (1) at (0.25, 0.5) {};
		\node [style=none] (2) at (-0.25, -0.25) {};
		\node [style=none] (3) at (0.25, -0.25) {};
		\node [style=none] (4) at (0, -0.25) {};
		\node [style=none] (5) at (0, 0.5) {};
	\end{pgfonlayer}
	\begin{pgfonlayer}{edgelayer}
		\draw [style=su2 puncture strand] (0.center) to (2.center);
		\draw [style=brane] (3.center) to (1.center);
		\draw [style=su2 puncture strand, in=90, out=-90, looseness=1.25] (5.center) to (4.center);
	\end{pgfonlayer}
\end{tikzpicture}}}
        \end{pmatrix}}
        \theta \oplus \theta_+\{-1\} \oplus \theta_-\{-1\} \oplus \theta_+\{-1\}
        \xrightarrow{\begin{pmatrix} \vcenter{\hbox{\begin{tikzpicture}[scale=.5]
	\begin{pgfonlayer}{nodelayer}
		\node [style=none] (0) at (-0.25, 0.5) {};
		\node [style=none] (1) at (0.25, 0.5) {};
		\node [style=none] (2) at (-0.25, -0.25) {};
		\node [style=none] (3) at (0.25, -0.25) {};
		\node [style=none] (4) at (0, -0.25) {};
		\node [style=none] (5) at (0, 0.5) {};
	\end{pgfonlayer}
	\begin{pgfonlayer}{edgelayer}
		\draw [style=su2 puncture strand] (0.center) to (2.center);
		\draw [style=su2 puncture strand] (3.center) to (1.center);
		\draw [style=brane, in=90, out=-90, looseness=1.25] (5.center) to (4.center);
	\end{pgfonlayer}
\end{tikzpicture}}} \\ \vcenter{\hbox{\begin{tikzpicture}[scale=.5]
	\begin{pgfonlayer}{nodelayer}
		\node [style=none] (0) at (-0.25, 0.5) {};
		\node [style=none] (1) at (0.25, 0.5) {};
		\node [style=none] (2) at (-0.25, -0.25) {};
		\node [style=none] (3) at (0.25, -0.25) {};
		\node [style=none] (4) at (0, -0.25) {};
		\node [style=none] (5) at (0.5, 0.5) {};
	\end{pgfonlayer}
	\begin{pgfonlayer}{edgelayer}
		\draw [style=su2 puncture strand] (0.center) to (2.center);
		\draw [style=su2 puncture strand] (3.center) to (1.center);
		\draw [style=brane, in=90, out=-90, looseness=1.25] (5.center) to (4.center);
	\end{pgfonlayer}
\end{tikzpicture}}} \\ \vcenter{\hbox{\begin{tikzpicture}[scale=.5]
	\begin{pgfonlayer}{nodelayer}
		\node [style=none] (0) at (-0.25, 0.5) {};
		\node [style=none] (1) at (0.25, 0.5) {};
		\node [style=none] (2) at (-0.25, -0.25) {};
		\node [style=none] (3) at (0.25, -0.25) {};
		\node [style=none] (4) at (0, -0.25) {};
		\node [style=none] (5) at (-0.5, 0.5) {};
	\end{pgfonlayer}
	\begin{pgfonlayer}{edgelayer}
		\draw [style=su2 puncture strand] (0.center) to (2.center);
		\draw [style=su2 puncture strand] (3.center) to (1.center);
		\draw [style=brane, in=90, out=-90, looseness=1.25] (5.center) to (4.center);
	\end{pgfonlayer}
\end{tikzpicture}}} \\ 0 \end{pmatrix}}
        \theta \bigg\} \\
    &\cong \bigg\{ \theta\{-2\}
        \xrightarrow{\begin{pmatrix}
        \vcenter{\hbox{\begin{tikzpicture}[scale=.5]
	\begin{pgfonlayer}{nodelayer}
		\node [style=none] (0) at (-0.25, 0.5) {};
		\node [style=none] (1) at (0.25, 0.5) {};
		\node [style=none] (2) at (-0.25, -0.25) {};
		\node [style=none] (3) at (0.25, -0.25) {};
		\node [style=none] (4) at (0.5, -0.25) {};
		\node [style=none] (5) at (0, 0.5) {};
	\end{pgfonlayer}
	\begin{pgfonlayer}{edgelayer}
		\draw [style=su2 puncture strand] (0.center) to (2.center);
		\draw [style=su2 puncture strand] (3.center) to (1.center);
		\draw [style=brane, in=90, out=-90, looseness=1.25] (5.center) to (4.center);
	\end{pgfonlayer}
\end{tikzpicture}}} & \vcenter{\hbox{\begin{tikzpicture}[scale=.5]
	\begin{pgfonlayer}{nodelayer}
		\node [style=none] (0) at (-0.25, 0.5) {};
		\node [style=none] (1) at (0.25, 0.5) {};
		\node [style=none] (2) at (-0.25, -0.25) {};
		\node [style=none] (3) at (0.25, -0.25) {};
		\node [style=none] (4) at (-0.5, -0.25) {};
		\node [style=none] (5) at (0, 0.5) {};
	\end{pgfonlayer}
	\begin{pgfonlayer}{edgelayer}
		\draw [style=su2 puncture strand] (0.center) to (2.center);
		\draw [style=su2 puncture strand] (3.center) to (1.center);
		\draw [style=brane, in=90, out=-90, looseness=1.25] (5.center) to (4.center);
	\end{pgfonlayer}
\end{tikzpicture}}}
        \end{pmatrix}}
        \theta_+\{-1\} \oplus \theta_-\{-1\} \bigg\} \\
    &\cong \big\{ \upsilon_-\{-1\} \xrightarrow{\begin{tikzpicture}[scale=.5]
	\begin{pgfonlayer}{nodelayer}
		\node [style=none] (0) at (-0.25, 0.5) {};
		\node [style=none] (1) at (0.25, 0.5) {};
		\node [style=none] (2) at (-0.25, -0.25) {};
		\node [style=none] (3) at (0.25, -0.25) {};
		\node [style=none] (4) at (0.5, -0.25) {};
		\node [style=none] (5) at (0, 0.5) {};
	\end{pgfonlayer}
	\begin{pgfonlayer}{edgelayer}
		\draw [style=su2 puncture strand] (0.center) to (2.center);
		\draw [style=su2 puncture strand] (3.center) to (1.center);
		\draw [style=brane, in=90, out=-90, looseness=1.25] (5.center) to (4.center);
	\end{pgfonlayer}
\end{tikzpicture}} \theta_+\{-1\} \big\},
\end{align}
so the map gets sent to $\vcenter{\hbox{\begin{tikzpicture}[scale=.5]
	\begin{pgfonlayer}{nodelayer}
		\node [style=none] (0) at (-0.25, 0.5) {};
		\node [style=none] (1) at (0.25, 0.5) {};
		\node [style=none] (2) at (-0.25, -0.25) {};
		\node [style=none] (3) at (0.25, -0.25) {};
		\node [style=none] (4) at (0.5, -0.25) {};
		\node [style=none] (5) at (0, 0.5) {};
	\end{pgfonlayer}
	\begin{pgfonlayer}{edgelayer}
		\draw [style=su2 puncture strand] (0.center) to (2.center);
		\draw [style=su2 puncture strand] (3.center) to (1.center);
		\draw [style=brane, in=90, out=-90, looseness=1.25] (5.center) to (4.center);
	\end{pgfonlayer}
\end{tikzpicture}}}$ from $\upsilon_-\{-1\}$ to $\theta_+\{-1\}$.
\end{proof}

\section{An easy object to braid} \label{Lambda object}

We will often denote objects with more than one black strand as ``products'' of objects with fewer black strands. What we mean is to place the black strands of both objects side by side among the same set of red strands. 
We do the same with the morphisms. It is always assumed that we place the black strands so as not to introduce any additional crossings. We will not use this notation if this is not possible.

For example, we write $\theta^n = \prod_{i=1}^n \theta$ for the module generated by
\begin{equation}
    \vcenter{\hbox{
    \begin{tikzpicture}[scale=0.5]
        \draw[style=su2 puncture strand] (0,0) -- (0,2);
        \draw[style=su2 puncture strand] (2,0) -- (2,2);
        \draw (.25,0) -- (.25,2);
        \draw (.5,0) -- (.5,2);
        \node at (1,1) {$\scriptstyle\cdots$};
        \draw (1.5,0) -- (1.5,2);
        \draw (1.75,0) -- (1.75,2);
        \node at (1,-.6) {$\underbrace{\hspace{.75cm}}_n$};
    \end{tikzpicture}}}
\end{equation}
and we write
$\upsilon_+ \times \theta_+^{n-1}$ for the object of $\mathcal{C}_{\bullet, n, F}$ coming from placing an additional $n-1$ black strands to the right of the objects whose cone defines $\upsilon_+$, and extending the map by the identity diagram on the new factors:
\begin{equation}
    \upsilon_+ \times \theta_+^{n-1} :=\big\{ \theta_+^{n} \xrightarrow{\begin{tikzpicture}[scale=.5]
	\begin{pgfonlayer}{nodelayer}
		\node [style=none] (0) at (-0.25, 0.5) {};
		\node [style=none] (1) at (0.25, 0.5) {};
		\node [style=none] (2) at (-0.25, -0.25) {};
		\node [style=none] (3) at (0.25, -0.25) {};
		\node [style=none] (4) at (0, -0.25) {};
		\node [style=none] (5) at (0.5, 0.5) {};
		\node [style=none] (6) at (0.75, 0.5) {};
		\node [style=none] (7) at (0.75, -0.25) {};
		\node [style=none] (8) at (1.25, 0.5) {};
		\node [style=none] (9) at (1.25, -0.25) {};
		\node [style=none] (10) at (1, 0.125) {$\scriptstyle\cdots$};
	\end{pgfonlayer}
	\begin{pgfonlayer}{edgelayer}
		\draw [style=su2 puncture strand] (0.center) to (2.center);
		\draw [style=su2 puncture strand] (3.center) to (1.center);
		\draw [style=brane, in=90, out=-90, looseness=1.25] (5.center) to (4.center);
		\draw (6.center) to (7.center);
		\draw (8.center) to (9.center);
	\end{pgfonlayer}
\end{tikzpicture}} \theta \times \theta_+^{n-1} \big\}.
\end{equation}

Define
\begin{equation}\label{eqn:T2_res}
    \Lambda_n := \left( \upsilon_+ \times \theta_+^{n-1} \right)^{\oplus n} \xrightarrow{d} \theta_+^n
\end{equation}
with
\begin{equation} \label{eqn:T2_res_map}
    \vcenter{\hbox{
    \begin{tikzpicture}[scale=0.5]
        \node at (-1.25,1) {$d_i=$};
        \draw[style=su2 puncture strand] (0,0) -- (0,2);
        \draw (2,0) to node[pos=0, below]{$i$} (1,2);
        \draw[style=su2 puncture strand] (0.5,0) -- (0.5,2);
        \draw (1.5,0) -- (1.5,2);
        \node at (2,1) {$\scriptstyle\cdots$};
        \draw (2.5,0) -- (2.5,2);
    \end{tikzpicture}}}
\end{equation}
and
\begin{equation}\label{eqn:BT2_res}
    \Lambda'_n := \left( \upsilon_- \times \theta_+^{n-1} \right)^{\oplus n} \xrightarrow{d'} \theta_+^n
\end{equation}
with
\begin{equation} \label{eqn:BT2_res_map}
    \vcenter{\hbox{
    \begin{tikzpicture}[scale=0.5]
        \node at (-1.25,1) {$d'_i=$};
        \draw[style=su2 puncture strand] (0,0) -- (0,2);
        \draw[style=su2 puncture strand] (1,0) -- (1,2);
        \draw (2,0) to node[pos=0, below]{$i$} (0.5,2);
        \draw (1.5,0) -- (1.5,2);
        \node at (2,1) {$\scriptstyle\cdots$};
        \draw (2.5,0) -- (2.5,2);
    \end{tikzpicture}}}
\end{equation}

\begin{proposition} \label{L_n braiding}
    $\mathbb{B}_{\tau(\pi)}(\Lambda_n)=\Lambda_n'$
\end{proposition}

\begin{proof}
$\mathbb{B}_{\tau(\pi)}$ sends $\upsilon_+ \times \theta_+^{n-1}$ to $\upsilon_- \times \theta_+^{n-1}$ by a similar calculation as the proof of Proposition \ref{braiding U} and acts as the identity on $\theta_+^n$. To see where it sends $d$, we use that $\vcenter{\hbox{\begin{tikzpicture}[scale=.5]
	\begin{pgfonlayer}{nodelayer}
		\node [style=none] (0) at (-0.25, 0.5) {};
		\node [style=none] (1) at (0.25, 0.5) {};
		\node [style=none] (2) at (-0.25, -0.25) {};
		\node [style=none] (3) at (0.25, -0.25) {};
		\node [style=none] (4) at (0, -0.25) {};
		\node [style=none] (5) at (0, 0.5) {};
	\end{pgfonlayer}
	\begin{pgfonlayer}{edgelayer}
		\draw [style=su2 puncture strand] (0.center) to (2.center);
		\draw [style=brane] (3.center) to (1.center);
		\draw [style=su2 puncture strand, in=90, out=-90, looseness=1.25] (5.center) to (4.center);
	\end{pgfonlayer}
\end{tikzpicture}}}:\upsilon_+ \to \theta_+\{-1\}$ is sent to $\vcenter{\hbox{\begin{tikzpicture}[scale=.5]
	\begin{pgfonlayer}{nodelayer}
		\node [style=none] (0) at (-0.25, 0.5) {};
		\node [style=none] (1) at (0.25, 0.5) {};
		\node [style=none] (2) at (-0.25, -0.25) {};
		\node [style=none] (3) at (0.25, -0.25) {};
		\node [style=none] (4) at (0.5, -0.25) {};
		\node [style=none] (5) at (0, 0.5) {};
	\end{pgfonlayer}
	\begin{pgfonlayer}{edgelayer}
		\draw [style=su2 puncture strand] (0.center) to (2.center);
		\draw [style=su2 puncture strand] (3.center) to (1.center);
		\draw [style=brane, in=90, out=-90, looseness=1.25] (5.center) to (4.center);
	\end{pgfonlayer}
\end{tikzpicture}}}:\upsilon_- \to \theta_+\{-1\}$ from Proposition \ref{braiding U map}. This result is unaffected by additional black strands to the right of both diagrams. We compose this result with any number of crossings of the black strands to find that $d_i$ is sent to $d_i'$.
\end{proof}

\section{Fukaya-Seidel categories of multiplicative Coulomb branches} \label{sec: ADLSZ}

Here we recall from \cite{ADLSZ} various results on how to construct objects and compute morphisms in $Fuk(\mathcal{M}^\times(\Gamma, \vec{d}), \mathcal{W}_{\mathbf{a}})$.  

\subsection{Drawing objects}

Recall from \cite{BFN} that if $\mathbf T$ is a maximal torus of the quiver gauge group, and $W$ is the corresponding Weyl group, then there is a 
map $\MCB(\Gamma, \vec d) \to \mathbf{T}/W$.  We denote its $W$-cover as $y: \widetilde \MCB(\Gamma, \vec d) \to \mathbf{T}$.

\begin{theorem} \label{trivialization} \cite[Thm. 1.2]{ADLSZ} 
Fix a maximal torus of the quiver gauge group, 
$\mathbf T \cong \prod_i \prod_{\alpha = 1}^{d_i} \C^*$.  For $\mathbf{t} \in \mathbf T$, we 
write its coordinates as
$t_{i, \alpha} \in \C^*$.  
Fix 
$\mathbf a \in \mathbf{T}_F = \prod_{i} \prod_{\alpha = 1}^{m_i} \C^*$
so that the coordinate entries 
$a_{i, \alpha} \in \C^*$ have distinct arguments. 
We write:  $\mathbf{T}_{O} \subset \mathbf{T}$ for
the complement of the following hyperplanes:
\begin{enumerate}
\itemsep0em
\item \label{intro root hyperplane}
The locus where some $t_{i, \alpha} = t_{i, \alpha'}$ 
\item The locus where some $t_{i, \alpha}$ coincides
with some $t_{j, \beta}$ for adjacent nodes $i, j$. 
\item The locus where some $t_{i, \alpha}$ coincides with some $a_{i, \beta}$.  
\end{enumerate}

We write $\widetilde \MCB(\Gamma, \vec d)_O := y^{-1}(T_O)$.  Then there is a $W$-equivariant isomorphism 
    $$(u, y):  \widetilde \MCB(\Gamma, \vec d)_O \to \mathbf T^\vee \times \mathbf T_{O}$$ 
    such that 
    $$\mathcal{W}_{\mathbf{a}} := \sum_{i, \alpha} u_{i,\alpha}$$
    extends to a $W$-invariant regular function on $\widetilde \MCB(\Gamma, \vec d)$. 
\end{theorem}

The map $u$  is given in \cite{ADLSZ} in terms of the monopole operators of \cite{BFN}.  The function $\mathcal{W}_{\mathbf{a}}$ was also shown to agree in local coordinates with the explicit proposal of \cite{aganagic-knot-1}. 
(In the present article, we will focus on the case $\Gamma = \bullet$, for which the hyperplane of type (2) in the theorem does not appear.)

Theorem \ref{trivialization} allows one to describe certain Lagrangians by diagrams on the page.  

\begin{definition}
    A $\vec d$ multi-curve in $\C^*$ is a collection of curves labeled by the nodes of the quiver, with $d_i$ curves labeled $i$. 

    We say the collection is admissible if curves are embedded, conic at infinity under the identification $\C^* = T^* S^1$, and no $i$ curve intersects a $j$ curve unless $i, j$ are distinct and not adjacent in the quiver.  If in addition
    every $i$ curve avoids all points $a_{i, \beta}$, we say the collection of curves is $\mathbf{a}$-admissible. 
\end{definition}

\begin{definition}
Given an admissible $\vec{d}$-multicurve $\gamma$, we will form a Lagrangian $L_\gamma \subset \MCB(\Gamma, \vec{d})$ as follows.  Temporarily number the $i$-curves from $1$ to $d_i$; we denote a given one as $\gamma_{i, \alpha}$ with $\alpha \in \{1, \ldots, d_i\}$.  We write $\widetilde \gamma$ for this ordered multicurve.  Then we may form a Lagrangian $\widetilde l_{\widetilde \gamma} = \prod \gamma_{i, \alpha} \subset \mathbf T_{O}$, and 
$$\widetilde L_{\widetilde \gamma} = (i \R_{>0})^{|\vec d|} \times l_{\widetilde \gamma} \subset \mathbf{T}^\vee \times \mathbf{T}_O.$$
Each $W$ orbit meets $\widetilde{L}_{\widetilde \gamma}$ and $\widetilde l_{\widetilde \gamma}$ in at most one point.  We write $L_\gamma$ and $l_\gamma$ for their images in the $W$-quotient; note these depend only on the original multicurve and not on the temporary numbering. 
Theorem \ref{five authors theorem} implies there is an embedding 
$(\mathbf{T}^\vee \times \mathbf{T}_O)/W \hookrightarrow \MCB(\Gamma, \vec{d})$, and we preserve the notation $L_\gamma$ for the image under the embedding.  
\end{definition}
Note that Theorem \ref{five authors theorem} also implies that $L_\gamma$ is is conic at infinity, and that if $\gamma$ is in fact $\mathbf{a}$-admissible, then $L_\gamma$ stays away from the stop associated to the superpotential $\mathcal{W}_{\mathbf{a}}$. 

The space $\MCB(\Gamma, \vec{d})$ is always affine.  If it is in addition smooth (this is known to hold when $\Gamma$ is of ADE type), we may consider the Fukaya-Seidel category  $Fuk(\MCB(\Gamma, \vec{d}), \mathcal{W}_{\mathbf{a}})$. 
Suppose in addition $\gamma$ contains no closed curves.  Then above discussion implies that $L_\gamma$ defines an object in this category, canonical up to shift.     
Often in the text, we will simply draw a multicurve $\gamma$ to name the associated Lagrangian $L_\gamma$. 

\begin{definition}
    We write $Fuk_{| | |}(\MCB(\Gamma, \vec{d}), \mathcal{W}_{\mathbf{a}}) \subset Fuk(\MCB(\Gamma, \vec{d}), \mathcal{W}_{\mathbf{a}})$ for the full subcategory generated by objects associated to multicurves without closed components. 
\end{definition}

\vspace{2mm}

Suppose now that the entries of $\mathbf{a}$ have distinct arguments; we record these by placing a `red' point 
labelled $i$ on the unit circle at each $\mathrm{arg}(a_{i, \alpha})$.  Now fix an additional collection of $\mathbf{\theta}$ of disjoint `black' points on the circle, disjoint from the red points, and with $d_i$ points labelled $i$.  To such a collection, we associate the multicurve $\gamma(\mathbf{\theta})$ given by taking the preimage of the black points under the map $\mathrm{arg}: \C^* \to S^1$.  We will write $$T_{\mathbf{\theta}} := L_{\gamma(\theta)}.$$

In our diagrams, we will often draw the `base' cylinder $\mathbb{C}^*_y = T^*S^1$ as a rectangle in the page, understanding the horizontal boundaries to be identified, and with the circle near $\infty$ given by the top of the rectangle, and the circle near zero at the bottom.  
In our conventions, wrapping is to the right at the bottom of the rectangle, and to the left at the top.  See Figure  \ref{fig:example objects} for examples of objects specified as multicurves. 

By contrast, we draw the `fiber' $\mathbb{C}^*_u$ cylinder as a cylinder. 

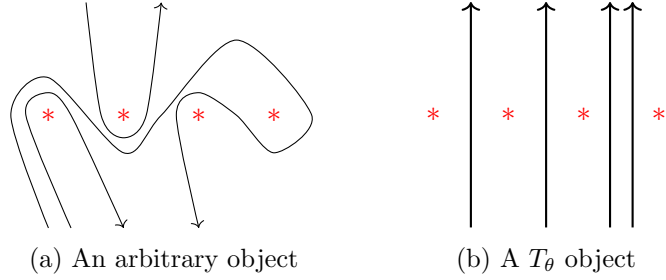
\begin{figure}
    \centering
    \begin{subfigure}[b]{0.3\textwidth}
    \centering
    \begin{tikzpicture}
        \node[red] at (0,0) {$*$};
        \node[red] at (1,0) {$*$};
        \node[red] at (2,0) {$*$};
        \node[red] at (3,0) {$*$};
        \draw[->] plot [smooth, tension=0.5] coordinates {(0,-1.5) (-.5,0) (0,.5) (1,-.5) (1.5,0) (2.5,1) (3.5,0) (3,-.5) (2.5,0) (2,.3) (1.7,0) (2,-1.5)};
        \draw[->] plot [smooth, tension=0.5] coordinates {(.3,-1.5) (-.3,0) (0,.3) (.3,0) (1,-1.5)};
        \draw[->] plot [smooth, tension=0.5] coordinates {(.5,1.5) (.7,0) (1,-.3) (1.3,0) (1.5,1.5)};
    \end{tikzpicture}
    \caption{An arbitrary object}
    \label{fig:example L}
    \end{subfigure}
    \begin{subfigure}[b]{0.3\textwidth}
    \centering
    \begin{tikzpicture}
        \node[red] at (0,0) {$*$};
        \node[red] at (1,0) {$*$};
        \node[red] at (2,0) {$*$};
        \node[red] at (3,0) {$*$};
        \draw[thick,->] (0.5,-1.5) -- (0.5,1.5);
        \draw[thick,->] (1.5,-1.5) -- (1.5,1.5);
        \draw[thick,->] (2.35,-1.5) -- (2.35,1.5);
        \draw[thick,->] (2.65,-1.5) -- (2.65,1.5);
    \end{tikzpicture} 
    \caption{A $T_\theta$ object}
    \label{fig:example Ttheta}
    \end{subfigure}
    \caption{Example of objects in $Fuk_{| | |}(\MCB(\Gamma, \vec{d}), \mathcal{W}(a))$}
    \label{fig:example objects}
\end{figure}

\subsection{Gradings} \label{anchor grading}

First let us recall how, in general, gradings on Lagrangian Floer homology can arise from lifts to a cover; see e.g. \cite[Sec. 3.2]{Sheridan} for a detailed account.  
For a symplectic manifold $X$, a subgroup $G \subset H^1(X, \Z)$,  Lagrangians $L, M \subset X$, and their fixed choice of lifts $\widetilde{L}, \widetilde{M}$ to the $G$-cover of $X$ (called `anchorings' in the literature), we have $$L \cap M = \bigsqcup_{g \in G } \widetilde{L} \cap g \cdot \widetilde{M}$$
Because disks are simply connected hence lift to the cover, the induced grading on Floer homologies is compatible with all structures. 
So long as $L, M$ admit lifts, the grading on $\Hom(L, L)$, $\Hom(M, M)$ and the relative grading on $\Hom(L, M)$ do not depend on the choice of lifts.  For most of this article, we will only use simply connected Lagrangians, which thus always admit lifts. 

When considering Lagrangians $L, M \subset X \setminus D$ for some divisor $D$, one can still grade the Lagrangian Floer homology {\em in $X$} by $G \subset H^1(X \setminus D, \Z)$ in the same manner, so long as one works over the group ring of $G' = \mathrm{ker} (H_1(X \setminus D, \Z) \to H_1(X, \Z))$, and counts disks in $X$ by the class of their boundary in this ring.  Said differently, since this kernel is the image of $H_2(X, X \setminus D) \cong H^{n-2}(D)$ in the long exact sequence, one should view the Novikov variable counting intersections with $D$ as graded.

Returning to our setting, in \cite{ADLSZ} the space $\MCB(\Gamma, \vec{d}), \mathcal{W}_{\mathbf{a}}$
had an integral $H^1$ class was there used to define a grading, there called ``q-grading'' -- here we will call it the J-grading.  

In \cite{ADLSZ}, the image in  $\MCB(\Gamma, \vec{d})$ of the locus from (1) of Theorem \ref{trivialization} is termed the ``root divisor'' and the image of the loci (2), (3) is termed ``matter divisor''.  These divisors are avoided by all multicurve Lagrangians, and wrappings may be performed in their complement.  The category $Fuk_{| | |}(\MCB(\Gamma, \vec{d}), \mathcal{W}_{\mathbf{a}})$ was in \cite{ADLSZ} defined over the ring $\Z[\hbar,u]$ (there $u$ was called $\eta$), where $\hbar$ and $u$  count intersections with the root and matter divisors respectively.  

We will correspondingly refer to the gradings induced on morphism spaces 
by the loops around these divisors as the $\hbar$ and $u$ gradings, respectively. 
We will use the existence of these gradings, but will only ever need to explicitly compute the $J$ grading. We will give an algorithm for computing the $J$-grading at the end of this section. 

\subsection{Drawing morphisms}

The main result of \cite{ADLSZ} is: 

\begin{theorem} \cite[Thm. 1.7]{ADLSZ} \label{five authors theorem}
    For $\Gamma$ of ADE type, and $\mathbf{a}$ consisting of points with distinct arguments, there is an embedding
    \begin{eqnarray*}
        \mathcal{C}_{\Gamma, \vec{d}, \arg(\mathbf{a})} & \hookrightarrow & Fuk_{| | |}(\MCB(\Gamma, \vec{d}), \mathcal{W}_{\mathbf{a}}) \\ 
        \theta & \mapsto & T_\theta
    \end{eqnarray*} 
    This embedding is linear over $\Z[u, \hbar]$ and there are choices of anchorings of the $T_\theta$ such that the map respects $u, \hbar, J$ gradings. 
\end{theorem}

\begin{remark}
    The discussion of $u$ and $\hbar$ gradings was not explicit in the currently available version of \cite{ADLSZ}, but the proof given there establishes the result as stated above.
\end{remark}

Consider some $\theta$ and suppose $d \in S^1$ is the location of a black dot.  Additionally suppose given $t \in \R$ such that $d+t$ is disjoint from all dots in $\theta \setminus d$.  Then we write $\theta(d \rightsquigarrow d+t)$ for the object of the KLRW category in which the black dot at $d$ is replaced by a black dot at $d+t$, and  
$$ 
[d \to d+t]: \theta \mapsto \theta(d \rightsquigarrow d+t) 
$$ 
for the morphism in the KLRW category which winds the black dot at $d$ forward by $t$.  (If $t$ is negative, then backward by $-t$.) 

In \cite{ADLSZ}, we described only the images of enough morphisms to generate the morphism spaces under composition.  Let us note a few more: 

\begin{lemma}
    The functor of Theorem \ref{five authors theorem} carries $[d \to d+t]$ to the morphism given as the identity on all components of the multicurve other than $d$ and $d+t$, and the length $|t|$ positive Reeb chord from $d$ to $d+t$ near $\infty$ for $t > 0$ and near $0$ for $t < 0$.  
\end{lemma}
\begin{proof}
    When the path $d \rightsquigarrow d+t$ crosses zero or one red or black strands, this is true by definition (see \cite[Sec. 8.4]{ADLSZ}).  More generally, observe that the composition of images of the one-crossing morphisms on the Fukaya side involves the count of a single disk (see Figure \ref{fig:example disk}), matching the KLRW composition.  
\end{proof}

\begin{figure}
    \centering
    \begin{tikzpicture}
        \node[red] at (0,1) {$*$};
        \draw[thick,->,teal] (0.5,-2) -- (0.5,3);
        \draw[thick,->,teal] (-1,-2) -- (-1,3);
        \draw[thick,->] (1.25,-2) -- (-2.25,3);
        \draw[thick,->] (.75,-2) -- (-2.75,3);
        \draw[thick,->,blue] (2.75,-2) -- (-4,3);
        \draw[thick,->,blue] (3.5,-2) -- (-3,3);
        \node[label={[right]$p_1$}]  (p1b) at (.5,-.95) [circle,fill,inner sep=1.5pt]{};
        \node[label={[left]$p_1$}]  (p1t) at (-1,.5) [circle,fill,inner sep=1.5pt]{};
        \node[label={[left]$p_2$}] (p2t) at (-2.45,2.55)  [circle,fill,inner sep=1.5pt]{};
        \node[label={[right]$p_2$}]  (p2b) at (-.35,.3) [circle,fill,inner sep=1.5pt]{};
        \node[label={[right]$p_3$}]  (p3t) at (-1,1.45) [circle,fill,inner sep=1.5pt]{};
        \node[label={[right]$p_3$}]  (p3b) at (.5,-.35) [circle,fill,inner sep=1.5pt]{};
        \begin{scope}[on background layer] 
        \fill [red!20] (p1t.center) -- (p2t.center) -- (p3t.center) -- cycle;
        \fill [red!20] (p1b.center) -- (p2b.center) -- (p3b.center) -- cycle;
        \end{scope}
    \end{tikzpicture} 
    \caption{The disk count yielding $p_1 \cdot p_2 = p_3$ where $p_1 = \vcenter{\hbox{\includegraphics{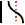}}}$, $p_2 = \vcenter{\hbox{\includegraphics{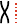}}}$, and $p_3 = \vcenter{\hbox{\includegraphics{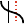}}}$}
    \label{fig:example disk}
\end{figure}

We illustrate the assertion of the Lemma in Figure \ref{fig:klrw_diagram_from_geometry}.

\begin{figure}
    \centering
    \begin{tikzpicture}[yscale=0.75]
        \node[red] at (0,0) {$*$};
        \node[red] at (3,0) {$*$};
        \draw[thick,->] (0.5,-2) -- (0.5,2);
        \draw[thick,->,teal] (0.75,-2) -- (0.75,2);
        \draw[thick,->] (1.5,-2) -- (1.5,2);
        \draw[thick,->] (2.25,-2) -- (2.25,2);
        \draw[thick,->,teal] (2.5,-2) -- (2.5,2);
        \draw[thick,->,teal] (3.5,-2) -- (3.5,2);
        \draw[dashed, ->] (0.5,0) -- (0.75,0);
        \draw[dashed, ->] (1.5,-1.75) -- (3.5,-1.75);
        \draw[dashed, ->] (2.25,0) -- (2.5,0);
    \end{tikzpicture}
    \caption{The image of the morphism $\vcenter{\hbox{\includegraphics{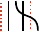}}}$ under the embedding of Theorem \ref{five authors theorem}}
    \label{fig:klrw_diagram_from_geometry}
\end{figure}

\subsection{Counting disks; cylindrical model}

To count holomorphic disks in the $\MCB(\Gamma, \vec{d})$, \cite{ADLSZ} developed a `cylindrical model' building on Lipschitz's approach to Heegard-Floer theory \cite{lipshitz2006cylindrical}.  Indeed, recall first that for any target curve $X$ and domain curve $C$, there is a bijection (a) between maps $C \to Sym^d X$ transverse to the diagonal and (b) maps $S \to C \times X$, where $S \to C$ is a $d:1$ cover with simple ramification.

Consider now a map $\Phi: C \to \MCB(\Gamma, \vec{d})$.  Per the trivialization recalled in Theorem \ref{trivialization} (note the trivialization depends on the choice of $\mathbf{a}$), we obtain, in at least in the complement of the relevant divisors, maps $C \supset C_O \to \mathbf{T}^\vee / W$ and $C \supset C_O \to \mathbf{T}_{O} / W$, i.e., to certain products of symmetric products. So long as the original map did not land entirely in the removed divisors, these two maps determine the original map $C \to \MCB(\Gamma, \vec{d})$.  
Since there is a global map $\MCB(\Gamma, \vec{d}) \to \mathbf{T}/ W$, the map $C_O \to \mathbf{T}_{O} / W$ extends uniquely to $C \to \mathbf{T} / W$, and we may characterize  it (assuming transversality to the diagonal) via some map 
$$\Phi_C \times \Phi_y: S \to C \times \mathbb{C}^*_y$$ 
as above.  
The other map $C_O \to \mathbf{T}^\vee / W$ is then characterized by some map $\Phi_u : S_O \to \mathbb{C}_u^*$ (same $S$ because $W$ was acting simultaneously on both factors of $\mathbf{T}^\vee \times \mathbf{T}_O$), which we may uniquely extend to a map $\Phi_u: S \to \mathbb{P}^1$.    

In summary: maps $C \to \MCB(\Gamma, \vec{d})$ transverse to certain divisors can be identified with a certain subset of maps 
$$\Phi_C \times \Phi_y \times \Phi_u: S \to C \times \mathbb{C}^*_y \times \mathbb{C}^*_u$$
A fundamental result of \cite{ADLSZ}
was a characterization of the image.  Here we recall the result in the special case $\Gamma = \bullet$  (which is somewhat simpler to state, and anyway the case of relevance here). 

\begin{theorem} \label{sl2 cylindrical model} \cite[Thm. 1.5]{ADLSZ}
    Suppose $\Phi_C: S \to C$ is a $d:1$ branched cover with simple branching.  Then a map 
    $\Phi_C \times \Phi_y \times \Phi_u: S \to C \times \mathbb{C}^*_y \times \mathbb{P}^1_u$ arises as desrcibed above from a (unique) map $C \to \MCB(\bullet, d)$ iff all zeros and poles of $\Phi_u$ are simple, the zeros occur exactly over $\Phi_y^{-1}(\mathbf{a})$, and the poles appear exactly over the branch points of $\Phi_C$. 
\end{theorem}

\subsection{Computing the J-grading}

We will need to compute explicitly the relative J-grading between pairs of morphisms below; here we describe an explicit algorithm.  We restrict attention to the case $\Gamma = \bullet$. Given a pair of intersection points $p,q \in Hom(L_0,L_1)$, their relative J-grading $J(q)-J(p)$ is by definition given by the change of phase of
\begin{equation} 
    f_0 (u,y) = \prod_i u^{-1}_i \prod_{i,j} (1-a_i/y_j) \prod_{i \neq j} (1-y_i/y_j)^{-1}
\end{equation}
around a clockwise loop connecting the two points which goes from $p$ to $q$ along $L_0$ and from $q$ to $p$ along $L_1$.

The change of phase can be calculated in terms of the cylindrical model as follows: 

\begin{equation} \label{eqn:J deg combinatorial}
    J(q)-J(p) =
    \#\mathrm{branch\,points}(\Phi_C) - \# \Phi_y^{-1}(\mathbf{a}) + \Phi_u^{-1}(0) - \Phi_u^{-1}(\infty)
\end{equation}

Note that the grading in \cite{webster} is twice this grading.

\section{Cone as surgery} \label{sec: conesurgery}

It is a basic idea in Lagrangian Floer theory that if two Lagrangians $L, M$ intersect transversely at a single point, then the cone on the corresponding morphism $L \to M$ should be isomorphic to the Polterovich surgery $L \# M$ of the two Lagrangians in question.  Indeed, it is obvious that 
for any given test Lagrangian $N$, there is (for sufficiently small surgery) a natural bijection between the generators $N \cap (L \# M)$ of $\Hom(N, L \# M)$ and the generators  
$(N \cap L) \sqcup (N \cap M)$ of $\Hom(N, Cone(L \to M))$.  

When $\dim L = 1$, it is also easy to see that e.g. the differential disks for $\Hom(N, L \# M)$ can be put in bijection with the collection of differential disks for $\Hom(N, L)$, differential disks for $\Hom(N, M)$, and composition disks $N \to L \to M$.  The corresponding result in higher dimensions is less easy, but has been established \cite{FOOO-surgery}.  
There is also a version for `surgery at infinity' established in \cite{GPS2}.

Let us give the analogous statement in the multicurve formalism. We will need only the 
`surgery at infinity' version. 

\begin{proposition} \label{multicurve cone as surgery}
    Fix $\Gamma, \vec{d}, \mathbf{a}$.  
    Let $\gamma$ be an $\mathbf{a}$-admissible $\vec{d}$-multicurve.  
    Let $c$ and $d$ be curves labelled $i$, each disjoint from $\gamma$, such that $\gamma \sqcup c$ and $\gamma \sqcup d$ are $\mathbf{a}$-admissible $\vec{d}+1_i$-multicurves.  Let $c \to d$ be a Reeb chord, and assume the corresponding surgery-at-infinity $c \# d$ is also disjoint from $\gamma$.  Then there is an exact triangle 
    \begin{equation} \label{multicurve surgery} L_{c \sqcup \gamma} \to L_{d \sqcup \gamma} \to L_{(c \# d) \sqcup \gamma} \xrightarrow{[1]}\end{equation}
    where the first morphism is given by the chosen Reeb chord $c \to d$ and `the identity' (i.e. the intersection point corresponding to the identity after some perturbation) on $\gamma$.  
\end{proposition}
\begin{proof}
    In fact this is just an example of the surgery (at infinity) exact triangle of \cite[Prop. 1.12]{GPS2}.  For the reader for whom the symplectic geometry justifying the previous sentence is not immediately obvious, let us sketch another argument.  The existence of {\em some} exact sequence of the form \eqref{multicurve surgery} follows by taking the curve $c \# d$, isotoping the surgery locus towards infinity, and applying the action filtration argument in the proof of \cite[Prop. 1.37]{GPS2}.\footnote{One could also apply the statement of said proposition, but that would require some consideration of the relationship between being near infinity in $\C^*$ and being near infinity in the Coulomb branch.  The reader who is comfortable with said consideration would presumably have been already comfortable with directly invoking \cite[Prop. 1.12]{GPS2}.}
    It remains to identify the first morphism in the sequence.  It suffices to do this after Hom pairing with any subcategory containing the objects in question; we use  $Fuk_{| | |}(\MCB(\Gamma, \vec{d}), \mathcal{W}_{\mathbf{a}})$.  We should show that, for multicurves $\delta$, the map 
    induced given by the action filtration argument
    $\Hom(L_{\delta}, L_{c \sqcup \gamma}) \to \Hom(L_{\delta},  L_{d \sqcup \gamma})$ is the same as the map given by the Reeb chord $c \to d$.  But the `cylindrical model' established in \cite{ADLSZ} reduces this to a (multicurve version of) the easy argument for cone=surgery for 1-dimensional Lagrangians.  
\end{proof}

\begin{corollary} \label{weak generation}
    The $T_\theta$
    generate $Fuk_{| | |}(\MCB(\Gamma, \vec{d}), \mathcal{W}_{\mathbf{a}}) \subset Fuk(\MCB(\Gamma, \vec{d}), \mathcal{W}_{\mathbf{a}})$.  
\end{corollary}
\begin{proof}
    The basic idea is to iteratively stretch and break multicurves using Proposition \ref{multicurve cone as surgery}.

    In more detail, assume without loss of generality that the $\mathbf{a}$ have different arguments.  Draw on the annulus the lines of fixed argument $\arg(a_i)$; we term them `red lines'. 
    Fix a multicurve $\eta$ of interest; assume without loss of generality that it is not tangent to red lines.  

    Choose an intersection of a red line and the multicurve, maximally distant along the red line from the corresponding $a_i$.  Let us write $k$ for the component of $\eta$ containing this point, and $\gamma := \eta \setminus k$.  We may push $k$ towards $0$ or $\infty$ along the red line, without meeting any other components of the multicurve.  We write correspondingly $k$ as the surgery at infinity of some $c$ and $d$, which are asymptotic to one side and the other of the red line.  Then the exact triangle \eqref{multicurve surgery} expresses $L_{\gamma'}$ as the cone $L_{c \cup \gamma} \to L_{d \cup \gamma}$ on two objects, each of which have one less total number of intersections with red lines.  

    We see by induction (on the number of intersections with red lines) that we may write  $L_{\eta}$ as an iterated cone of $L_\mu$, where each multicurve $L_\mu$ is disjoint from the red lines.  (If $\eta = \coprod \eta_i$, then there are $\prod_i \# \pi_0 (\eta_i \setminus \mathrm{red\, lines})$ such $L_\mu$ in the iterated cone.)  Such a multicurve is a union of (curves isotopic to) (1) lines of constant argument, and (2) curves which have both endpoints at zero or infinity and enclose none of the $a_i$. 
    But if $\mu$ has any type (2) curves, it is a zero object, as it can be displaced from any other Lagrangian by compactly supported Hamiltonian isotopy.  After suppressing  zero objects, we have expressed our original $L_\eta$ as a twisted complex in objects $T_\theta$. 
\end{proof}

\begin{remark}
    Rather than work inductively, one could argue for Corollary \ref{weak generation} by stretching all
    the curves simultaneously and applying 
    \cite[Prop. 1.37]{GPS2}. 
\end{remark}

\section{Zero objects} \label{zeroes}

\begin{figure}
    \centering
    \begin{subfigure}[b]{0.3\textwidth}
    \centering
     \begin{tikzpicture}
        \node[red] at (0,0) {$*$};
        \draw[thick] (-0.5,-1.5) -- (-0.5,1);
        \draw[thick] (0.5,-1.5) -- (0.5,1);
        \draw[thick] (-0.5, 1) arc (180:0:0.5);
        \draw[thick] (-1,-1.5) -- (-1,1.5);
    \end{tikzpicture} 
    \caption{The Lagrangian $U \times T_i$}
    \label{fig:stabxT1}
    \end{subfigure}
    \begin{subfigure}[b]{0.3\textwidth}
    \centering
     \begin{tikzpicture}
        \node[red] at (0,0) {$*$};
        \draw[thick] (-0.5,-1.5) -- (-0.5,1);
        \draw[thick] (0.5,-1.5) -- (0.5,1);
        \draw[thick] (-0.5, 1) arc (180:0:0.5);
        \draw[thick] (1,-1.5) -- (1,1.5);
    \end{tikzpicture} 
    \caption{The Lagrangian $U \times T_{i+1}$}
    \label{fig:stabxT2}
    \end{subfigure}
    \caption{Two equivalent Lagrangians}
\end{figure}
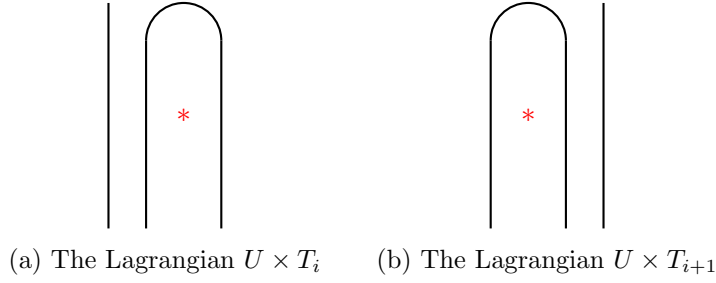

Consider the two Lagrangians $U \times T_i$ and $U \times T_{i+1}$ shown in Figures \ref{fig:stabxT1} and \ref{fig:stabxT2}.

\begin{proposition}\label{slide move map}
    The lowest J-degree in $H^* \mathrm{Hom}(U \times T_{i+1}, U \times T_i)$ has a unique map up to scalar multiple, represented by the intersection point $p_1 q_1 x_p^0 x_q^0$ in Figure \ref{fig:map:stabxT2->stabxT1}.
\end{proposition}

\begin{proof}
    In order to compute the morphisms from $U \times T_{i+1}$ to $U \times T_i$, we will wrap $U \times T_{i+1}$, as shown in Figure \ref{fig:map:stabxT2->stabxT1}. From the figure, the intersection points in the base are $p_1 q_1$ and $p_1 q_2$. Label the intersection points in the fiber as $x_p^i$ and $x_q^i$ for $i \in \mathbb{Z}_{\geq 0}$.
    To count disks, will use the `cylindrical model' of  \cite[Thm 1.5]{ADLSZ}, recalled above in Theorem \ref{sl2 cylindrical model}.  

    In the base $\C^*_y$, there is a unique nontrivial disk, depicted in Figure \ref{fig:map:stabxT2->stabxT1 base}.  A disk mapping to $Sym^2 (\C^*_y)$ must therefore, in the cylindrical model, be described as a map from a disjoint union of two disks, one of which maps to the visible disk in Figure \ref{fig:map:stabxT2->stabxT1 base}, and the other which maps to a point, necessarily $p_1$. 
    Let us now consider the possible corresponding maps $\Phi_u: S \to \mathbb{P}^1_u$.  The component of $S$ on which $\Phi_y = p_1$ must, since it does not pass through either $\mathbf{a}$ or have branch points, be mapped by $\Phi_u$ to the complement of $0, \infty$, and moreover can have its boundary along only two of the depicted curves in $\mathbb{C}^*_u$. The only possibility is for it to map to a point.  Consider the nontrivial component.  As the corresponding $\Phi_C$ again has no branch points, $\Phi_u$ must avoid $\infty$; moreover, $\Phi_y^{-1}(\mathbf{a})$ is a single point.  From this one can see that $\Phi_u$ must have image of the form given in Figure 
\ref{fig:map:stabxT2->stabxT1 fiber}.

    We find $\partial(p_1 q_2 x_p^i x_q^j) = \pm p_1 q_1 x_p^i x_q^{j+1}$ and $\partial(p_1 q_1 x_p^i x_q^j) = 0$. Therefore, $\Hom(U \times T_{i+1}, U \times T_i)$ is generated by $p_1 q_2 x_p^i x_q^0$ for $i \in \mathbb{Z}_{\geq 0}$.
    Using equation \eqref{eqn:J deg combinatorial}, $J(p_1 q_i x_p^{i+1} x_q^j) - J(p_1 q_i x_p^i x_q^j) = 1$.  Thus $p_1 q_1 x_p^0 x_q^0$ is a unique closed and not exact intersection point in the lowest J degree.
\end{proof}

\begin{proposition}\label{slide move}
    The map in Proposition \ref{slide move map} is an isomorphism.
\end{proposition}

\begin{proof} Using Proposition \ref{multicurve cone as surgery}, 
    the two Lagrangians can be resolved as
    \begin{equation}
        U \times T_i \cong \left\{ T_i \times T_{i+1} \xrightarrow{\begin{tikzpicture}[scale=.5]
	\begin{pgfonlayer}{nodelayer}
		\node [style=none] (0) at (0, 0.5) {};
		\node [style=none] (2) at (0, -0.25) {};
		\node [style=none] (4) at (-0.25, -0.25) {};
		\node [style=none] (5) at (0.25, 0.5) {};
		\node [style=none] (6) at (-0.5, -0.25) {};
		\node [style=none] (7) at (-0.5, 0.5) {};
	\end{pgfonlayer}
	\begin{pgfonlayer}{edgelayer}
		\draw [style=su2 puncture strand] (0.center) to (2.center);
		\draw [style=brane, in=90, out=-90, looseness=1.25] (5.center) to (4.center);
		\draw [style=brane, in=90, out=-90, looseness=1.25] (7.center) to (6.center);
	\end{pgfonlayer}
\end{tikzpicture}} T_i \times T_i \right\}
    \end{equation}
    and
     \begin{equation}
        U \times T_{i+1} \cong \left\{ T_{i+1} \times T_{i+1} \xrightarrow{\begin{tikzpicture}[scale=.5]
	\begin{pgfonlayer}{nodelayer}
		\node [style=none] (0) at (-0.25, 0.5) {};
		\node [style=none] (2) at (-0.25, -0.25) {};
		\node [style=none] (4) at (-0.5, -0.25) {};
		\node [style=none] (5) at (0, 0.5) {};
		\node [style=none] (6) at (0.25, -0.25) {};
		\node [style=none] (7) at (0.25, 0.5) {};
	\end{pgfonlayer}
	\begin{pgfonlayer}{edgelayer}
		\draw [style=su2 puncture strand] (0.center) to (2.center);
		\draw [style=brane, in=90, out=-90, looseness=1.25] (5.center) to (4.center);
		\draw [style=brane, in=90, out=-90, looseness=1.25] (7.center) to (6.center);
	\end{pgfonlayer}
\end{tikzpicture}} T_i \times T_{i+1} \right\}
    \end{equation}
    where the lowest term in is homological degree zero. 
    
    By Proposition \ref{slide move map}, the lowest J-degree in $H^* \mathrm{Hom}(U \times T_{i+1}, U \times T_i)$ has a unique map up to scalar multiple.  It is straightforward to check that the following is a chain map of minimal J-degree, hence represents this class: 
    \begin{equation}
        f_1 = -\vcenter{\hbox{\begin{tikzpicture}[scale=.5]
	\begin{pgfonlayer}{nodelayer}
		\node [style=none] (0) at (-0.25, 0.5) {};
		\node [style=none] (2) at (-0.25, -0.25) {};
		\node [style=none] (4) at (-0.5, -0.25) {};
		\node [style=none] (5) at (0.25, 0.5) {};
		\node [style=none] (6) at (0, -0.25) {};
		\node [style=none] (7) at (0, 0.5) {};
	\end{pgfonlayer}
	\begin{pgfonlayer}{edgelayer}
		\draw [style=su2 puncture strand] (0.center) to (2.center);
		\draw [style=brane, in=90, out=-90, looseness=1.25] (5.center) to (4.center);
		\draw [style=brane, in=90, out=-90, looseness=1.25] (7.center) to (6.center);
	\end{pgfonlayer}
\end{tikzpicture}}}, \qquad f_0 = -\vcenter{\hbox{\begin{tikzpicture}[scale=.5]
	\begin{pgfonlayer}{nodelayer}
		\node [style=none] (0) at (0, -0.25) {};
		\node [style=none] (2) at (0, 0.5) {};
		\node [style=none] (4) at (0.25, 0.5) {};
		\node [style=none] (5) at (-0.5, -0.25) {};
		\node [style=none] (6) at (-0.25, 0.5) {};
		\node [style=none] (7) at (-0.25, -0.25) {};
	\end{pgfonlayer}
	\begin{pgfonlayer}{edgelayer}
		\draw [style=su2 puncture strand] (0.center) to (2.center);
		\draw [style=brane, in=-90, out=90, looseness=1.25] (5.center) to (4.center);
		\draw [style=brane, in=-90, out=90, looseness=1.25] (7.center) to (6.center);
	\end{pgfonlayer}
\end{tikzpicture}}}.
    \end{equation}

    To see that $f$ is an isomorphism, we need a chain map $g: U \times T_i \to U \times T_{i+1}$ such that $f \circ g$ and $g \circ f$ are both homotopic to the identity map. Let $g$ be
    \begin{align}
        g_1 = \vcenter{\hbox{\begin{tikzpicture}[scale=.5]
	\begin{pgfonlayer}{nodelayer}
		\node [style=none] (0) at (-0.25, -0.25) {};
		\node [style=none] (2) at (-0.25, 0.5) {};
		\node [style=none] (4) at (-0.5, 0.5) {};
		\node [style=none] (5) at (0.25, -0.25) {};
		\node [style=none] (6) at (0, 0.5) {};
		\node [style=none] (7) at (0, -0.25) {};
		\node [style=dot] (8) at (0.175, 0) {};
	\end{pgfonlayer}
	\begin{pgfonlayer}{edgelayer}
		\draw [style=su2 puncture strand] (0.center) to (2.center);
		\draw [style=brane, in=-90, out=90, looseness=1.25] (5.center) to (4.center);
		\draw [style=brane, in=-90, out=90, looseness=1.25] (7.center) to (6.center);
	\end{pgfonlayer}
\end{tikzpicture}}}, \qquad g_0 = \vcenter{\hbox{\begin{tikzpicture}[scale=.5]
	\begin{pgfonlayer}{nodelayer}
		\node [style=none] (0) at (0, 0.5) {};
		\node [style=none] (2) at (0, -0.25) {};
		\node [style=none] (4) at (0.25, -0.25) {};
		\node [style=none] (5) at (-0.5, 0.5) {};
		\node [style=none] (6) at (-0.25, -0.25) {};
		\node [style=none] (7) at (-0.25, 0.5) {};
		\node [style=dot] (8) at (-0.425, 0.25) {};
	\end{pgfonlayer}
	\begin{pgfonlayer}{edgelayer}
		\draw [style=su2 puncture strand] (0.center) to (2.center);
		\draw [style=brane, in=90, out=-90, looseness=1.25] (5.center) to (4.center);
		\draw [style=brane, in=90, out=-90, looseness=1.25] (7.center) to (6.center);
	\end{pgfonlayer}
\end{tikzpicture}}}.
    \end{align}
    Using the chain homotopy $h: U \times T_{i+1} \to U \times T_{i+1}[1]$ given by $h_1 = \vcenter{\hbox{\begin{tikzpicture}[scale=.5]
	\begin{pgfonlayer}{nodelayer}
		\node [style=none] (0) at (-0.25, -0.25) {};
		\node [style=none] (2) at (-0.25, 0.5) {};
		\node [style=none] (4) at (-0.5, 0.5) {};
		\node [style=none] (5) at (0.25, -0.25) {};
		\node [style=none] (6) at (0, 0.5) {};
		\node [style=none] (7) at (0, -0.25) {};
	\end{pgfonlayer}
	\begin{pgfonlayer}{edgelayer}
		\draw [style=su2 puncture strand] (0.center) to (2.center);
		\draw [style=brane, in=-90, out=90, looseness=1.25] (5.center) to (4.center);
		\draw [style=brane, in=-90, out=90, looseness=1.25] (7.center) to (6.center);
	\end{pgfonlayer}
\end{tikzpicture}}}$, it follows that $g \circ f \sim \mathrm{id}_{U \times T_{i+1}}$. Similarly, using $s: U \times T_i \to U \times T_i$ given by $s_1 = \vcenter{\hbox{\begin{tikzpicture}[scale=.5]
	\begin{pgfonlayer}{nodelayer}
		\node [style=none] (0) at (0, 0.5) {};
		\node [style=none] (2) at (0, -0.25) {};
		\node [style=none] (4) at (0.25, -0.25) {};
		\node [style=none] (5) at (-0.5, 0.5) {};
		\node [style=none] (6) at (-0.25, -0.25) {};
		\node [style=none] (7) at (-0.25, 0.5) {};
	\end{pgfonlayer}
	\begin{pgfonlayer}{edgelayer}
		\draw [style=su2 puncture strand] (0.center) to (2.center);
		\draw [style=brane, in=90, out=-90, looseness=1.25] (5.center) to (4.center);
		\draw [style=brane, in=90, out=-90, looseness=1.25] (7.center) to (6.center);
	\end{pgfonlayer}
\end{tikzpicture}}}$, it follows that $f \circ g \sim \mathrm{id}_{U \times T_i}$.
\end{proof}

\begin{corollary}\label{double stab is zero}
    The Lagrangian shown in Figure \ref{fig:double_stab} is a zero object.
\end{corollary}

\begin{proof}
    This Lagrangian is isomorphic to the cone of the isomorphism $f$ from Proposition \ref{slide move}.
    
    For an alternate argument, can also explicitly exhibit  disks showing the identity is exact.  One can show $\partial(u_1 v_2 1_u 1_v) = \pm 1_u x_v u_1 v_1 \pm 1_p 1_q p_1 q_1$.  The first of these terms is similar to the disk in Figure \ref{fig:map:stabxT2->stabxT1}; the second is depicted in Figure \ref{fig:selfhomUxU}.  (To actually prove that there is a unique such disk can be done by an argument similar to that we give later in Proposition \ref{U and I} below.)  There is a similar disk showing $\partial(x_p 1_q p_1 q_2) = \pm 1_u x_v u_1 v_1$.  Adding these shows the identity is exact. 
\end{proof}

\begin{corollary}\label{general zero objects}
    Any Lagrangian that formed from the product of the multicurve shown in Figure \ref{fig:double_stab} with another multicurve is zero.
\end{corollary}

\begin{proof}
    Denote the Lagrangian in Figure \ref{fig:double_stab} by $U \times U$. Consider a multicurve Lagrangian $\widetilde{\mathbb{U}}$ given by adding to $U \times U$ any number of $T_\theta$ which do not intersect $U \times U$. The presence of these additional $T$ do not affect the proofs of Proposition \ref{slide move} and Corollary \ref{double stab is zero}, so $\widetilde{\mathbb{U}}$ is still a zero object.
    Moreover, any Lagrangian formed as the product of $U \times U$ with another multicurve can resolved as a complex consisting entirely of subcomplexes given by $\widetilde{\mathbb{U}}$ objects, hence is also a zero object. 

    Alternatively, the same disks as in Corollary \ref{double stab is zero} exhibit a primitive of the identity morphism. 
\end{proof}

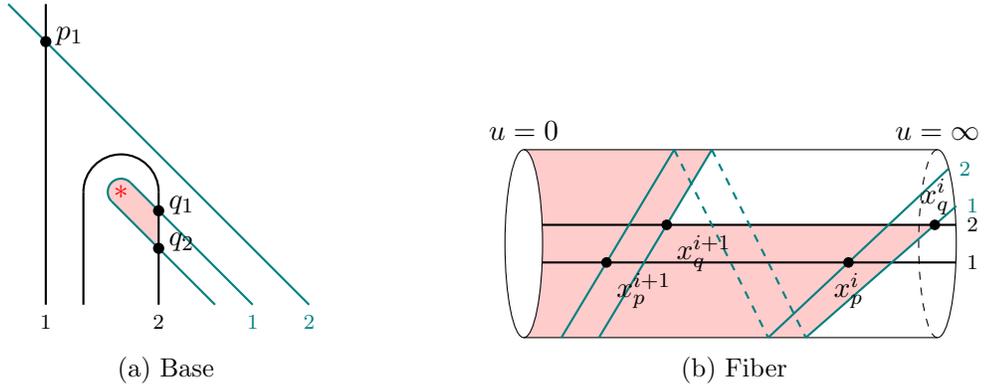
\begin{figure}
    \centering
    \begin{subfigure}[b]{0.45\textwidth}
        \centering
    \begin{tikzpicture}
        \node[red] at (0,0) {$*$};
        \draw[thick] (-0.5,-1.5) -- (-0.5,0);
        \draw[thick] (0.5,-1.5) node[below] {$\scriptstyle 2$} -- (0.5,0);
        \draw[thick] (-0.5,0) arc (180:0:0.5);
        \draw[thick, teal] (1.25,-1.5) -- (-.125,-.125);
        \draw[thick, teal] (1.75,-1.5) node[below] {$\scriptstyle 1$} -- (.125,.125);
        \draw[thick, teal] (-.125, -.125) arc (225:45:{sqrt(2)/8});
        \draw[thick] (-1,-1.5) node[below] {$\scriptstyle 1$} -- (-1,2.5);
        \draw[thick, teal] (2.5,-1.5) node[below] {$\scriptstyle 2$} -- (-1.5,2.5);
        \node[label={[right]$p_1$}] at (-1,2) [circle,fill,inner sep=1.5pt]{};
        \node[label={[right]$q_2$}] (q2) at (.5,-.75) [circle,fill,inner sep=1.5pt]{};
        \node[label={[right]$q_1$}] (q1) at (.5,-.25) [circle,fill,inner sep=1.5pt]{};
        \begin{scope}[on background layer] 
            \fill [red!20] (q1.center) -- (q2.center) -- (-.125, -.125) arc (225:45:{sqrt(2)/8}) -- cycle;
        \end{scope}
    \end{tikzpicture} 
        \caption{Base}
        \label{fig:map:stabxT2->stabxT1 base}
    \end{subfigure}
    \begin{subfigure}[b]{0.45\textwidth}
        \centering
    \begin{tikzpicture}
		\draw (-5.5,2.5) -- (0,2.5);
		\draw (-5.5,0) -- (0,0);
		\draw (-5.5,1.25) ellipse (0.25 cm and 1.25 cm);
		\draw (0,0) arc[
				start angle=-90,
				end angle=90,
				x radius=0.25cm,
				y radius=1.25cm
				];
	    \draw[dashed] (0,2.5) arc[
				start angle=90,
				end angle=270,
				x radius=0.25cm,
				y radius=1.25cm
				];
		\draw[thick, name path=t1] (.25,1) node[right] {$\scriptstyle 1$} -- (-5.25,1);
		\draw[thick, name path=t2] (.25,1.5) node[right] {$\scriptstyle 2$} -- (-5.25,1.5);
        \draw[thick, teal, name path=w1] (0.25,1.75) node[right] {$\scriptstyle 1$} -- (-1.75,0);
        \draw[thick, teal, name path=w2] (0.15,2.25) node[right] {$\scriptstyle 2$} -- (-2.25,0);
        \draw[thick, teal, dashed] (-3,2.5) -- (-1.75,0);
        \draw[thick, teal, dashed] (-3.5,2.5) -- (-2.25,0);
        \draw[thick, teal, name path=w3] (-3,2.5) -- (-4.5,0);
        \draw[thick, teal, name path=w4] (-3.5,2.5) -- (-5,0);
        \fill[name intersections={of=t2 and w1, by=i1}] (intersection-1) circle (2pt) node[above] {$x_q^i$};
        \fill[name intersections={of=t1 and w2}] (intersection-1) circle (2pt) node[below] {$x_p^i$};
        \fill[name intersections={of=t2 and w3, by=i2}] (intersection-1) circle (2pt) node[below right] {$x_q^{i+1}$};
        \fill[name intersections={of=t1 and w4}] (intersection-1) circle (2pt) node[below right] {$x_p^{i+1}$};
        \node[above] at (0,2.5) {$u=\infty$};
        \node[above] at (-5.5,2.5) {$u=0$};
        \begin{scope}[on background layer] 
            \fill [red!20] (i2) -- (i1) -- (-1.75,0) -- (-5.5,0) arc (-90:90:.25cm and 1.25cm) -- (-3,2.5) -- cycle;
        \end{scope}
    \end{tikzpicture}
        \caption{Fiber}
        \label{fig:map:stabxT2->stabxT1 fiber}
    \end{subfigure}
    \caption{Cylindrical model presentation of the unique disk in this geometry.  Note that at the $u = \infty$ end, the Lagrangians stop wrapping after $i=0$.}
    \label{fig:map:stabxT2->stabxT1}
\end{figure}

\begin{figure}
    \centering
    \begin{subfigure}[b]{0.3\textwidth}
    \centering
     \begin{tikzpicture}
        \node[red] at (0,0) {$*$};
        \draw[thick] (-0.5,-1.5) -- (-0.5,0.5);
        \draw[thick] (0.5,-1.5) -- (0.5,0.5);
        \draw[thick] (-0.5, 0.5) arc (180:0:0.5);
        \draw[thick] (-1,-1.5) -- (-1,1);
        \draw[thick] (-0.75,-1.5) -- (-0.75,1);
        \draw[thick] (-1, 1) arc (180:0:0.125);
    \end{tikzpicture}
    \caption{Obviously a zero object}
    \label{fig:double_stab_slide}   
    \end{subfigure}
    \begin{subfigure}[b]{0.3\textwidth}
    \centering
     \begin{tikzpicture}
        \node[red] at (0,0) {$*$};
        \draw[thick] (-0.5,-1.5) -- (-0.5,0.5);
        \draw[thick] (0.5,-1.5) -- (0.5,0.5);
        \draw[thick] (-0.5, 0.5) arc (180:0:0.5);
        \draw[thick] (-1,-1.5) -- (-1,0.5);
        \draw[thick] (1,-1.5) -- (1,0.5);
        \draw[thick] (-1, 0.5) arc (180:0:1);
    \end{tikzpicture}
    \caption{Not obviously a zero object}
    \label{fig:double_stab}   
    \end{subfigure}
    \caption{Two zero objects}
\end{figure}
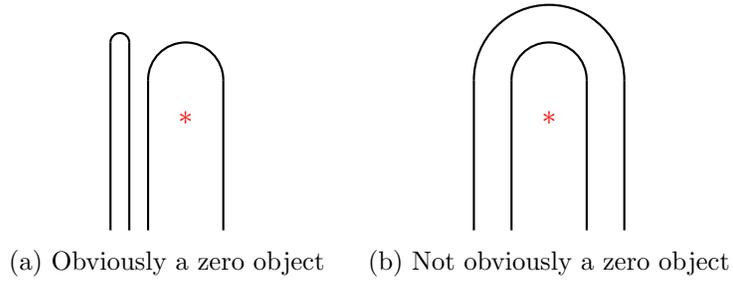

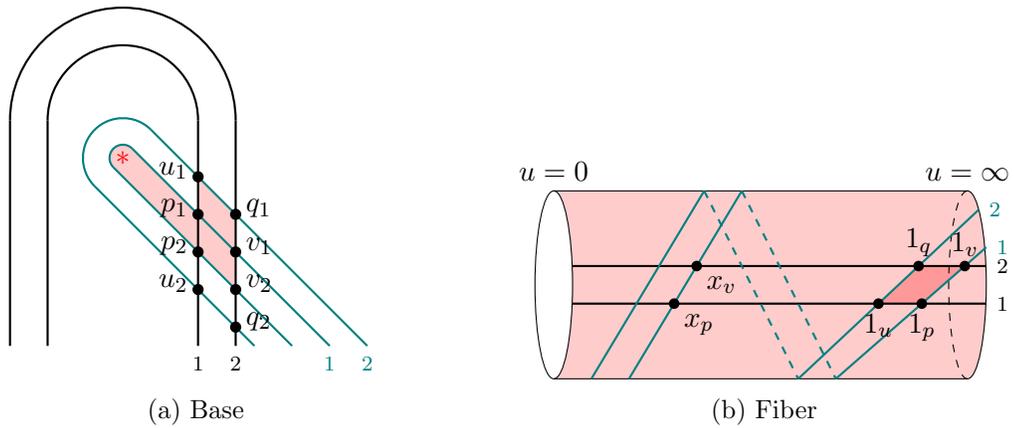
\begin{figure}
    \centering
    \begin{subfigure}[b]{0.45\textwidth}
        \centering
    \begin{tikzpicture}
        \node[red] at (0,0) {$*$};
        \draw[thick] (-1.5,-2.5) -- (-1.5,0.5);
        \draw[thick] (1.5,-2.5) node[below] {$\scriptstyle 2$} -- (1.5,0.5);
        \draw[thick] (-1.5, 0.5) arc (180:0:1.5);
        \draw[thick] (-1,-2.5) -- (-1,0.5);
        \draw[thick] (1,-2.5) node[below] {$\scriptstyle 1$} -- (1,0.5);
        \draw[thick] (-1, 0.5) arc (180:0:1);
        \draw[thick, teal] (2.25,-2.5) -- (-.125,-.125);
        \draw[thick, teal] (2.75,-2.5) node[below] {$\scriptstyle 1$} -- (.125,.125);
        \draw[thick, teal] (-.125, -.125) arc (225:45:{sqrt(2)/8});
        \draw[thick, teal] (1.75,-2.5) -- (-.375,-.375);
        \draw[thick, teal] (3.25,-2.5) node[below] {$\scriptstyle 2$} -- (.375,.375);
        \draw[thick, teal] (-.375, -.375) arc (225:45:{3*sqrt(2)/8});
        \node[label={[left]$u_1$}] (u1) at (1,-.25) [circle,fill,inner sep=1.5pt]{};
        \node[label={[left]$p_1$}] (p1) at (1,-.75) [circle,fill,inner sep=1.5pt]{};
        \node[label={[left]$p_2$}] at (1,-1.25) [circle,fill,inner sep=1.5pt]{};
        \node[label={[left]$u_2$}] at (1,-1.75) [circle,fill,inner sep=1.5pt]{};
        \node[label={[right]$q_1$}] (q1) at (1.5,-.75) [circle,fill,inner sep=1.5pt]{};
        \node[label={[right]$v_1$}] at (1.5,-1.25) [circle,fill,inner sep=1.5pt]{};
        \node[label={[right]$v_2$}] (v2) at (1.5,-1.75) [circle,fill,inner sep=1.5pt]{};
        \node[label={[right]$q_2$}] at (1.5,-2.25) [circle,fill,inner sep=1.5pt]{};
        \begin{scope}[on background layer] 
            \fill [red!20] (u1.center) -- (q1.center) -- (v2.center) -- (-.125, -.125) arc (225:45:{sqrt(2)/8}) -- (.125,.125) -- (p1.center) -- cycle;
        \end{scope}
    \end{tikzpicture}
    \caption{Base}
    \end{subfigure}
    \begin{subfigure}[b]{0.45\textwidth}
        \centering
    \begin{tikzpicture}
		\draw (-5.5,2.5) -- (0,2.5);
		\draw (-5.5,0) -- (0,0);
		\draw (-5.5,1.25) ellipse (0.25 cm and 1.25 cm);
		\draw (0,0) arc[
				start angle=-90,
				end angle=90,
				x radius=0.25cm,
				y radius=1.25cm
				];
	    \draw[dashed] (0,2.5) arc[
				start angle=90,
				end angle=270,
				x radius=0.25cm,
				y radius=1.25cm
				];
		\draw[thick, name path=t1] (.25,1) node[right] {$\scriptstyle 1$} -- (-5.25,1);
		\draw[thick, name path=t2] (.25,1.5) node[right] {$\scriptstyle 2$} -- (-5.25,1.5);
        \draw[thick, teal, name path=w1] (0.25,1.75) node[right] {$\scriptstyle 1$} -- (-1.75,0);
        \draw[thick, teal, name path=w2] (0.15,2.25) node[right] {$\scriptstyle 2$} -- (-2.25,0);
        \draw[thick, teal, dashed] (-3,2.5) -- (-1.75,0);
        \draw[thick, teal, dashed] (-3.5,2.5) -- (-2.25,0);
        \draw[thick, teal, name path=w3] (-3,2.5) -- (-4.5,0);
        \draw[thick, teal] (-3.5,2.5) -- (-5,0);
        \fill[name intersections={of=t1 and w1, by=p}] (intersection-1) circle (2pt) node[below] {$1_p$};
        \fill[name intersections={of=t2 and w1, by=v}] (intersection-1) circle (2pt) node[above] {$1_v$};
        \fill[name intersections={of=t1 and w2, by=u}] (intersection-1) circle (2pt) node[below] {$1_u$};
        \fill[name intersections={of=t2 and w2, by=q}] (intersection-1) circle (2pt) node[above] {$1_q$};
        \fill[name intersections={of=t1 and w3}] (intersection-1) circle (2pt) node[below right] {$x_p$};
        \fill[name intersections={of=t2 and w3}] (intersection-1) circle (2pt) node[below right] {$x_v$};
        \node[above] at (0,2.5) {$u=\infty$};
        \node[above] at (-5.5,2.5) {$u=0$};
        \begin{scope}[on background layer] 
            \fill [red!20] (-5.5,0) arc (-90:90:.25cm and 1.25cm) -- (0,2.5) arc (90:-90:.25cm and 1.25cm) -- cycle;
            \fill [red!40] (u.center) -- (q.center) -- (v.center) -- (p.center) -- cycle;
        \end{scope}
    \end{tikzpicture}
    \caption{Fiber}
    \label{fig:selfhomUxUfiber}
    \end{subfigure}
    \caption{The disk from $u_1 v_2 1_u 1_v$ to $p_1 q_1 1_p 1_q$}
    \label{fig:selfhomUxU}
\end{figure}

\section{Easy braiding object, revisited} \label{more lambda object}

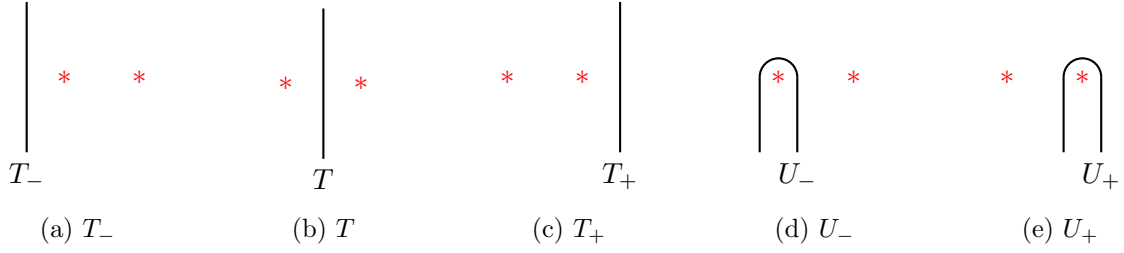
\begin{figure}
    \centering
    \begin{subfigure}[b]{0.19\textwidth}
        \centering
        \begin{tikzpicture}[scale=.5]
            \node[red] at (0,0) {$*$};
            \node[red] at (2,0) {$*$};
            \draw[thick] (-1,-2) node[below] {$T_-$} -- (-1,2);
        \end{tikzpicture}
        \caption{$T_-$}
    \end{subfigure}
    \begin{subfigure}[b]{0.19\textwidth}
        \centering
        \begin{tikzpicture}[scale=.5]
            \node[red] at (0,0) {$*$};
            \node[red] at (2,0) {$*$};
            \draw[thick] (1,-2) node[below] {$T$} -- (1,2);
        \end{tikzpicture}
        \caption{$T$}
    \end{subfigure}
    \begin{subfigure}[b]{0.19\textwidth}
        \centering
        \begin{tikzpicture}[scale=.5]
            \node[red] at (0,0) {$*$};
            \node[red] at (2,0) {$*$};
            \draw[thick] (3,-2) node[below] {$T_+$} -- (3,2);
        \end{tikzpicture}
        \caption{$T_+$}
    \end{subfigure}
    \begin{subfigure}[b]{0.19\textwidth}
        \centering
        \begin{tikzpicture}[scale=.5]
            \node[red] at (0,0) {$*$};
            \node[red] at (2,0) {$*$};
            \draw[thick] (-.5,-2) -- (-.5,0);
            \draw[thick] (.5,-2) node[below] {$U_-$} -- (.5,0);
            \draw[thick] (-.5,0) arc (180:0:.5);
        \end{tikzpicture}
        \caption{$U_-$}
    \end{subfigure}
    \begin{subfigure}[b]{0.19\textwidth}
        \centering
        \begin{tikzpicture}[scale=.5]
            \node[red] at (0,0) {$*$};
            \node[red] at (2,0) {$*$};
            \draw[thick] (1.5,-2) -- (1.5,0);
            \draw[thick] (2.5,-2) node[below] {$U_+$} -- (2.5,0);
            \draw[thick] (1.5,0) arc (180:0:.5);
        \end{tikzpicture}
        \caption{$U_+$}
    \end{subfigure}
    \caption{The curves $T_-$, $T$, $T_+$, $U_-$ and $U_+$}
    \label{fig:Ts and Us}
\end{figure}

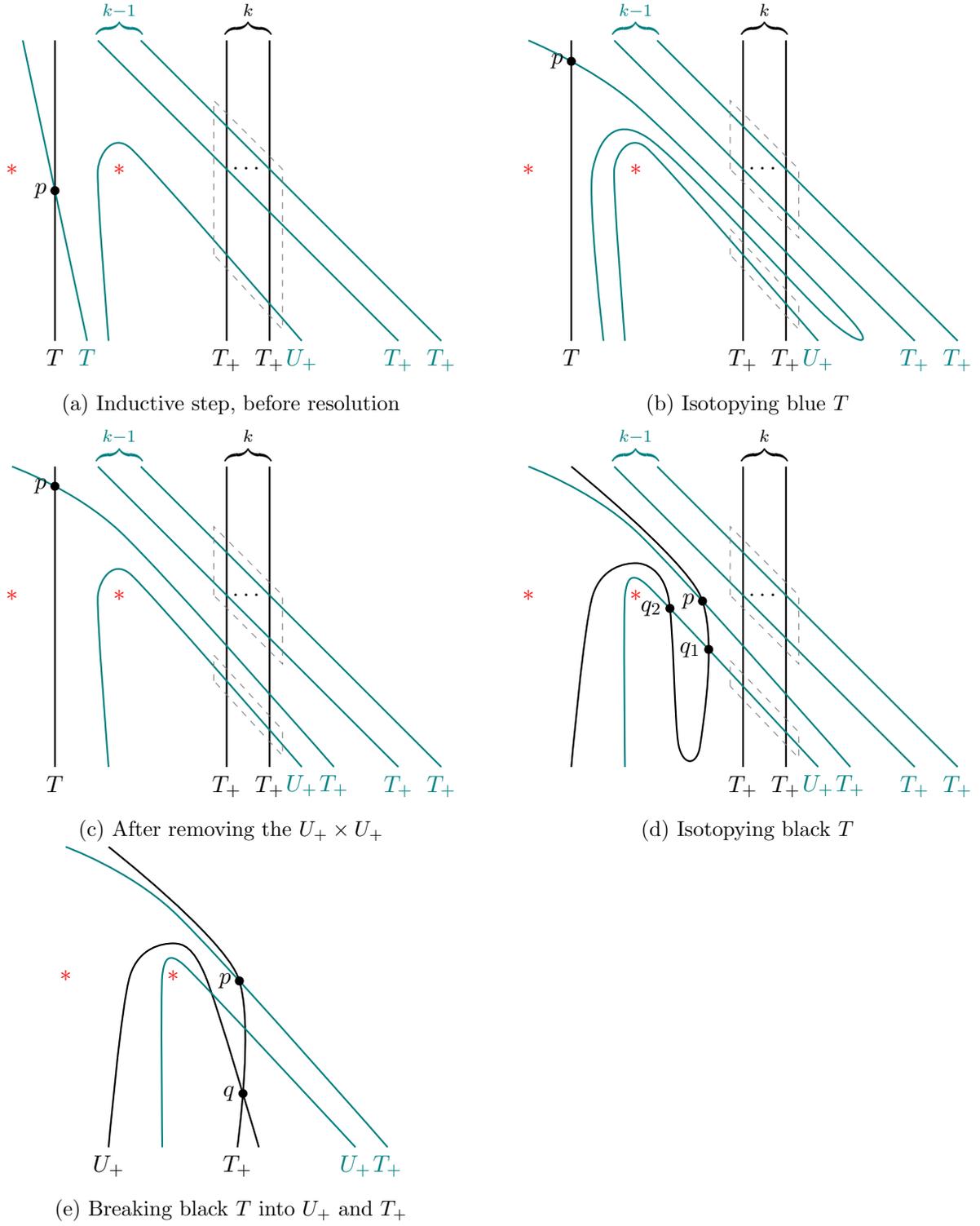
\begin{figure}
\begin{subfigure}[b]{0.49\textwidth}
    \centering
    \begin{tikzpicture}[scale=.7]
        \node[red] at (0,0) {$*$};
        \node[red] at (-2.5,0) {$*$};
        \draw[thick, name path=t31] (2.5,-4) node[below] {$T_+$} -- (2.5,3);
        \node at (3,0) {$\cdots$};
        \draw[thick, name path=t32] (3.5,-4) node[below] {$T_+$} -- (3.5,3);
        \draw[thick, name path=t2] (-1.5,-4) node[below] {$T$} -- (-1.5,3);
        \node[teal, below] at (4.25,-4) {$U_+$};
        \draw[thick, teal, name path = uw] plot [smooth, tension=0.3] coordinates {(-.25,-4) (-.5,0) (.25,.5) (4.25,-4)};
        \draw[thick, teal, name path=t31w] (6.5,-4) node[below] {$T_+$} -- (-.5,3);
        \draw[thick, teal, name path=t32w] (7.5,-4) node[below] {$T_+$} -- (.5,3);
        \draw[thick, teal, name path=t2w] (-.75,-4) node[below] {$T$} -- (-2.25,3);
        \node at (3,3.5) {$\overbrace{\hspace{.75cm}}^k$};
        \node[teal] at (0,3.5) {$\overbrace{\hspace{.75cm}}^{k-1}$};
        \fill[name intersections={of=t2 and t2w}] (intersection-1) circle (3pt) node[left] {$p$};
        \draw[dashed, gray,
            name intersections={of=t31 and uw, by=i1},
            name intersections={of=t31 and t32w, by=i2},
            name intersections={of=t32 and t32w, by=i3},
            name intersections={of=t32 and uw, by=i4}] ([shift={(-.3cm,0cm)}]i1) -- ([shift={(-.3cm,.6cm)}]i2) -- ([shift={(.3cm,0cm)}]i3) -- ([shift={(.3cm,-.6cm)}]i4) -- cycle;
    \end{tikzpicture}  
    \caption{Inductive step, before resolution}
    \label{fig:addingT2}
\end{subfigure}
\begin{subfigure}[b]{0.49\textwidth}
    \centering
    \begin{tikzpicture}[scale=.7]
        \node[red] at (0,0) {$*$};
        \node[red] at (-2.5,0) {$*$};
        \draw[thick, name path=t31] (2.5,-4) node[below] {$T_+$} -- (2.5,3);
        \node at (3,0) {$\cdots$};
        \draw[thick, name path=t32] (3.5,-4) node[below] {$T_+$} -- (3.5,3);
        \draw[thick, name path=t2] (-1.5,-4) node[below] {$T$} -- (-1.5,3);
        \node[teal, below] at (4.25,-4) {$U_+$};
        \draw[thick, teal, name path = uw] plot [smooth, tension=0.3] coordinates {(-.25,-4) (-.5,0) (.25,.5) (4.25,-4)};
        \draw[thick, teal, name path=t31w] (6.5,-4) node[below] {$T_+$} -- (-.5,3);
        \draw[thick, teal, name path=t32w] (7.5,-4) node[below] {$T_+$} -- (.5,3);
        \draw[thick, teal, name path=t2w] plot [smooth, tension=0.5] coordinates {(-.75,-4) (-1,0) (.25,.75) (4.5,-3.5) (5,-3.5) (0,1.5) (-2.5,3)};
        \fill[name intersections={of=t2 and t2w}] (intersection-1) circle (3pt) node[left] {$p$};
        \node at (3,3.5) {$\overbrace{\hspace{.75cm}}^k$};
        \node[teal] at (0,3.5) {$\overbrace{\hspace{.75cm}}^{k-1}$};
        \draw[dashed, gray,
            name intersections={of=t31 and t31w, by=i1},
            name intersections={of=t31 and t32w, by=i2},
            name intersections={of=t32 and t32w, by=i3},
            name intersections={of=t32 and t31w, by=i4}] ([shift={(-.3cm,0cm)}]i1) -- ([shift={(-.3cm,.6cm)}]i2) -- ([shift={(.3cm,0cm)}]i3) -- ([shift={(.3cm,-.6cm)}]i4) -- cycle;
            \draw[dashed, gray,
            name intersections={of=t31 and uw, by=i5},
            name intersections={of=t32 and uw, by=i6}] ([shift={(-.3cm,0cm)}]i5) -- ([shift={(-.3cm,.6cm)}]i5) -- ([shift={(.3cm,0cm)}]i6) -- ([shift={(.3cm,-.6cm)}]i6) -- cycle;
    \end{tikzpicture}  
    \caption{Isotopying blue $T$}
    \label{fig:isotopyT2blue}
\end{subfigure}
\begin{subfigure}[b]{0.49\textwidth}
    \centering
    \begin{tikzpicture}[scale=.7]
        \node[red] at (0,0) {$*$};
        \node[red] at (-2.5,0) {$*$};
        \draw[thick, name path=t31] (2.5,-4) node[below] {$T_+$} -- (2.5,3);
        \node at (3,0) {$\cdots$};
        \draw[thick, name path=t32] (3.5,-4) node[below] {$T_+$} -- (3.5,3);
        \draw[thick, name path=t2] (-1.5,-4) node[below] {$T$} -- (-1.5,3);
        \node[teal, below] at (4.25,-4) {$U_+$};
        \draw[thick, teal, name path = uw] plot [smooth, tension=0.3] coordinates {(-.25,-4) (-.5,0) (.25,.5) (4.25,-4)};
        \draw[thick, teal, name path=t31w] (6.5,-4) node[below] {$T_+$} -- (-.5,3);
        \draw[thick, teal, name path=t32w] (7.5,-4) node[below] {$T_+$} -- (.5,3);
        \node[teal, below] at (5,-4) {$T_+$};
        \draw[thick, teal, name path=t2w] plot [smooth, tension=0.5] coordinates {(5,-4) (0,1.5) (-2.5,3)};
        \fill[name intersections={of=t2 and t2w}] (intersection-1) circle (3pt) node[left] {$p$};
        \node at (3,3.5) {$\overbrace{\hspace{.75cm}}^k$};
        \node[teal] at (0,3.5) {$\overbrace{\hspace{.75cm}}^{k-1}$};
        \draw[dashed, gray,
            name intersections={of=t31 and t31w, by=i1},
            name intersections={of=t31 and t32w, by=i2},
            name intersections={of=t32 and t32w, by=i3},
            name intersections={of=t32 and t31w, by=i4}] ([shift={(-.3cm,0cm)}]i1) -- ([shift={(-.3cm,.6cm)}]i2) -- ([shift={(.3cm,0cm)}]i3) -- ([shift={(.3cm,-.6cm)}]i4) -- cycle;
            \draw[dashed, gray,
            name intersections={of=t31 and uw, by=i5},
            name intersections={of=t32 and uw, by=i6}] ([shift={(-.3cm,0cm)}]i5) -- ([shift={(-.3cm,.6cm)}]i5) -- ([shift={(.3cm,0cm)}]i6) -- ([shift={(.3cm,-.6cm)}]i6) -- cycle;
    \end{tikzpicture}  
    \caption{After removing the $U_+ \times U_+$}
    \label{fig:removingblueUxU}
\end{subfigure}
\begin{subfigure}[b]{0.49\textwidth}
    \centering
    \begin{tikzpicture}[scale=.7]
        \node[red] at (0,0) {$*$};
        \node[red] at (-2.5,0) {$*$};
        \draw[thick, name path=t31] (2.5,-4) node[below] {$T_+$} -- (2.5,3);
        \node at (3,0) {$\cdots$};
        \draw[thick, name path=t32] (3.5,-4) node[below] {$T_+$} -- (3.5,3);
        \draw[thick, name path=t2] plot [smooth, tension=.5] coordinates {(-1.5,-4) (-1,0) (0,.75) (.75,0) (1,-3.5) (1.5, -3.5) (1.5,0) (-1.5,3)};
        \node[teal, below] at (4.25,-4) {$U_+$};
        \draw[thick, teal, name path = uw] plot [smooth, tension=0.3] coordinates {(-.25,-4) (-.25,0) (.25,.25) (4.25,-4)};
        \draw[thick, teal, name path=t31w] (6.5,-4) node[below] {$T_+$} -- (-.5,3);
        \draw[thick, teal, name path=t32w] (7.5,-4) node[below] {$T_+$} -- (.5,3);
        \node[teal, below] at (5,-4) {$T_+$};
        \draw[thick, teal, name path=t2w] plot [smooth, tension=0.5] coordinates {(5,-4) (0,1.5) (-2.5,3)};
        \fill[name intersections={of=t2 and t2w}] (intersection-1) circle (3pt) node[left] {$p$};
        \fill[name intersections={of=t2 and uw}] (intersection-1) circle (3pt) node[left] {$q_2$};
        \fill[name intersections={of=t2 and uw}] (intersection-2) circle (3pt) node[left] {$q_1$};
        \node at (3,3.5) {$\overbrace{\hspace{.75cm}}^k$};
        \node[teal] at (0,3.5) {$\overbrace{\hspace{.75cm}}^{k-1}$};
        \draw[dashed, gray,
            name intersections={of=t31 and t31w, by=i1},
            name intersections={of=t31 and t32w, by=i2},
            name intersections={of=t32 and t32w, by=i3},
            name intersections={of=t32 and t31w, by=i4}] ([shift={(-.3cm,0cm)}]i1) -- ([shift={(-.3cm,.6cm)}]i2) -- ([shift={(.3cm,0cm)}]i3) -- ([shift={(.3cm,-.6cm)}]i4) -- cycle;
            \draw[dashed, gray,
            name intersections={of=t31 and uw, by=i5},
            name intersections={of=t32 and uw, by=i6}] ([shift={(-.3cm,0cm)}]i5) -- ([shift={(-.3cm,.6cm)}]i5) -- ([shift={(.3cm,0cm)}]i6) -- ([shift={(.3cm,-.6cm)}]i6) -- cycle;
    \end{tikzpicture}  
    \caption{Isotopying black $T$}
    \label{fig:isotopyT2black}
\end{subfigure}
\begin{subfigure}[b]{0.49\textwidth}
    \centering
    \begin{tikzpicture}[scale=.7]
        \node[red] at (0,0) {$*$};
        \node[red] at (-2.5,0) {$*$};
        \draw[thick, name path=u2] plot [smooth, tension=.5] coordinates {(-1.5,-4) (-1,0) (0,.75) (.75,0) (2, -4)};
        \node[below] at (-1.5,-4) {$U_+$};
        \draw[thick, name path=t2] plot [smooth, tension=.5] coordinates {(1.5, -4) (1.5,0) (-1.5,3)};
        \node[below] at (1.5,-4) {$T_+$};
        \node[teal, below] at (4.25,-4) {$U_+$};
        \draw[thick, teal] plot [smooth, tension=0.3] coordinates {(-.25,-4) (-.25,0) (.25,.25) (4.25,-4)};
        \node[teal, below] at (5,-4) {$T_+$};
        \draw[thick, teal, name path=t2w] plot [smooth, tension=0.5] coordinates {(5,-4) (0,1.5) (-2.5,3)};
        \fill[name intersections={of=t2 and t2w}] (intersection-1) circle (3pt) node[left] {$p$};
        \fill[name intersections={of=t2 and u2}] (intersection-1) circle (3pt) node[left] {$q$};
    \end{tikzpicture}  
    \caption{Breaking black $T$ into $U_+$ and $T_+$}
    \label{fig:breakT2black}
\end{subfigure}
\caption{Resolving the $T$ branes in equation \eqref{eqn:extra_T2}}
\label{fig:T2res}
\end{figure}

\begin{proposition} \label{T2 res prop}
    The object $\Lambda_n$ defined in equation \eqref{eqn:T2_res} is in fact isomorphic to $\theta^n$. 
\end{proposition}

\begin{proof}
    In principle, one could try and prove the statement directly with KLRW diagramatics.  Instead, we 
    will the embedding of Theorem \ref{five authors theorem} and argue geometrically using 
    the methods of Corollary \ref{weak generation} and the vanishing result of Corollary \ref{general zero objects}.  

    We will write $T_-, T, T_+$ for the curves corresponding to the $\theta_-, \theta, \theta_+$. 
    Trading the cones in the definition of $\upsilon_+$ and $\upsilon_-$    
    for geometric gluings, we see that their images are the $U_+$ and $U_-$ in Figure \ref{fig:Ts and Us}.
    
    The image of $\Lambda_n$ is then a complex of products of $U_+$ and $T_+$ Lagrangians. 
    We will show how to resolve the components of $T^n$ one at a time to get this complex.

    We have $T \cong \{U^+ \to T^+ \}$, and correspondingly 
    \begin{align}
        T^n &\cong
        \bigg\{ T^{n-1} \times T_+
        \xrightarrow{\begin{pmatrix} id_{T^{n-1}} \times \vcenter{\hbox{\begin{tikzpicture}[scale=.5]
	\begin{pgfonlayer}{nodelayer}
		\node [style=none] (0) at (-0.25, 0.5) {};
		\node [style=none] (1) at (0.25, 0.5) {};
		\node [style=none] (2) at (-0.25, -0.25) {};
		\node [style=none] (3) at (0.25, -0.25) {};
		\node [style=none] (4) at (0, -0.25) {};
		\node [style=none] (5) at (0, 0.5) {};
	\end{pgfonlayer}
	\begin{pgfonlayer}{edgelayer}
		\draw [style=su2 puncture strand] (0.center) to (2.center);
		\draw [style=brane] (3.center) to (1.center);
		\draw [style=su2 puncture strand, in=90, out=-90, looseness=1.25] (5.center) to (4.center);
	\end{pgfonlayer}
\end{tikzpicture}}} & id_{T^{n-1}} \times \vcenter{\hbox{\begin{tikzpicture}[scale=.5]
	\begin{pgfonlayer}{nodelayer}
		\node [style=none] (0) at (-0.25, 0.5) {};
		\node [style=none] (1) at (0.25, 0.5) {};
		\node [style=none] (2) at (-0.25, -0.25) {};
		\node [style=none] (3) at (0.25, -0.25) {};
		\node [style=none] (4) at (0, -0.25) {};
		\node [style=none] (5) at (0.5, 0.5) {};
	\end{pgfonlayer}
	\begin{pgfonlayer}{edgelayer}
		\draw [style=su2 puncture strand] (0.center) to (2.center);
		\draw [style=su2 puncture strand] (3.center) to (1.center);
		\draw [style=brane, in=90, out=-90, looseness=1.25] (5.center) to (4.center);
	\end{pgfonlayer}
\end{tikzpicture}}} \end{pmatrix}}
        T^{n-1} \times T_+ \oplus  T^{n-1} \times T \bigg\} \\
        &\cong
        \bigg\{ T^{n-1} \times U_+ \xrightarrow{id_{T^{n-1}} \times \vcenter{\hbox{\begin{tikzpicture}[scale=.5]
	\begin{pgfonlayer}{nodelayer}
		\node [style=none] (0) at (-0.25, 0.5) {};
		\node [style=none] (1) at (0.25, 0.5) {};
		\node [style=none] (2) at (-0.25, -0.25) {};
		\node [style=none] (3) at (0.25, -0.25) {};
		\node [style=none] (4) at (0, -0.25) {};
		\node [style=none] (5) at (0, 0.5) {};
	\end{pgfonlayer}
	\begin{pgfonlayer}{edgelayer}
		\draw [style=su2 puncture strand] (0.center) to (2.center);
		\draw [style=brane] (3.center) to (1.center);
		\draw [style=su2 puncture strand, in=90, out=-90, looseness=1.25] (5.center) to (4.center);
	\end{pgfonlayer}
\end{tikzpicture}}}} T^{n-1} \times T_+ \bigg\}
    \end{align}

    We will argue inductively in $k$ that $T^n$ is isomorphic to the complex
    \begin{equation} \label{eqn:extra_T2}
        \left( T^{n-k} \times U_+ \times T_+^{k-1} \right)^{\oplus k} \xrightarrow{d} T^{n-k} \times T_+^{k}
    \end{equation}
    where $d_i$ is
    \begin{equation}
    \vcenter{\hbox{
    \begin{tikzpicture}[scale=0.5]
        \draw[style=su2 puncture strand] (-1,0) -- (-1,2);
        \draw[style=su2 puncture strand] (1,0) -- (1,2);
        \draw (2.5,0) to node[pos=0, below]{$\scriptstyle{i}$} (1.5,2);
        \draw (0.5,0) -- (0.5,2);
        \draw (-.5,0) -- (-.5,2);
        \node at (0,1) {$\scriptstyle\cdots$};
        \draw (2,0) -- (2,2);
        \node at (2.5,1) {$\scriptstyle\cdots$};
        \draw (3,0) -- (3,2);
        \node at (0,2.6) {$\overbrace{\hspace{.3cm}}^{n-k}$};
    \end{tikzpicture}}}
    \end{equation}
    We have already checked this holds for $k = 1$; suppose now we have established \eqref{eqn:extra_T2} for $k$ and wish to show it for $k + 1$. 

    Fix attention on the rightmost component of $T^{n-k}$.  We will attempt to resolve it into $U_+ \to T_+$, and keep track of what happens to each term in \eqref{eqn:extra_T2} and what happens to the diagrams in $d$.   
    The diagram $d_i$ corresponds to an intersection point $(id_{T^{n-k-1}}, p, c_i)$ with $id_{T^{n-k-1}} \in T^{n-k-1} \cap T^{n-k-1}$, $p \in T \cap T$ and $c_i \in (U_+ \times T_+^{n-1}) \cap T_+^n$.  
    
    The presence of the leftmost $T^{n-k-1}$ does not affect the remainder of the argument, so for expository and diagrammatic simplicity, we draw pictures for $n = k+1$. 
    The intersection points $(p,c_i)$ before resolving are shown in Figure \ref{fig:addingT2}.

    First we resolve the $T$ in $T \times U_+ \times T_+^{k-1}$, by  isotoping it toward the bottom of the page as shown in Figure \ref{fig:isotopyT2blue}, and then breaking.  While this isotopy creates new intersection points between the multicurves, note that it does not create intersections of the corresponding Lagrangians: such an intersection is a tuple of intersections of the components of the multicurves, and the leftmost black $T$ intersects only one such curve, namely the moving blue $T_+$. Thus all intersection points between the Lagrangians use the original intersection point $p$, so the new intersections of the moving blue $T_+$ with the black curves do not actually give rise to intersections of Lagrangians. 
    The resulting Lagrangian is a cone over a map from $U_+^2 \times T_+^{k-1}$ to $U_+ \times T_+^k$, but $U_+^2 \times T_+^{k-1}$ is zero, so we can throw that term away, leaving only $U_+ \times T_+^k$, shown in Figure \ref{fig:removingblueUxU}. 
    
    Next, we isotope $T \times T_+^k$ as shown in Figure \ref{fig:isotopyT2black}. We would like to know that any intersection point of the Lagrangians of the form $(p \cdots)$ is sent ``to itself'' by the continuation map. However, as the isotopy creates new intersection points involving the black $T$ and the blue $U_+$,  this need not be the case a priori.  However, we will check that any intersection point involving $p$ remains closed, and, since these new intersection points form an acyclic complex, the class of $(p \cdots)$ in homology is the same as that of the image of $(p \cdots)$ under the continuation map. 
    
    We study disks using the cylindrical model recalled in Theorem \ref{sl2 cylindrical model}.  In the $\C^*_y$, there is a disk from $q_2$ to $q_1$, but no other disks.  Since this disk neither has branch points nor passes through $\mathbf{a}$, the corresponding fiber disk in $\C^*_u$ must be constant.  Thus in the cylindrical model, all relevant curves are simply disjoint copies of disks, where all maps, except the map to the aforementioned disk $q_2$ to $q_1$, are constant.  Thus $\partial (q_2 \cdots) = q_1 \cdots$.

    We write the resulting Lagrangian as a cone over a map $f: U_+ \times T_+^k \to T_+^{k+1}$ as in Figure \ref{fig:breakT2black}. This map is the identity on the $T_+^{k+1}$ component of $U_+ \times T_+^k$. 
    
    Since $p$ is on the $T_+^{k+1}$ piece of $\mathrm{cone}(f)$, the corresponding term $d_i$ in the differential from $U_+ \times T_+^k$ only maps to the $T_+^{k+1}$ term of $\mathrm{cone}(f)$.
    
    Before isotoping, the diagram corresponding to $c_i$ is
    \begin{equation}
    \vcenter{\hbox{
    \begin{tikzpicture}[scale=0.5]
        \draw[style=su2 puncture strand] (1,0) -- (1,2);
        \draw (2.5,0) to node[pos=0, below]{$\scriptstyle{i}$} (1.5,2);
        \draw[style=su2 puncture strand] (0.5,0) -- (0.5,2);
        \draw (2,0) -- (2,2);
        \node at (2.5,1) {$\scriptstyle\cdots$};
        \draw (3,0) -- (3,2);
    \end{tikzpicture}}}
    \end{equation}
    After isotoping, $p$ is an intersection point from the second leftmost $T_+$ (at the top of the diagram) to the leftmost $T_+$ (at the bottom of the diagram). We insert the corresponding strand into the diagram. The result is
    \begin{equation}
    \vcenter{\hbox{
    \begin{tikzpicture}[scale=0.5]
        \draw[style=su2 puncture strand] (0,0) -- (0,2);
        \draw (2.5,0) to node[pos=0, below]{$\scriptstyle{i+1}$} (1,2);
        \draw[style=su2 puncture strand] (0.5,0) -- (0.5,2);
        \draw (1.5,0) -- (1.5,2);
        \draw (2,0) -- (2,2);
        \node at (2.5,1) {$\scriptstyle\cdots$};
        \draw (3,0) -- (3,2);
    \end{tikzpicture}}}
    \end{equation}
    
    Finally restoring the 
     $id_{T^{n-k-1}}$, the diagram is
    \begin{equation}
    \vcenter{\hbox{
    \begin{tikzpicture}[scale=0.5]
        \draw[style=su2 puncture strand] (.5,0) -- (.5,2);
        \draw (2.5,0) to node[pos=0, below]{$\scriptstyle{i+1}$} (1,2);
        \draw[style=su2 puncture strand] (-1.5,0) -- (-1.5,2);
        \draw (0,0) -- (0,2);
        \draw (-1,0) -- (-1,2);
        \node at (-.5,1) {$\scriptstyle\cdots$};
        \draw (1.5,0) -- (1.5,2);
        \draw (2,0) -- (2,2);
        \node at (2.5,1) {$\scriptstyle\cdots$};
        \draw (3,0) -- (3,2);
        \node at (-.5,2.6) {$\overbrace{\hspace{.3cm}}^{n-k-1}$};
    \end{tikzpicture}}}
    \end{equation}
    Call this new map $d_{i+1}'$ and set $d_1' = f$. The resulting complex is
    \begin{equation}
        \left( T^{n-k-1} \times U_+ \times T_+^{k} \right)^{\oplus k+1} \xrightarrow{d'} T^{n-k-1} \times T_+^{k+1},
    \end{equation}
    completing the inductive step, and the proof.
\end{proof}

\section{Braiding} \label{sec: braiding}

The 
local constancy of Fukaya categories gives a monodromy action 
on $Fuk(\MCB(\Gamma, \vec{d}), \mathcal{W}_{\mathbf{a}})$ as 
$\mathbf{a}$ varies.  It is easy to see this action preserves  
$Fuk_{| | |}(\MCB(\Gamma, \vec{d}), \mathcal{W}_{\mathbf{a}})$.  Indeed, varying $\mathbf{a}$ means varying the `red points' in the annulus.  If we choose a diffeomorphism of the annulus implementing said variation, we may push the multicurves forward along the diffeomorphism.  This gives the action on $Fuk_{| | |}(\MCB(\Gamma, \vec{d}), \mathcal{W}_{\mathbf{a}})$.  {\em Unusually, we will view the clockwise twists of the red points as the positive elements in the braid group. }

In the present section, we will determine this action explicitly when $\Gamma = \bullet$.  More precisely, fix a choice of $\mathbf{a}$ with distinct arguments, e.g. $n$ points on the unit circle; we have then per Theorem \ref{five authors theorem} and Corollary \ref{weak generation} the equivalence: 
$$\mathbb{A}: \mathcal{C}_{\bullet, d, \arg(\mathbf{a})} \xrightarrow{\sim}  Fuk_{| | |}(\MCB(\bullet, d), \mathcal{W}_{\mathbf{a}})$$

Fix attention on some pair amongst the entries of $\mathbf{a}$ with adjacent arguments; we will study the counterclockwise half-twist.  We will write $\mathbb{B}^W$ for Webster's braiding endofunctor on $\mathcal{C}_{\bullet, d, \arg(\mathbf{a})}$ for this half-twist, and $\mathbb{B}^A$ for the geometrically defined braiding endofunctor on $Fuk_{| | |}(\MCB(\bullet, d), \mathcal{W}_{\mathbf{a}})$. 

\begin{theorem}
    There is a natural isomorphism $\mathbb{A} \circ \mathbb{B}^W \cong \mathbb{B}^A \circ \mathbb{A}$.
\end{theorem}
\begin{proof}
    
It will suffice to construct the natural transformation on generators of the category; we will use the $\theta$ elements.  
So for any $\theta$, we should construct an isomorphism
\begin{equation} \label{desired intertwining iso}
    \nu_\theta: \mathbb{B}^A \circ \mathbb{A}(\theta) \cong \mathbb{A} \circ \mathbb{B}^W(\theta)
\end{equation}
and, for any $f: \theta \to \theta'$, \begin{equation} \label{desired intertwining check}
 (\mathbb{A} \circ \mathbb{B}^W)(f) \circ \nu_\theta= 
\nu_{\theta'} \circ (\mathbb{B}^A \circ \mathbb{A})(f)  
 \in \Hom(\mathbb{B}^A \circ \mathbb{A}(\theta), \mathbb{A} \circ \mathbb{B}^W(\theta'))
\end{equation}
Note that given \eqref{desired intertwining iso}, the Hom space above is isomorphic to $\Hom(\theta, \theta')$, which is (an explicit module of diagrams) concentrated in degree zero.  In particular, the aforementioned equality is a property and not a structure.  

Let us discuss objects. 
The essential case is that of  $n$ strands between the two chosen points, i.e. the object we have called $\theta^n$ previously.  By Proposition \ref{T2 res prop}, $\theta^n \cong \Lambda_n$.  Figure \ref{fig:geometric braiding} gives a canonical identification $\mathbb{B}^A(U_+ \times T_+^{n-1}) = U_- \times T_+^{n-1}$.  The maps originally in the resolution of $T^n$ are transformed as follows:
$$\bigg( \vcenter{\hbox{\begin{tikzpicture}[scale=.5]
	\begin{pgfonlayer}{nodelayer}
		\node [style=none] (0) at (-0.25, 0.5) {};
		\node [style=none] (1) at (0.25, 0.5) {};
		\node [style=none] (2) at (-0.25, -0.25) {};
		\node [style=none] (3) at (0.25, -0.25) {};
		\node [style=none] (4) at (0, -0.25) {};
		\node [style=none] (5) at (0, 0.5) {};
	\end{pgfonlayer}
	\begin{pgfonlayer}{edgelayer}
		\draw [style=su2 puncture strand] (0.center) to (2.center);
		\draw [style=brane] (3.center) to (1.center);
		\draw [style=su2 puncture strand, in=90, out=-90, looseness=1.25] (5.center) to (4.center);
	\end{pgfonlayer}
\end{tikzpicture}}} \times id_{T_+^{n-1}}:U_+ \times T_+^{n-1} \to T_+^n \bigg) \mapsto \bigg( \vcenter{\hbox{\begin{tikzpicture}[scale=.5]
	\begin{pgfonlayer}{nodelayer}
		\node [style=none] (0) at (-0.25, 0.5) {};
		\node [style=none] (1) at (0.25, 0.5) {};
		\node [style=none] (2) at (-0.25, -0.25) {};
		\node [style=none] (3) at (0.25, -0.25) {};
		\node [style=none] (4) at (0.5, -0.25) {};
		\node [style=none] (5) at (0, 0.5) {};
	\end{pgfonlayer}
	\begin{pgfonlayer}{edgelayer}
		\draw [style=su2 puncture strand] (0.center) to (2.center);
		\draw [style=su2 puncture strand] (3.center) to (1.center);
		\draw [style=brane, in=90, out=-90, looseness=1.25] (5.center) to (4.center);
	\end{pgfonlayer}
\end{tikzpicture}}}\times id_{T_+^{n-1}}:U_- \times T_+^{n-1} \to T_+^n \bigg).$$ 
The resulting resolution is obviously 
$\mathbb{A}(\Lambda_n')$, hence 
we have given an isomorphism $\mathbb{B}^A \circ \mathbb{A}(\Lambda_n) \cong \mathbb{A}(\Lambda_n')$, which we may compose with the identification by Proposition \ref{L_n braiding} of $\Lambda_n' \cong \mathbb{B}^W (\Lambda_n)$. 
This is our $\nu_{\theta^n}$. 

For any other object, the same argument works: all strands not between the two points being braided just come along for the ride.  

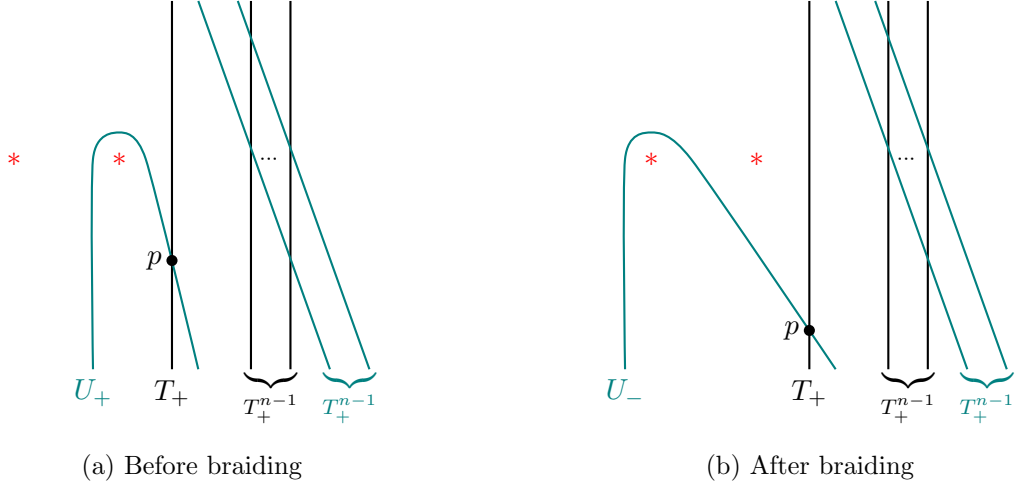
\begin{figure}
    \centering
    \begin{subfigure}[b]{0.49\textwidth}
    \centering
    \begin{tikzpicture}[scale=.7]
        \node[red] at (0,0) {$*$};
        \node[red] at (-2,0) {$*$};
        \draw[thick, name path=t3] (1,-4) node[below] {$T_+$} -- (1,3);
        \draw[thick] (2.5,-4) -- (2.5,3);
        \node at (2.875,0) {$\scriptstyle\cdots$};
        \draw[thick] (3.25,-4) -- (3.25,3);
        \node[teal, below] at (-.5,-4) {$U_+$};
        \draw[thick, teal, name path=u2] plot [smooth, tension=0.4] coordinates {(-.5,-4) (-.5,0) (0,.5) (.5,0) (1.5,-4)};
        \draw[thick, teal] (4,-4) -- (1.5,3);
        \draw[thick, teal] (4.75,-4) -- (2.25,3);
        \node at (2.875,-4.55) {$\underbrace{\hspace{.5cm}}_{T_+^{n-1}}$};
        \node[teal] at (4.375,-4.55) {$\underbrace{\hspace{.5cm}}_{T_+^{n-1}}$};
        \fill[name intersections={of=t3 and u2}] (intersection-1) circle (3pt) node[left] {$p$};
    \end{tikzpicture}
    \caption{Before braiding}
    \end{subfigure}
    \begin{subfigure}[b]{0.49\textwidth}
    \centering
    \begin{tikzpicture}[scale=.7]
        \node[red] at (0,0) {$*$};
        \node[red] at (-2,0) {$*$};
        \draw[thick, name path=t3] (1,-4) node[below] {$T_+$} -- (1,3);
        \draw[thick] (2.5,-4) -- (2.5,3);
        \node at (2.875,0) {$\scriptstyle\cdots$};
        \draw[thick] (3.25,-4) -- (3.25,3);
        \node[teal, below] at (-2.5,-4) {$U_-$};
        \draw[thick, teal, name path=u2] plot [smooth, tension=0.4] coordinates {(-2.5,-4) (-2.5,0) (-2,.5) (-1.25,0) (1.5,-4)};
        \draw[thick, teal] (4,-4) -- (1.5,3);
        \draw[thick, teal] (4.75,-4) -- (2.25,3);
        \node at (2.875,-4.55) {$\underbrace{\hspace{.5cm}}_{T_+^{n-1}}$};
        \node[teal] at (4.375,-4.55) {$\underbrace{\hspace{.5cm}}_{T_+^{n-1}}$};
        \fill[name intersections={of=t3 and u2}] (intersection-1) circle (3pt) node[left] {$p$};
    \end{tikzpicture}
    \caption{After braiding}
    \end{subfigure}
    \caption{Action of geometric braiding on a map from $U_+ \times T_+^{n-1}$ to $T_+^n$}
    \label{fig:geometric braiding}
\end{figure}

\vspace{2mm}
    We have now established that $\Phi:= (\mathbb{B}^W)^{-1} \circ  \mathbb{A}^{-1} \circ  \mathbb{B}^A \circ \mathbb{A}$ is an automorphism of the KLRW category preserving equivalence classes of objects; pulling back the above $\nu_\theta$ along $\mathbb{A} \circ \mathbb{B}^W$ gives isomorphisms $\Phi(\theta) \to \theta$.  Both $\mathbb{B}^W$ and $\mathbb{B}^A$ respect gradings, as does $\mathbb{A}$, hence so does $\Phi$.     
    It remains only to verify the hypotheses (1), (2) of Corollary \ref{cor:automorphisms} to check that, possibly after correcting by appropriate signs, $\eta_\theta$ determines a natural isomorphism. 
    But both $\mathbb{B}^W$ and $\mathbb{A}^{-1} \mathbb{B}^A \mathbb{A}$ independently (and obviously) restrict to functors naturally isomorphic to the identity on the subcategory on objects and morphisms not passing between the pair of red strands on which braiding is acting, hence verifying (1).  As for (2), we check explicitly that the given morphism is preserved independently by $\mathbb{B}^W$ and $\mathbb{B}^A$. 
    Indeed, for $\mathbb{B}^W$, the very definition of the braiding bimodule is that $\begin{tikzpicture}[scale=.5]
	\begin{pgfonlayer}{nodelayer}
		\node [style=none] (0) at (0.375, 0.5) {};
		\node [style=none] (1) at (0.125, 0.5) {};
		\node [style=none] (2) at (-0.375, -0.25) {};
		\node [style=none] (3) at (0.125, -0.25) {};
		\node [style=none] (4) at (-0.125, -0.25) {};
		\node [style=none] (5) at (-0.125, 0.5) {};
	\end{pgfonlayer}
	\begin{pgfonlayer}{edgelayer}
		\draw [style=su2 puncture strand] (3.center) to (1.center);
		\draw [style=su2 puncture strand] (5.center) to (4.center);
		\draw [style=brane, in=90, out=-90] (0.center) to (2.center);
	\end{pgfonlayer}
\end{tikzpicture}$  commutes with a crossing of the two red strands.  For $\mathbb{B}^A$, the parallel transport implementing braiding is can be taken to be the identity away from a region enclosing the two red points, and $\begin{tikzpicture}[scale=.5]
	\begin{pgfonlayer}{nodelayer}
		\node [style=none] (0) at (0.375, 0.5) {};
		\node [style=none] (1) at (0.125, 0.5) {};
		\node [style=none] (2) at (-0.375, -0.25) {};
		\node [style=none] (3) at (0.125, -0.25) {};
		\node [style=none] (4) at (-0.125, -0.25) {};
		\node [style=none] (5) at (-0.125, 0.5) {};
	\end{pgfonlayer}
	\begin{pgfonlayer}{edgelayer}
		\draw [style=su2 puncture strand] (3.center) to (1.center);
		\draw [style=su2 puncture strand] (5.center) to (4.center);
		\draw [style=brane, in=90, out=-90] (0.center) to (2.center);
	\end{pgfonlayer}
\end{tikzpicture}$ involves objects and a morphism far separated from such a region (see Figure \ref{fig:crossing two red strands}). 
\end{proof}

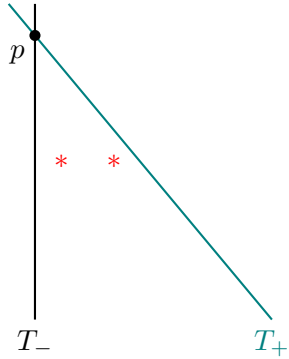
\begin{figure}
    \centering
    \begin{tikzpicture}[scale=.7]
        \node[red] at (-.5,0) {$*$};
        \node[red] at (.5,0) {$*$};
        \draw[thick, name path=t0] (-1,-3) node[below] {$T_-$} -- (-1,3);
        \draw[thick, teal, name path=t2] (3.5,-3) node[below] {$T_+$} -- (-1.5,3);
        \fill[name intersections={of=t0 and t2}] (intersection-1) circle (3pt) node[below left] {$p$};
    \end{tikzpicture}
    \caption{An intersection point invariant under braiding}
    \label{fig:crossing two red strands}
\end{figure}
 
Since counterclockwise half-twists and their inverses generate the braid group, this establishes the desired intertwining. 

\begin{remark}
Because it arises from the monodromy of a local system of categories over the configuration space of points on the plane, the braid group action from monodromy is obviously a `strong' braid group action -- i.e., in terms of generators and relations, the natural transformation ensuring compatibility with the braid relation $s_1 s_2 s_1 = s_2 s_1 s_2$ should carry certain coherences.  
Webster has also shown that his action is strong in this sense \cite{webster}.  

For a braid group action on (not infinity) categories, being a strong action is a property of said natural transformation, not additional structure; thus, to check that two strong braid group actions on categories intertwine is the same as checking that the underlying weak braid group actions (ignoring the aforementioned coherences) intertwine. 

However, the corresponding higher categorical notion involves further structures (which the monodromy action would automatically carry, and which Webster established in his  setting).  We have not checked that these further structures are intertwined.   

The upshot of all of this is that from the intertwining we have proven, it follows that the Webster and Aganagic formulas for Khovanov homology are canonically isomorphic, and that the formulas for the Khovanov chain complex are quasi-isomorphic by a quasi-isomorphism  which is canonical up to a homotopy, but we have not established that said homotopy is itself canonical up to higher homotopies.  Of course, we expect that it is.
\end{remark}

\section{Caps and cups} \label{sec: cupcap}

\subsection{Intertwining cups}

Let $\Pi$ be a choice of disjoint paths between disjoint pairs of elements of $\mathbf{a}$.  
We write $E_\Pi$ for the multi-curve Lagrangian given in the base by figure-eights near the paths around each pair, as in the $E$ of Figure \ref{fig:cups and caps}. 
Note that the resulting Lagrangian need not be immersed (one is free to wiggle the curve-in-$\C_u^*$ part at the moment of self-intersection on the base), and anyway even if one did take an immersed representative, the immersed point would be invisible to disks because the immersion is the image of an embedding on a cover, to which any disk must necessarily lift. 

Consider the situation where $\mathbf{a}$ consists of 2n points on the unit circle minus $1$.  We write $\Pi(n)$ for the pairing matching adjacent elements by the arc in the unit circle minus $1$. 

We did not officially allow such immersed curves for multi-curve Lagrangians in our general setup.  Nevertheless: 

\begin{lemma}
    A multicurve Lagrangian  with figure-eight components is isomorphic to an element of $Fuk_{| | |}(\MCB(\Gamma, \vec{d}), \mathcal{W}_{\mathbf{a}})$. 
\end{lemma}
\begin{proof}
    We run the argument of Corollary \ref{weak generation}, including all the vertical lines through points in $\mathbf{a}$ among the `red lines' of that proof, and first breaking all non-figure-eight components.  Having done so, we are reduced to consider an object with figure-eights and vertical lines; it remains to remove the figure-eights.  For a figure-eight, it meets the red lines through its $\mathbf{a}$ points in four points -- two above, and two below.  Let us take the two below intersections and stretch them downward.  Their complement is disconnected, and the usual action filtration argument then implies that the figure-eight can be resolved as a cone on some morphism between two downward $U$ objects as in Figure \ref{fig:resolution of E} (`times' the remaining components of the multicurve).  
\end{proof}

\begin{figure}
    \centering
    \begin{tikzpicture}
        \node[red] at (0,0) {$*$};
        \node[red] at (2,0) {$*$};
        \draw[thick] plot [smooth, tension=0.6] coordinates {(-.5,-2) (-.5,0) (0,.5) (.5,0) (.5,-2)};
        \draw[thick, name path=l2] plot [smooth, tension=0.6] coordinates {(1.5,-2) (1.5,0) (2,.5) (2.5,0) (2.5,-2)};
        \draw[dashed, ->] (-.5,-1.8) -- (1.5,-1.8);
        \draw[dashed, ->] (.5,-2) -- (2.5,-2);
    \end{tikzpicture}
    \caption{A figure-eight breaking into two $U$ objects}
    \label{fig:resolution of E}
\end{figure}
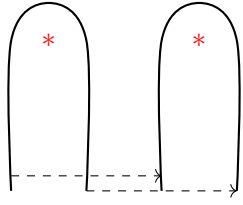

Let $\theta_{\Pi(n)}$ be the KLRW object given by $\vcenter{\hbox{\begin{tikzpicture}[scale=.5]
	\begin{pgfonlayer}{nodelayer}
		\node [style=none] (0) at (-0.825, 0.5) {};
		\node [style=none] (1) at (-1.075, 0.5) {};
		\node [style=none] (2) at (-0.825, -0.25) {};
		\node [style=none] (3) at (-1.075, -0.25) {};
		\node [style=none] (4) at (-0.575, -0.25) {};
		\node [style=none] (5) at (-0.575, 0.5) {};
		\node [style=none] (6) at (-0.075, 0.5) {};
		\node [style=none] (7) at (-0.325, 0.5) {};
		\node [style=none] (8) at (-0.075, -0.25) {};
		\node [style=none] (9) at (-0.325, -0.25) {};
		\node [style=none] (10) at (0.175, -0.25) {};
		\node [style=none] (11) at (0.175, 0.5) {};
		\node [style=none] (12) at (1, 0.5) {};
		\node [style=none] (13) at (0.75, 0.5) {};
		\node [style=none] (14) at (1, -0.25) {};
		\node [style=none] (15) at (0.75, -0.25) {};
		\node [style=none] (16) at (1.25, -0.25) {};
		\node [style=none] (17) at (1.25, 0.5) {};
		\node [style=none] (18) at (0.5, 0.125) {$\scriptstyle\cdots$};
	\end{pgfonlayer}
	\begin{pgfonlayer}{edgelayer}
		\draw [style=brane] (0.center) to (2.center);
		\draw [style=su2 puncture strand] (3.center) to (1.center);
		\draw [style=su2 puncture strand, in=90, out=-90, looseness=1.25] (5.center) to (4.center);
		\draw [style=brane] (6.center) to (8.center);
		\draw [style=su2 puncture strand] (9.center) to (7.center);
		\draw [style=su2 puncture strand, in=90, out=-90, looseness=1.25] (11.center) to (10.center);
		\draw [style=brane] (12.center) to (14.center);
		\draw [style=su2 puncture strand] (15.center) to (13.center);
		\draw [style=su2 puncture strand, in=90, out=-90, looseness=1.25] (17.center) to (16.center);
	\end{pgfonlayer}
\end{tikzpicture}}}$.  Then there is a simple KLRW module $S_{\Pi(n)}$ characterized by $\Hom(\theta_*, S_{\Pi(n)}) = 0$ unless $* = \Pi(n)$, and $\Hom(\theta_{\Pi(n)}, S_{\Pi(n)}) = \Z$, where said $\Z$ is in $u, \hbar, J$ gradings zero.  (This grading constraint forces all nontrivial elements of $\Hom(\theta_{\Pi(n)}, \theta_{\Pi(n)})$ to act trivially on 
$\Hom(\theta_{\Pi(n)}, S_{\Pi(n)})$, hence characterizing the module.) 

\begin{lemma}
     $\mathbb{A}(S_{\Pi(n)}) \cong E_{\Pi(n)}$
\end{lemma}
\begin{proof}
    It is geometrically obvious that $\Hom(T_*, E_{\Pi(n)}) = 0$ except for $T = T_{\Pi(n)} := \mathbb{A}(\theta_{\Pi(n)})$; indeed, otherwise there are simply no intersections.  It remains to check that $\Hom( T_{\Pi(n)}, E_{\Pi(n)}) \cong \Z$.

    Let us focus on a single component of $T_{\Pi(n)}$ and the component of $E_{\Pi(n)}$ that it intersects. Their intersection points are $px^i$ and $qx^i$ for $i \in \Z_{\geq 0}$, shown in Figure \ref{fig:T E intersections}. There is only one possible disk in the base, from $q$ to $p$, which passes through $\mathbf{a}$ once, so any possible corresponding maps $\Phi_u$ must have one zero. The only possibility for the image of $\Phi_u$ is shown in Figure \ref{fig:T E intersections}.  (The argument that there is a unique such disk is identical to that for the corresponding disk in Proposition \ref{slide move map} above.)  We find $\partial(qx^i) = \pm px^{i+1}$. Since each component of $T_{\Pi(n)}$ only intersects one component of $E_{\Pi(n)}$, we just take the tensor product of $n$ copies of the complex for a single component. The homology is a single intersection point, so $\Hom( T_{\Pi(n)}, E_{\Pi(n)}) \cong \Z$.
\end{proof}

Since what we called $\cup_A^n$ in Theorem \ref{thm: intertwining} was defined to be $E_{\Pi(n)}$ in \cite{aganagic-knot-2}, and what we called $\cup_W^n$ was defined to be $S_{\Pi(n)}$ in \cite{webster}, this establishes the  intertwining of cups asserted in Theorem \ref{thm: intertwining}, and completes its proof. 

\begin{figure}
    \centering
    \begin{subfigure}[b]{0.4\textwidth}
    \centering
    \begin{tikzpicture}
        \node[red] at (0,0) {$*$};
        \node[red] at (2,0) {$*$};
        \draw[thick, name path=e] plot [smooth cycle, tension=.5] coordinates {(0,.5) (2,-.5) (2.5,0) (2,.5) (0,-.5) (-.5,0)};
        \draw[thick, teal, name path=t] (1.5,-1.5) -- (1.5,1.5);
        \fill[name intersections={of=e and t, by=p}] (intersection-1) circle (2pt) node[above left] {$p$};
        \fill[name intersections={of=e and t, by=q}] (intersection-2) circle (2pt) node[below left] {$q$};
        \begin{scope}[on background layer]
            \clip (1.5,-1) rectangle (2.5,1);
            \fill[red!20] plot [smooth cycle, tension=.5] coordinates {(0,.5) (2,-.5) (2.5,0) (2,.5) (0,-.5) (-.5,0)};
        \end{scope}
    \end{tikzpicture}
    \caption{$\Hom(T,E)$, base}
    \end{subfigure}
    \begin{subfigure}[b]{0.4\textwidth}
    \centering
    \begin{tikzpicture}[scale=.7]
		\draw (-5.5,2.5) -- (0,2.5);
		\draw (-5.5,0) -- (0,0);
		\draw (-5.5,1.25) ellipse (0.25 cm and 1.25 cm);
		\draw (0,0) arc[
				start angle=-90,
				end angle=90,
				x radius=0.25cm,
				y radius=1.25cm
				];
	    \draw[dashed] (0,2.5) arc[
				start angle=90,
				end angle=270,
				x radius=0.25cm,
				y radius=1.25cm
				];
		\draw[thick, name path=t] (.25,1) -- (-5.25,1);
        \draw[thick, teal, name path=w2] (0.15,2.25) -- (-2.25,0);
        \draw[thick, teal, dashed] (-3.5,2.5) -- (-2.25,0);
        \draw[thick, teal, name path=w4] (-3.5,2.5) -- (-5,0);
        \node[above] at (0,2.5) {$\scriptstyle{u=\infty}$};
        \node[above] at (-5.5,2.5) {$\scriptstyle{u=0}$};
        \fill[name intersections={of=w2 and t, by=i1}] (intersection-1) circle (3pt) node[above] {$x^i$};
        \fill[name intersections={of=w4 and t, by=i2}] (intersection-1) circle (3pt) node[below right] {$x^{i+1}$};
        \begin{scope}[on background layer] 
            \fill [red!20] (i2) -- (i1) -- (-2.25,0) -- (-5.5,0) arc (-90:90:.25cm and 1.25cm) -- (-3.5,2.5) -- cycle;
        \end{scope}
    \end{tikzpicture}
    \caption{$\Hom(T,E)$, fiber}
    \end{subfigure}
    \caption{$\Hom(T,E)$}
    \label{fig:T E intersections}
\end{figure}

\subsection{Aganagic's caps}

In  fact, rather than pairing one object $\cup_A^n$ with its image under braiding, Aganagic's original proposal  \cite{aganagic-knot-2} paired  $\cup_A^n$ with the braiding of a differently described object $\cap_A^n$.  It has since been understood that $\cup_A^n \cong \cap_A^n$; here we describe these objects and sketch a proof. 

The object $\cap_A^n$ was described as follows.  Again pair up points by paths $\Pi$, and now consider the Lagrangian obtained by a variant of the multicurve prescription where we take these paths in the base, and now a different curve in the `fiber' $\C^*_u$ direction, as seen in the $I$  in Figure \ref{fig:cups and caps}. 
As described, this Lagrangian simply ends and has boundary over the points in $\mathbf{a}$; to correct this, one either should add over said points the part of the fiber `above' (closer to $u= \infty$) the curve and round corners, or alternatively, push the fiber curve towards $u = \infty$ as one approaches the $\mathbf{a}$.  (This is consistent with the wrapping prescription because, precisely over the points of $\mathbf{a}$, the fiberwise stop vanishes.)  We denote the resulting Lagrangian $I_\Pi$.\footnote{We will not check here that the resulting $I_\Pi$ can be made in a suitably conic manner to really be a bona fide element of the wrapped Fukaya category; this is the sense in which the following argument is a sketch.  Once one accepts this object and that moreover one can calculate with it using the cylindrical model, the following discussion is rigorous.} 

Here we will show that $I_\Pi \cong E_\Pi$.  For simplicity, we just discuss the case where $|\mathbf{a}| = 2$ and $d = 1$; the general case is just many disjoint copies of the same argument. 

\begin{figure}
    \centering
    \begin{subfigure}[b]{0.24\textwidth}
    \centering
    \begin{tikzpicture}
        \node[red] at (0,0) {$*$};
        \node[red] at (2,0) {$*$};
        \draw[thick] plot [smooth cycle, tension=.5] coordinates {(0,.5) (2,-.5) (2.5,0) (2,.5) (0,-.5) (-.5,0)};
    \end{tikzpicture}
    \caption{$E$, base}
    \end{subfigure}
    \begin{subfigure}[b]{0.24\textwidth}
    \centering
    \begin{tikzpicture}[scale=.5]
		\draw (-5.5,2.5) -- (0,2.5);
		\draw (-5.5,0) -- (0,0);
		\draw (-5.5,1.25) ellipse (0.25 cm and 1.25 cm);
		\draw (0,0) arc[
				start angle=-90,
				end angle=90,
				x radius=0.25cm,
				y radius=1.25cm
				];
	    \draw[dashed] (0,2.5) arc[
				start angle=90,
				end angle=270,
				x radius=0.25cm,
				y radius=1.25cm
				];
		\draw[thick] (.25,1) -- (-5.25,1);
        \node[above] at (0,2.5) {$\scriptstyle{u=\infty}$};
        \node[above] at (-5.5,2.5) {$\scriptstyle{u=0}$};
    \end{tikzpicture}
    \caption{$E$, fiber}
    \end{subfigure}
    \begin{subfigure}[b]{0.24\textwidth}
    \centering
    \begin{tikzpicture}
        \draw[thick] (0,0) -- (2,0);
        \node[red] at (0,0) {$*$};
        \node[red] at (2,0) {$*$};
    \end{tikzpicture}
    \caption{$I$, base}
    \end{subfigure}
    \begin{subfigure}[b]{0.24\textwidth}
    \centering
    \begin{tikzpicture}[scale=.5]
		\draw (-5.5,2.5) -- (0,2.5);
		\draw (-5.5,0) -- (0,0);
		\draw (-5.5,1.25) ellipse (0.25 cm and 1.25 cm);
		\draw (0,0) arc[
				start angle=-90,
				end angle=90,
				x radius=0.25cm,
				y radius=1.25cm
				];
	    \draw[dashed] (0,2.5) arc[
				start angle=90,
				end angle=270,
				x radius=0.25cm,
				y radius=1.25cm
				];
        \node[above] at (0,2.5) {$\scriptstyle{u=\infty}$};
        \node[above] at (-5.5,2.5) {$\scriptstyle{u=0}$};
        \draw[thick, in=-90, out=180] (.25,1.25) to (-3,2.5);
        \draw[thick, in=90, out=180] (.25,1) to (-3,0);
        \draw[thick, dashed] (-3,2.5) arc[
				start angle=90,
				end angle=270,
				x radius=0.25cm,
				y radius=1.25cm
				];
    \end{tikzpicture}
    \caption{$I$, fiber}
    \end{subfigure}
    \caption{$E$ and $I$}
    \label{fig:cups and caps}
\end{figure}

\begin{lemma} \label{U and I}
    The $U$ Lagrangian is isomorphic to the Lagrangian $J$ shown in Figure \ref{fig:stab}.
\end{lemma}

\begin{proof}
    There are unique intersection points $p \in \Hom(U,J)$ and $q \in \Hom(J,U)$. We will show explicitly that $p \cdot q = id \in \Hom(J, J)$ and $q \cdot p = id \in \Hom(U, U)$.

    The only point in $H^*\Hom(J,J)$ is the identity, which we call $r$. Therefore the only disk in $\C^*_y$ that we need to consider is the one shown in Figure \ref{fig:stab map composition base}, which passes through $\mathbf{a}$ once. The corresponding $\Phi_u$ must have have one zero and not pass through $\infty$. The only possibility for $\Phi_u$ must have image shown in Figure \ref{fig:stab map composition fiber}. We want to check that there is actually such a map $\Phi_u$ that sends $\Phi_y^{-1}(\mathbf{a})$ to zero.

    The following argument was explained to us by Peng Zhou; it is similar to arguments in \cite{ADLSZ}. 
    The domain $S$ is a disk with 3 marked boundary points $p$, $q$, and $r$ and one marked point in the interior. The moduli space of such disks is two-dimensional with 6 one-dimensional boundary components (when the marked point approaches any one of the three marked boundary points or any one of the three intervals between them). The moduli spaces of maps $\mathcal{M}_{S \to \C_u^*}$ (where the marked point is sent to zero) and $\mathcal{M}_{S \to \C_y^*}$ (where the marked point is sent to $\mathbf{a}$) are each one-dimensional with boundary. As we approach one boundary of $\mathcal{M}_{S \to \C_y^*}$, the marked point approaches the segment between $p$ and $q$. At the other boundary, it approaches the segment between $p$ and $r$. The boundaries of $\mathcal{M}_{S \to \C_u^*}$ are when the marked point approaches $r$ or the segment between $q$ and $p$. 
    In fact, we while we don't know precisely how the interior of said moduli behave (the simplest possibility is depicted in Figure \ref{fig:stab map moduli map} below), the intersection number 
    $\mathcal{M}_{S \to \C_u^*} \cap \mathcal{M}_{S \to \C_y^*}$ is determined by the boundary behavior, hence must be $\pm 1$.  

    There is also only one point (the identity) in $H^*\Hom(U,U)$. Call it $s$. A similar argument using the disks and moduli spaces shown in Figure \ref{fig:stab map other composition} shows that $p \cdot q = s$.
\end{proof}

\begin{figure}
    \centering
    \begin{subfigure}[b]{0.24\textwidth}
    \centering
    \begin{tikzpicture}
        \draw[thick] (-.5,-2) -- (-.5,0) arc(180:0:.5) -- (.5,-2);
        \node[red] at (0,0) {$*$};
    \end{tikzpicture}
    \caption{$U$, base}
    \end{subfigure}
    \begin{subfigure}[b]{0.24\textwidth}
    \centering
    \begin{tikzpicture}[scale=.5]
		\draw (-5.5,2.5) -- (0,2.5);
		\draw (-5.5,0) -- (0,0);
		\draw (-5.5,1.25) ellipse (0.25 cm and 1.25 cm);
		\draw (0,0) arc[
				start angle=-90,
				end angle=90,
				x radius=0.25cm,
				y radius=1.25cm
				];
	    \draw[dashed] (0,2.5) arc[
				start angle=90,
				end angle=270,
				x radius=0.25cm,
				y radius=1.25cm
				];
		\draw[thick] (.25,1) -- (-5.25,1);
        \node[above] at (0,2.5) {$\scriptstyle{u=\infty}$};
        \node[above] at (-5.5,2.5) {$\scriptstyle{u=0}$};
    \end{tikzpicture}
    \caption{$U$, fiber}
    \end{subfigure}
    \begin{subfigure}[b]{0.24\textwidth}
    \centering
    \begin{tikzpicture}
        \draw[thick] (0,0) -- (0,-2);
        \node[red] at (0,0) {$*$};
    \end{tikzpicture}
    \caption{$J$, base}
    \end{subfigure}
    \begin{subfigure}[b]{0.24\textwidth}
    \centering
    \begin{tikzpicture}[scale=.5]
		\draw (-5.5,2.5) -- (0,2.5);
		\draw (-5.5,0) -- (0,0);
		\draw (-5.5,1.25) ellipse (0.25 cm and 1.25 cm);
		\draw (0,0) arc[
				start angle=-90,
				end angle=90,
				x radius=0.25cm,
				y radius=1.25cm
				];
	    \draw[dashed] (0,2.5) arc[
				start angle=90,
				end angle=270,
				x radius=0.25cm,
				y radius=1.25cm
				];
        \node[above] at (0,2.5) {$\scriptstyle{u=\infty}$};
        \node[above] at (-5.5,2.5) {$\scriptstyle{u=0}$};
        \draw[thick, in=-90, out=180] (.25,1.25) to (-3,2.5);
        \draw[thick, in=90, out=180] (.25,1) to (-3,0);
        \draw[thick, dashed] (-3,2.5) arc[
				start angle=90,
				end angle=270,
				x radius=0.25cm,
				y radius=1.25cm
				];
    \end{tikzpicture}
    \caption{$J$, fiber}
    \end{subfigure}
    \caption{$U$ and $J$}
    \label{fig:stab}
\end{figure}

\begin{figure}
    \centering
    \begin{subfigure}[b]{0.3\textwidth}
    \centering
    \begin{tikzpicture}
        \draw[thick, name path=l1] (0,0) -- (0,-2);
        \draw[thick, cyan, name path=l2] plot [smooth, tension=0.6] coordinates {(.5,-2) (-.5,0) (0,.5) (.5,0) (1.5,-2)};
        \draw[teal, thick, name path=l3] (0,-.5) -- (2,-2);
        \draw[teal, thick] [domain=-3.14/2:-15,variable=\t,smooth,samples=75]
        plot ({\t r}: {-3.14/(4*\t)});
        \node[red] at (0,0) {$*$};
        \fill[name intersections={of=l1 and l2, by=p}] (intersection-1) circle (2pt) node[left] {$p$};
        \fill[name intersections={of=l2 and l3, by=q}] (intersection-1) circle (2pt) node[right] {$q$};
        \fill[name intersections={of=l1 and l3, by=r}] (intersection-1) circle (2pt) node[below right] {$r$};
        \begin{scope}[on background layer]
            \clip (p) -- (r) -- (q) -- (2,2) -- (-2,2) -- (-2,-2) -- cycle;
            \fill[red!20] plot [smooth, tension=0.6] coordinates {(.5,-2) (-.5,0) (0,.5) (.5,0) (1.5,-2)};
        \end{scope}
    \end{tikzpicture}
    \caption{Base}
    \label{fig:stab map composition base}
    \end{subfigure}
    \begin{subfigure}[b]{0.3\textwidth}
    \centering
    \begin{tikzpicture}[scale=.7]
		\draw (-5.5,2.5) -- (0,2.5);
		\draw (-5.5,0) -- (0,0);
		\draw (-5.5,1.25) ellipse (0.25 cm and 1.25 cm);
		\draw (0,0) arc[
				start angle=-90,
				end angle=90,
				x radius=0.25cm,
				y radius=1.25cm
				];
	    \draw[dashed] (0,2.5) arc[
				start angle=90,
				end angle=270,
				x radius=0.25cm,
				y radius=1.25cm
				];
        \node[above] at (0,2.5) {$\scriptstyle{u=\infty}$};
        \node[above] at (-5.5,2.5) {$\scriptstyle{u=0}$};
        \draw[thick, in=-90, out=180, name path=l1] (.25,.75) to (-2,2.5);
        \draw[thick, in=90, out=180] (.25,.5) to (-2,0);
        \draw[thick, dashed] (-2,2.5) arc[
				start angle=90,
				end angle=270,
				x radius=0.25cm,
				y radius=1.25cm
				];
        \draw[thick, cyan, name path=l2] (-5.25,1) -- (.25,1);
        \draw[thick, teal, in=-90, out=180, name path=l3a] (.25,1.75) to (-3,2.5);
        \draw[thick, teal, in=90, out=180, name path=l3b] (.25,1.5) to (-3,0);
        \draw[thick, teal, dashed] (-3,2.5) arc[
				start angle=90,
				end angle=270,
				x radius=0.25cm,
				y radius=1.25cm
				];
        \fill[name intersections={of=l1 and l3a, by=r}] (intersection-1) circle (3pt) node[above right] {$r$};
        \fill[name intersections={of=l1 and l2, by=p}] (intersection-1) circle (3pt) node[above right] {$p$};
        \fill[name intersections={of=l3b and l2, by=q}] (intersection-1) circle (3pt) node[below] {$q$};
        \begin{scope}[on background layer]
            \clip (-2,2.5) to[in=180, out=-90] (.25,.75) to (q) to (-3,0) to (-5.5,0) to (-5.5,2.5);
            \fill[red!20] (0,0) to (-5.5,0) arc (-90:90:.25cm and 1.25cm) to (-3,2.5) to[out=-90, in=180] (.25,1.75) to (.25,1.5) to[out=180,in=90] (-3,0) -- cycle;
            \fill[red!20] (q) -- (p) -- (-1,1.5) -- (-2,1.5) -- cycle;
        \end{scope}
    \end{tikzpicture}
    \caption{Fiber}
    \label{fig:stab map composition fiber}
    \end{subfigure}
    \begin{subfigure}[b]{0.3\textwidth}
    \centering
    \begin{tikzpicture}
        \draw (0:1.5) \foreach \x in {60,120,...,360} {  -- (\x:1.5) };
        \node at (30:1.75) {$(pr)$};
        \node at (90:1.75) {$(r)$};
        \node at (150:1.75) {$(rq)$};
        \node at (210:1.75) {$(q)$};
        \node at (270:1.75) {$(qp)$};
        \node at (330:1.75) {$(p)$};
        \draw[orange] (90:1.3) -- (270:1.3) node[pos=.6, right] {$\scriptstyle{\mathcal{M}_{S \to \C_u^*}}$};
        \draw[magenta, in=-30, out=210] (30:1.3) to (150:1.3) node[below] {$\scriptstyle{\mathcal{M}_{S \to \C_y^*}}$};
    \end{tikzpicture}
    \caption{$\mathcal{M}_{S \to \C_y^*}$ and $\mathcal{M}_{S \to \C_u^*}$}
    \label{fig:stab map moduli map}
    \end{subfigure}
    \caption{The disk for $p \cdot q = r$}
    \label{fig:stab map composition}
\end{figure}
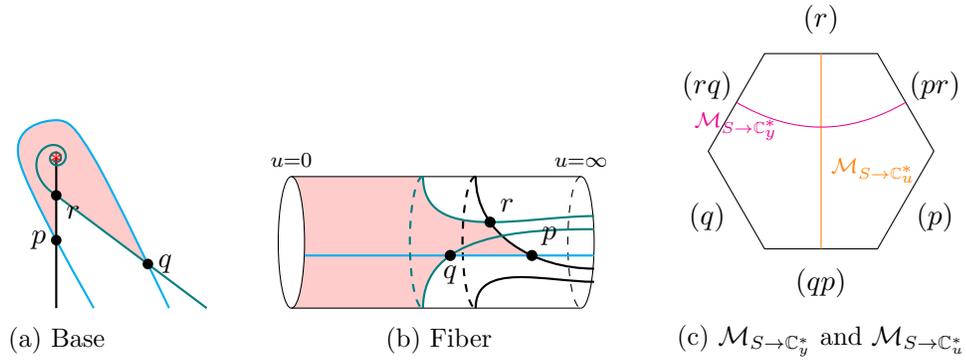

\begin{figure}
    \centering
    \begin{subfigure}[b]{0.3\textwidth}
    \centering
    \begin{tikzpicture}
        \draw[thick, cyan, name path=l2] (0,0) -- (1.5,-2);
        \draw[thick, name path=l1] plot [smooth, tension=0.6] coordinates {(-.5,-2) (-.5,0) (0,.5) (.5,0) (.5,-2)};
        \draw[teal, thick, teal, name path=l3] plot [smooth, tension=0.6] coordinates {(2,-2) (-.25,-.25)(.25,.25) (3,-2)};
        \node[red] at (0,0) {$*$};
        \fill[name intersections={of=l1 and l2, by=q}] (intersection-1) circle (2pt) node[right] {$q$};
        \fill[name intersections={of=l2 and l3, by=p}] (intersection-1) circle (2pt) node[below left] {$p$};
        \fill[name intersections={of=l1 and l3, by=s}] (intersection-1) circle (2pt) node[above right] {$s$};
        \begin{scope}[on background layer]
            \clip (s) to[out=-85, in=90] (q) -- (p) -- (-.5,-1) -- (-.5,1) -- cycle;
            \fill[red!20] plot [smooth, tension=0.6] coordinates {(2,-2) (-.25,-.25)(.25,.25) (3,-2)};
        \end{scope}
    \end{tikzpicture}
    \caption{Base}
    \end{subfigure}
    \begin{subfigure}[b]{0.3\textwidth}
    \centering
    \begin{tikzpicture}[scale=.7]
		\draw (-5.5,2.5) -- (0,2.5);
		\draw (-5.5,0) -- (0,0);
		\draw (-5.5,1.25) ellipse (0.25 cm and 1.25 cm);
		\draw (0,0) arc[
				start angle=-90,
				end angle=90,
				x radius=0.25cm,
				y radius=1.25cm
				];
	    \draw[dashed] (0,2.5) arc[
				start angle=90,
				end angle=270,
				x radius=0.25cm,
				y radius=1.25cm
				];
        \node[above] at (0,2.5) {$\scriptstyle{u=\infty}$};
        \node[above] at (-5.5,2.5) {$\scriptstyle{u=0}$};
        \draw[thick, cyan, in=-90, out=180, name path=l2a] (.25,1.5) to (-2,2.5);
        \draw[thick, cyan, in=90, out=180, name path=l2b] (.25,1.25) to (-2,0);
        \draw[thick,cyan, dashed] (-2,2.5) arc[
				start angle=90,
				end angle=270,
				x radius=0.25cm,
				y radius=1.25cm
				];
        \draw[thick, name path=l1] (-5.25,.75) -- (.25,.75);
        \draw[thick, teal, name path=l3, in=60, out=210] (0.15,2.25) to (-4,0);
        \draw[thick, teal, dashed, in=-60, out=150] (-4,0) to (-5.55,2.25);
        \fill[name intersections={of=l1 and l3, by=s}] (intersection-1) circle (3pt) node[above] {$s$};
        \fill[name intersections={of=l1 and l2b, by=q}] (intersection-1) circle (3pt) node[below] {$q$};
        \fill[name intersections={of=l3 and l2a, by=p}] (intersection-1) circle (3pt) node[above] {$p$};
        \begin{scope}[on background layer]
            \clip (0,2.5) to (0.15,2.25) to[in=35, out=210] (s) to (q) to (-2,0) to (-5.5,0) to (-5.5,2.5);
            \fill[red!20] (0,0) to (-5.5,0) arc (-90:90:.25cm and 1.25cm) to (-2,2.5) to[out=-90, in=180] (.25,1.5) to (.25,1.25) to[out=180,in=90] (-2,0) -- cycle;
        \end{scope}
    \end{tikzpicture}
    \caption{Fiber}
    \end{subfigure}
    \begin{subfigure}[b]{0.3\textwidth}
    \centering
    \begin{tikzpicture}
        \draw (0:1.5) \foreach \x in {60,120,...,360} {  -- (\x:1.5) };
        \node at (30:1.75) {$(ps)$};
        \node at (90:1.75) {$(s)$};
        \node at (150:1.75) {$(sq)$};
        \node at (210:1.75) {$(q)$};
        \node at (270:1.75) {$(qp)$};
        \node at (330:1.75) {$(p)$};
        \draw[magenta] (90:1.3) -- (270:1.3) node[pos=.6, right] {$\scriptstyle{\mathcal{M}_{S \to \C_y^*}}$};
        \draw[orange, in=-30, out=210] (30:1.3) to (150:1.3) node[below] {$\scriptstyle{\mathcal{M}_{S \to \C_u^*}}$};
    \end{tikzpicture}
    \caption{$\mathcal{M}_{S \to \C_y^*}$ and $\mathcal{M}_{S \to \C_u^*}$}
    \end{subfigure}
    \caption{The disk for $q \cdot p = s$}
    \label{fig:stab map other composition}
\end{figure}
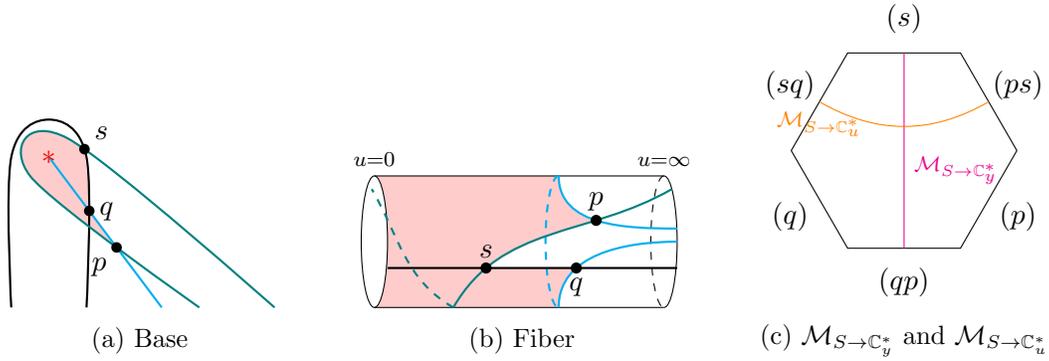

\begin{proposition}
    $I \cong \mathbb{A}(S)$
\end{proposition}

\begin{proof}
    The resolution of the simple module $S_\theta$ given in \eqref{eqn:simple resolution} is the cone over a map in homological degree zero from $\upsilon_-$ to $\upsilon_+$. It is the unique map in $H^0\Hom(\upsilon_-,\upsilon_+)$ up to scalar multiple that has $C_1$ degree zero. The only two maps in $Hom(J_-,J_+)$ in $C_1$ degree zero are $s t_1$ and $s t_2$ shown in Figure \ref{fig:gluing stabs}. Of these points, only $s t_1$ has homological degree zero. The result of surgery at this point is $I$.
\end{proof}

\begin{corollary}
    $E \cong I$. 
\end{corollary}

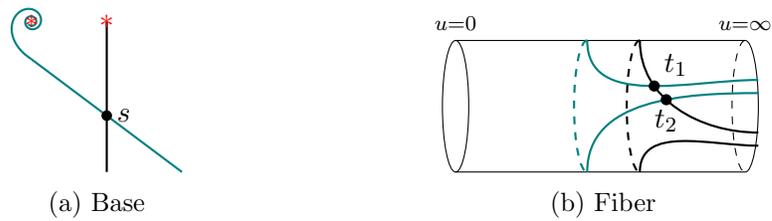
\begin{figure}
    \centering
    \begin{subfigure}[b]{0.4\textwidth}
    \centering
    \begin{tikzpicture}
        \draw[thick, name path=l1] (1,0) -- (1,-2);
        \draw[teal, thick, name path=l3] (0,-.5) -- (2,-2);
        \draw[teal, thick] [domain=-3.14/2:-15,variable=\t,smooth,samples=75]
        plot ({\t r}: {-3.14/(4*\t)});
        \node[red] at (0,0) {$*$};
        \node[red] at (1,0) {$*$};
        \fill[name intersections={of=l1 and l3}] (intersection-1) circle (2pt) node[right] {$s$};
    \end{tikzpicture}
    \caption{Base}
    \end{subfigure}
    \begin{subfigure}[b]{0.4\textwidth}
    \centering
    \begin{tikzpicture}[scale=.7]
		\draw (-5.5,2.5) -- (0,2.5);
		\draw (-5.5,0) -- (0,0);
		\draw (-5.5,1.25) ellipse (0.25 cm and 1.25 cm);
		\draw (0,0) arc[
				start angle=-90,
				end angle=90,
				x radius=0.25cm,
				y radius=1.25cm
				];
	    \draw[dashed] (0,2.5) arc[
				start angle=90,
				end angle=270,
				x radius=0.25cm,
				y radius=1.25cm
				];
        \node[above] at (0,2.5) {$\scriptstyle{u=\infty}$};
        \node[above] at (-5.5,2.5) {$\scriptstyle{u=0}$};
        \draw[thick, in=-90, out=180, name path=l1] (.25,.75) to (-2,2.5);
        \draw[thick, in=90, out=180] (.25,.5) to (-2,0);
        \draw[thick, dashed] (-2,2.5) arc[
				start angle=90,
				end angle=270,
				x radius=0.25cm,
				y radius=1.25cm
				];
        \draw[thick, teal, in=-90, out=180, name path=l3a] (.25,1.75) to (-3,2.5);
        \draw[thick, teal, in=90, out=180, name path=l3b] (.25,1.5) to (-3,0);
        \draw[thick, teal, dashed] (-3,2.5) arc[
				start angle=90,
				end angle=270,
				x radius=0.25cm,
				y radius=1.25cm
				];
        \fill[name intersections={of=l1 and l3a}] (intersection-1) circle (3pt) node[above right] {$t_1$};
        \fill[name intersections={of=l1 and l3b}] (intersection-1) circle (3pt) node[below] {$t_2$};
    \end{tikzpicture}
    \caption{Fiber}
    \end{subfigure}
    \caption{The intersection points in $Hom(J_-,J_+)$}
    \label{fig:gluing stabs}
\end{figure}

\bibliographystyle{plain}
\bibliography{ref}

\end{document}